\documentclass[11pt]{article}
\usepackage{authblk}
\usepackage[toc,page]{appendix}
\usepackage[top=2cm, bottom=2cm, left=2cm, right=2cm]{geometry}

\usepackage{color}
\usepackage{helvet}         
\usepackage{courier}        
\usepackage{type1cm}        

\usepackage{framed} 
\usepackage{tikz}
\usepackage{makeidx}         
\usepackage{graphicx}        
\usepackage{multicol}        
\usepackage[bottom]{footmisc}

\usepackage{amsmath}
\usepackage{amssymb}
\usepackage{bbold}
\usepackage{amsthm}
\usepackage{subcaption}
\usepackage{sidecap}
\usepackage{floatrow}
\usepackage{pdflscape}
\usepackage{comment}
\usepackage[font=small]{caption}
\usepackage{enumitem}
\usepackage{esint}

\usepackage[hidelinks]{hyperref}

\usepackage{scalerel}
\usepackage{shuffle}
\usepackage{mathrsfs}
\newtheorem{theorem}{Theorem}

\newtheorem{lemma}[theorem]{Lemma}

\newtheorem{proposition}[theorem]{Proposition}

\theoremstyle{definition}
\newtheorem{definition}[theorem]{Definition}

\theoremstyle{remark}

\newtheorem{remark}[theorem]{\bf Remark}

\numberwithin{theorem}{section}
\numberwithin{figure}{section}
\numberwithin{equation}{section}

\begin{document}
\title{Gaussian free field in annulus: \\BPZ equations and crossing probabilities for level lines}
\bigskip{}
\author[1]{Chongzhi Huang\thanks{huangchzh2001prob@gmail.com}}
\author[2]{Mingchang Liu\thanks{liumc\_prob@163.com}}
\author[1]{Hao Wu\thanks{hao.wu.proba@gmail.com.}}
\affil[1]{Tsinghua University, China}
\affil[2]{Capital Normal University, China}
\date{}

%
%


\global\long\def\qnum#1{\left[#1\right]_q }
\global\long\def\qfact#1{\left[#1\right]_q! }
\global\long\def\qbin#1#2{\left[\begin{array}{c}
	#1\\
	#2 
	\end{array}\right]_q}

\global\long\def\defpatt{\shuffle}
\global\long\def\rainbow{\includegraphics[scale=0.15]{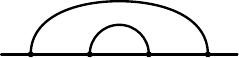}}
\global\long\def\rainbowBig{\includegraphics[scale=0.3]{figures/link-2}}

\newcommand{\nradpartfn}[2]{\mathcal{Z}^{#2}_{#1\mathrm{\textnormal{-}rad}}}
\global\long\def\covmap{h}

\global\long\def\mslitdriv{\omega}
\global\long\def\nnofloops{\mathscr{L}}

\global\long\def\Selberg{S}
\global\long\def\Diff{\Theta}
\global\long\def\SinDiff{\Xi}

\global\long\def\contour{\mathscr{C}}


\global\long\def\U{\mathbb{U}}
\global\long\def\T{\mathbb{T}}
\global\long\def\HH{\mathbb{H}}
\global\long\def\R{\mathbb{R}}
\global\long\def\C{\mathbb{C}}
\global\long\def\N{\mathbb{N}}
\global\long\def\Z{\mathbb{Z}}
\global\long\def\E{\mathbb{E}}
\global\long\def\PP{\mathbb{P}}
\global\long\def\rate{\mathcal{J}}
\global\long\def\QQ{\mathbb{Q}}
\global\long\def\A{\mathbb{A}}
\global\long\def\S{\mathbb{S}}
\global\long\def\one{\mathbb{1}}

\newcommand{\PPspiral}[1]{\mathbb{P}^{\mu}_{#1\mathrm{\textnormal{-}rad}}}
\newcommand{\PPnospiral}[1]{\mathbb{P}^{0}_{#1\mathrm{\textnormal{-}rad}}}
\newcommand{\PPnospiralrho}{\mathbb{P}^{0; \bs{\rho}}}
\newcommand{\PPspiralrho}{\mathbb{P}^{\mu; \bs{\rho}}}

\global\long\def\CR{\mathrm{CR}}
\global\long\def\ST{\mathrm{ST}}
\global\long\def\SF{\mathrm{SF}}
\global\long\def\cov{\mathrm{cov}}
\global\long\def\dist{\mathrm{dist}}
\global\long\def\SLE{\mathrm{SLE}}
\global\long\def\hSLE{\mathrm{hSLE}}
\global\long\def\CLE{\mathrm{CLE}}
\global\long\def\GFF{\mathrm{GFF}}
\global\long\def\inte{\mathrm{int}}
\global\long\def\ext{\mathrm{ext}}
\global\long\def\inrad{\mathrm{inrad}}
\global\long\def\outrad{\mathrm{outrad}}
\global\long\def\dimH{\mathrm{dim}}
\global\long\def\capa{\mathrm{cap}}
\global\long\def\diam{\mathrm{diam}}
\global\long\def\sign{\mathrm{sgn}}
\global\long\def\cat{\mathrm{Cat}}
\global\long\def\cst{\mathrm{C}}
\global\long\def\ck{\mathrm{C}_{\kappa}}
\global\long\def\free{\mathrm{free}}
\global\long\def\hF{{}_2\mathrm{F}_1}
\global\long\def\simple{\mathrm{simple}}
\global\long\def\even{\mathrm{even}}
\global\long\def\odd{\mathrm{odd}}
\global\long\def\st{\mathrm{ST}}
\global\long\def\usf{\mathrm{USF}}
\global\long\def\Leb{\mathrm{Leb}}
\global\long\def\LP{\mathrm{LP}}
\global\long\def\I{\mathrm{I}}
\global\long\def\II{\mathrm{II}}
\global\long\def\hcap{\mathrm{hcap}}
\global\long\def\Poisson{\mathrm{P}}
\global\long\def\cross{\mathrm{cross}}
\global\long\def\wind{\mathrm{wind}}
\global\long\def\Green{\mathrm{G}}
\global\long\def\Mod{\mathrm{mod}}

\global\long\def\LA{\mathcal{A}}
\global\long\def\LB{\mathcal{B}}
\global\long\def\LC{\mathcal{C}}
\global\long\def\LD{\mathcal{D}}
\global\long\def\LF{\mathcal{F}}
\global\long\def\LK{\mathcal{K}}
\global\long\def\LE{\mathcal{E}}
\global\long\def\LG{\mathcal{G}}
\global\long\def\LGmu{\mathcal{G}_{\mu}}
\global\long\def\LI{\mathcal{I}}
\global\long\def\LJ{\mathcal{J}}
\global\long\def\LL{\mathcal{L}}
\global\long\def\LM{\mathcal{M}}
\global\long\def\LN{\mathcal{N}}
\global\long\def\OO{\mathcal{O}}
\global\long\def\LQ{\mathcal{Q}}
\global\long\def\LR{\mathcal{R}}
\global\long\def\LT{\mathcal{T}}
\global\long\def\LS{\mathcal{S}}
\global\long\def\LU{\mathcal{U}}
\global\long\def\LV{\mathcal{V}}
\global\long\def\LW{\mathcal{W}}
\global\long\def\LX{\mathcal{X}}
\global\long\def\LY{\mathcal{Y}}
\global\long\def\PartF{\mathcal{Z}}
\global\long\def\LH{\mathcal{H}}
\global\long\def\LJ{\mathcal{J}}

\global\long\def\blm{m}

\global\long\def\LZ{\mathcal{Z}}
\newcommand{\LZann}[1]{\mathcal{Z}_{#1\mathrm{\textnormal{-}ann}}}
\newcommand{\varphiann}[1]{\varphi_{#1\mathrm{\textnormal{-}ann}}}
\newcommand{\LZcro}[1]{\mathcal{Z}_{#1\mathrm{\textnormal{-}cro}}}
\newcommand{\LFcro}[1]{\mathcal{F}_{#1\mathrm{\textnormal{-}cro}}}
\newcommand{\varphicro}[1]{\varphi_{#1\mathrm{\textnormal{-}cro}}}
\newcommand{\Pcro}[1]{\mathsf{P}_{#1\mathrm{\textnormal{-}cro}}}
\newcommand{\Ecro}[1]{\mathsf{E}_{#1\mathrm{\textnormal{-}cro}}}
\newcommand{\Pind}[1]{\mathsf{P}_{#1\mathrm{\textnormal{-}ind}}}
\newcommand{\Eind}[1]{\mathsf{E}_{#1\mathrm{\textnormal{-}ind}}}

\newcommand{\LZtwo}{\mathcal{Z}_{\includegraphics[scale=0.15]{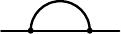}}^{(\kappa)}}
\newcommand{\Etwo}{\mathbb{E}_{\includegraphics[scale=0.15]{figures/link-0}}}
\newcommand{\Ptwo}{\PP_{\includegraphics[scale=0.15]{figures/link-0}}}
\newcommand{\PPtworho}{\PP_{\includegraphics[scale=0.15]{figures/link-0}}^{(\kappa; \boldsymbol{\rho})}}
\newcommand{\LEtwo}{\mathcal{E}_{\includegraphics[scale=0.15]{figures/link-0}}}
\newcommand{\Etworho}{\mathbb{E}_{\includegraphics[scale=0.15]{figures/link-0}}^{(\kappa; \boldsymbol{\rho})}}
\newcommand{\LItwo}{\mathcal{I}_{\includegraphics[scale=0.15]{figures/link-0}}}
\newcommand{\LUtwo}{\mathcal{U}_{\includegraphics[scale=0.15]{figures/link-0}}}
\newcommand{\LZtwor}{\LZtwo^{(\mathfrak{r})}}
\newcommand{\LHtwo}{\mathcal{H}_{\includegraphics[scale=0.15]{figures/link-0}}}
\newcommand{\LFtwo}{\mathcal{F}_{\includegraphics[scale=0.15]{figures/link-0}}}
\newcommand{\chambertwo}{\chamber_{\includegraphics[scale=0.15]{figures/link-0}}}

\newcommand{\PPkapparho}{\PP_{\includegraphics[scale=0.15]{figures/link-0}}^{(\kappa; \bs{\rho})}}

\newcommand{\LZfourb}{\mathcal{Z}_{\includegraphics[scale=0.15]{figures/link-2}}}
\newcommand{\LHfourb}{\mathcal{H}_{\includegraphics[scale=0.15]{figures/link-2}}}
\newcommand{\LFfourb}{\mathcal{F}_{\includegraphics[scale=0.15]{figures/link-2}}}
\newcommand{\LFfourbRenorm}{\widehat{\mathcal{F}}_{\includegraphics[scale=0.15]{figures/link-2}}}

\newcommand{\QQrainbow}[1]{\mathbb{Q}_{\rainbow_{#1}}}
\newcommand{\LZrainbow}[1]{\mathcal{Z}_{\rainbow_{#1}}}
\newcommand{\Frainbow}[1]{\mathscr{F}_{\rainbow_{#1}}}
\newcommand{\chamberrainbow}[1]{\chamber_{\rainbow_{#1}}}

\newcommand{\chamberradial}[1]{\chamber_{#1\mathrm{\textnormal{-}rad}}}
\newcommand{\PPradial}[1]{\mathbb{P}_{#1\mathrm{\textnormal{-}rad}}}
\newcommand{\Eradial}[1]{\mathbb{E}_{#1\mathrm{\textnormal{-}rad}}}
\newcommand{\rateradial}[1]{\mathcal{J}_{#1\mathrm{\textnormal{-}rad}}}
\newcommand{\LIradial}[1]{\mathcal{I}_{#1\mathrm{\textnormal{-}rad}}}
\newcommand{\LVradial}[1]{\mathcal{V}_{#1\mathrm{\textnormal{-}rad}}}
\newcommand{\tilderateradial}[1]{\tilde{\mathcal{J}}_{#1\mathrm{\textnormal{-}rad}}}
\newcommand{\hatrateradial}[1]{\hat{\mathcal{J}}_{#1\mathrm{\textnormal{-}rad}}}
\newcommand{\LZradial}[1]{\mathcal{Z}_{#1\mathrm{\textnormal{-}rad}}}
\newcommand{\LFradial}[1]{\mathcal{F}_{#1\mathrm{\textnormal{-}rad}}}
\newcommand{\varphiradial}[1]{\varphi_{#1\mathrm{\textnormal{-}rad}}}

\newcommand{\coulombGasHRenorm}{\widehat{\coulombGasH}}

\global\long\def\coulomb{\LH}
\global\long\def\auxcoulomb{\hat{\coulomb}}
\global\long\def\coulombGas{\LF}
\global\long\def\coulombnew{\LK}
\global\long\def\coulombLine{\LG}
\global\long\def\kfunc{p}

\global\long\def\eps{\epsilon}
\global\long\def\ov{\overline}
\global\long\def\QQrp{\QQ_{\alpha; \bs{s}}^{(p)}}

\global\long\def\bn{\mathbf{n}}
\global\long\def\MR{MR}
\global\long\def\cond{\,|\,}
\global\long\def\bigcond{\,\big|\,}
\global\long\def\Bigcond{\;\Big|\;}
\global\long\def\la{\langle}
\global\long\def\ra{\rangle}
\global\long\def\tree{\Upsilon}
\global\long\def\prob{\mathbb{P}}
\global\long\def\hm{\mathrm{Hm}}
%

\global\long\def\Im{\operatorname{Im}}
\global\long\def\Re{\operatorname{Re}}

\global\long\def\ud{\mathrm{d}}
\global\long\def\pder#1{\frac{\partial}{\partial#1}}
\global\long\def\pdder#1{\frac{\partial^{2}}{\partial#1^{2}}}
\global\long\def\pddder#1{\frac{\partial^{3}}{\partial#1^{3}}}
\global\long\def\der#1{\frac{\ud}{\ud#1}}

\global\long\def\bZnn{\mathbb{Z}_{\geq 0}}
\global\long\def\bZpos{\mathbb{Z}_{> 0}}
\global\long\def\bZneg{\mathbb{Z}_{< 0}}

\global\long\def\Vfunc{\LG}
\global\long\def\gfunc{g^{(\rr)}}
\global\long\def\hfunc{h^{(\rr)}}

\global\long\def\SimplexInt{\rho}
\global\long\def\CubeInt{\widetilde{\rho}}

\global\long\def\ii{\mathrm{i}}
\global\long\def\ee{\mathrm{e}}
\global\long\def\rr{\mathfrak{r}}
\global\long\def\chamber{\mathfrak{X}}
\global\long\def\Wchamber{\mathfrak{W}}

\global\long\def\SimplexIntKappa8{\SimplexInt}

\global\long\def\nested{\boldsymbol{\underline{\Cap}}}
\global\long\def\unnested{\boldsymbol{\underline{\cap\cap}}}
\global\long\def\unnested{\boldsymbol{\underline{\cap\cap}}}

\global\long\def\acycle{\vartheta}
\global\long\def\bcycle{\tilde{\acycle}}

\global\long\def\metric{\mathrm{dist}}

\global\long\def\adj#1{\mathrm{adj}(#1)}

\global\long\def\bs{\boldsymbol}

\global\long\def\edge#1#2{\langle #1,#2 \rangle}
\global\long\def\graph{G}

\newcommand{\conn}{\varsigma}
\newcommand{\realacycle}{\smash{\mathring{\acycle}}}
\newcommand{\realpt}{\smash{\mathring{x}}}
\newcommand{\corrind}{\LC}
\newcommand{\bssymb}{\pi}
\newcommand{\coeff}{p}
\newcommand{\MainConst}{C}

\global\long\def\removeLink{/}

\global\long\def\domainofdef{\mathfrak{U}}
\global\long\def\Test_space{C_c^\infty}
\global\long\def\Distr_space{(\Test_space)^*}

\global\long\def\bs{\boldsymbol}
\global\long\def\cst{\mathrm{C}}

\newcommand{\red}{\textcolor{red}}
\newcommand{\blue}{\textcolor{blue}}
\newcommand{\green}{\textcolor{green}}
\newcommand{\magenta}{\textcolor{magenta}}
\newcommand{\cyan}{\textcolor{cyan}}

\newcommand{\coulombGasH}{\mathcal{H}}
\newcommand{\secondbeta}{\intloop}

\newcommand{\cev}[1]{\reflectbox{\ensuremath{\vec{\reflectbox{\ensuremath{#1}}}}}}

\global\long\def\anticonf{\zeta}
\global\long\def\intloop{\varrho}
\global\long\def\Gloop{\smash{\mathring{\intloop}}}

\global\long\def\SLEmeasure{\mathrm{P}}
\global\long\def\SLEmeasureEx{\mathrm{E}}

\global\long\def\fugacity{\nu}
\global\long\def\meanderMat{\mathcal{M}}
\global\long\def\LM{\mathcal{M}}
\global\long\def\meanderMatrix{\meanderMat_{\fugacity}}
\global\long\def\meanderMatrixPrime{\meanderMat_{\fugacity(\kappa')}}
\global\long\def\meanderRenorm{\widehat{\mathcal{M}}}

\global\long\def\PartFRenorm{\widehat{\PartF}}
\global\long\def\coulombGasRenorm{\widehat{\coulombGas}}

\global\long\def\hexa{\scalebox{1.3}{\hexagon}}

\global\long\def\np{p}

\global\long\def\FKdual{\mathcal{L}}

\global\long\def\fixedindex{\flat}

\global\long\def\Dirichletradial{\mathcal{I}_{\mathrm{rad}}}
\global\long\def\LVone{\mathcal{V}_{1\mathrm{\textnormal{-}rad}}}
\maketitle
\vspace{-1cm}
\begin{center}
\begin{minipage}{0.95\textwidth}
\abstract{
We consider level lines of Gaussian free field (GFF) in annulus with alternating boundary conditions. 
We calculate the probability that all level lines cross the annulus. Such probability is given by the ratio between two partition functions. 
These two partition functions are constructed via Dub\'edat's regularized Dirichlet energy. 
We show that these partition functions are solutions to annulus Belavin-Polyakov-Zamolodchikov (BPZ) equations. 
In the annulus setup, the number of variables exceeds the number of BPZ equations, so the BPZ system alone does not determine the partition functions uniquely. By establishing sufficiently good control of the two partition functions constructed above, we are nevertheless able to derive the crossing probability.
}
\medbreak
\end{minipage}
\end{center}

\tableofcontents

\section{Introduction}
In the context of the multiple SLE, the Belavin-Polyakov-Zamolodchikov (BPZ) equations provide the analytical framework to ensure that the partition function $\LZ$ yields a local martingale under the Loewner flow. This requirement is the manifestation of the commutation relations for multi-path SLE, which assert that the joint law of a family of curves $(\gamma^1, \dots, \gamma^n)$ is consistent regardless of their sampling order. As established by Dub\'{e}dat~\cite{DubedatCommutationSLE}, these relations are fundamentally rooted in the combination of conformal invariance and the domain Markov property. 
The BPZ equations in the chordal setting are widely investigated~\cite{BauerBernardKytolaMultipleSLE, GrahamSLE, FloresKlebanPDE1, KytolaPeltolaPurePartitionFunctions, 
KytolaPeltolaConformalCovBoundaryCorrelation,zhang2025multiplechordalslekappaquantum} and their radial analogue also gets attention recently~\cite{KrusellWangWuCommutationRelation, zhang2025multipleradialslekappaquantum, HuangPeltolaWuMultiradialSLEResamplingBP}. 
For the annulus setup, the BPZ equations must incorporate a derivative with respect to the modulus, reflecting the geometric deformation of the doubly connected domain as the curves grow. 
In this article, we discuss BPZ equations in the annulus setup and relate them to the level lines of Gaussian free field (GFF).

\paragraph*{Annulus and its universal cover.} 
Fix $r>0$ and we define
\begin{align}\label{eqn::annulus_strip_def}
\A_r=\{z\in\C: \ee^{-r}<|z|<1\}, \qquad \S_r=\{z\in\C: 0<\Im{z}<r\}, \qquad q: z\mapsto \ee^{\ii z}. 
\end{align}
The annulus $\A_r$ is doubly connected. The infinite strip $\S_r=q^{-1}(\A_r)$ is the universal cover of $\A_r$ and it is simply connected.  
Fix $n\ge 1$, we define 
\[\LX_n=\{(\theta_1, \ldots, \theta_n): \theta_1<\cdots<\theta_n<\theta_1+2\pi\}.\]
For $\bs{\alpha}=(\alpha_1, \ldots, \alpha_n)\in\LX_n$, we write $\bs{x}=\ee^{\ii\bs{\alpha}}$ meaning that $\bs{x}=(x_1, \ldots, x_n)$ with $x_j=\ee^{\ii\alpha_j}$ for $1\le j\le n$. 
For $\bs{\beta}=(\beta_1, \ldots, \beta_n)\in\LX_n$, we write $\bs{y}=\ee^{\ii\bs{\beta}-r}$ meaning that $\bs{y}=(y_1, \ldots, y_n)$ with $y_j=\ee^{\ii\beta_j-r}$ for $1\le j\le n$. 
Note that $\bs{x}$ are marked points on the outer-boundary of $\A_r$ and $\bs{y}$ are marked points on the inner-boundary of $\A_r$. 
The preimages of $\bs{x}$ and of $\bs{y}$ are given by 
\[q^{-1}(\bs{x})=\cup_{\ell\in\Z}(\alpha_1+2\pi\ell, \ldots, \alpha_n+2\pi\ell), \qquad q^{-1}(\bs{y})=\cup_{\ell\in\Z}(\beta_1+2\pi\ell+\ii r, \ldots, \beta_n+2\pi\ell+\ii r).\]

\paragraph*{GFF in annulus.}
Fix an even number $n=2N$. For $\bs{\alpha}, \bs{\beta}\in\LX_n$, we write $\bs{x}=\ee^{\ii\bs{\alpha}}$ and $\bs{y}=\ee^{\ii\bs{\beta}-r}$.
We consider GFF in annulus $(\A_r; \bs{x}; \bs{y})$ with alternating boundary data:\footnote{with the convention that $x_{2N+j}=x_j$ and $y_{2N+j}=y_j$.} 
\begin{align}\label{eqn::boundarydata}
\pi \text{ on } \bigcup_{j=1}^N (x_{2j-1}x_{2j})\cup(y_{2j-1}y_{2j}), \qquad 0\text{ on }\bigcup_{j=1}^N (x_{2j}x_{2j+1})\cup(y_{2j}y_{2j+1}).
\end{align}
Let $\gamma^j$ be the level line of the field starting from $x_{j}$ for $1\le j\le n$ (see Theorem~\ref{thm::GFF_levellines_annulus} and details in Section~\ref{sec::GFF_levellines_annulus}) and denote $\bs{\gamma}=(\gamma^1, \ldots, \gamma^n)$. 
For each $\gamma^j$, it almost surely terminates in $\{x_{j+1}, x_{j+3}, \ldots, x_{j+2N-1}\}\cup \{y_j, y_{j+2}, \ldots, y_{j+2N-2}\}$. 
When $\gamma^j$ terminates in $\{y_j, y_{j+2}, \ldots, y_{j+2N-2}\}$, we say that it crosses the annulus. 
In this article, we focus on the case when all $\gamma^j$ cross the annulus and denote 
\begin{equation}\label{eqn::crossingevent_def}
\cross(\bs{\gamma})=\{\text{all }\gamma^1, \ldots, \gamma^n\text{ cross the annulus }\A_r\text{ and terminate in }\bs{y}\}.
\end{equation} 
On the event $\cross(\bs{\gamma})$, let us further check the location of the terminal points of $\bs{\gamma}$. As the level lines do not cross each other, there exists $m\in\{1, \ldots, N\}$ such that $\gamma^j$ terminates in $y_{2m+j}$ for all $1\le j\le n$, see Figure~\ref{fig::levellines_annulus_cover}. For each fixed $m$, the winding of the level lines can be different by a multiple of $2\pi$. 
On the event $\cross(\bs{\gamma})$, for $m\in\{1, \ldots, N\}$ and $\ell\in\Z$, we define\footnote{with the convention that $\beta_{2N+j}=\beta_j+2\pi$.} 
\begin{equation}\label{eqn::windingevent_def}
\{\wind(\bs{\gamma})=(m;\ell)\}=\{q^{-1}(\gamma^j)\text{ connects }\alpha_j \text{ to }\beta_{2m+j}+2\pi\ell+\ii r \text{ for all }1\le j\le n\}. 
\end{equation}

\begin{figure}[ht!]
\begin{subfigure}[t]{0.45\textwidth}
\begin{center}
\includegraphics[width=0.7\textwidth]{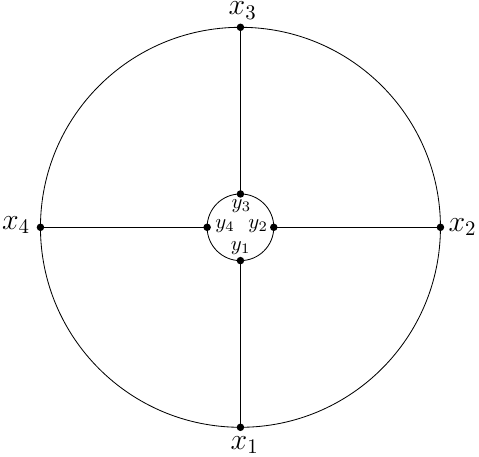}
\end{center}
\caption{$\gamma^j$ connects $x_j$ to $y_j$.}
\end{subfigure}
\begin{subfigure}[t]{0.45\textwidth}
\begin{center}
\includegraphics[width=0.7\textwidth]{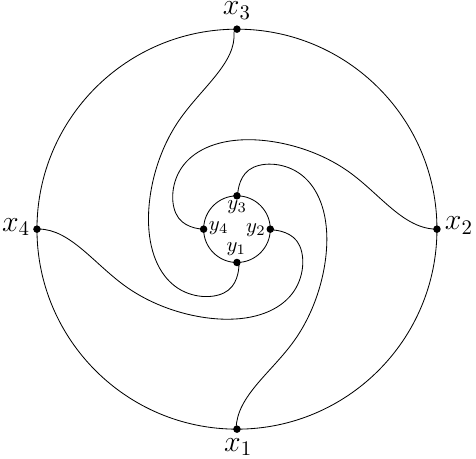}
\end{center}
\caption{$\gamma^j$ connects $x_j$ to $y_{j+2}$.} 
\end{subfigure}
\medbreak
\begin{subfigure}[t]{0.45\textwidth}
\begin{center}
\includegraphics[width=0.8\textwidth]{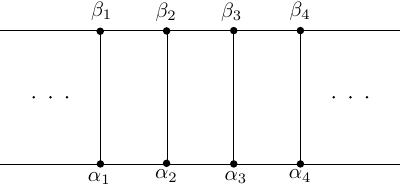}
\end{center}
\caption{$\wind(\bs{\gamma})=(0,0)$.}
\end{subfigure}
\begin{subfigure}[t]{0.45\textwidth}
\begin{center}
\includegraphics[width=0.8\textwidth]{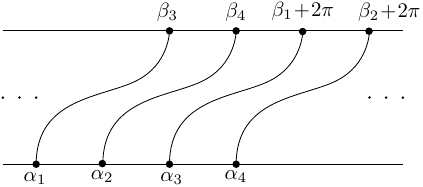}
\end{center}
\caption{$\wind(\bs{\gamma})=(1;0)$.}
\end{subfigure}

\caption{\label{fig::levellines_annulus_cover} 
Level lines crossing the annulus may have different winding.}
\end{figure}

We will derive the law of the level lines $\bs{\gamma}$ and calculate $\PP[\cross(\bs{\gamma}), \wind(\bs{\gamma})=(m;\ell)]$. 
We summarize our conclusion below.
\begin{itemize}
\item The law of the level lines and the probability of the crossing event are encoded by partition functions. 
We define two partition functions $\LZann{n}$ and $\LZcro{n}^{(\ell)}$ in Section~\ref{subsec::intro_pf_BPZ} using Jacobi theta functions. 
They are derived from the regularized Dirichlet energy, introduced by Dub\'edat~\cite{DubedatSLEFreeField}. 
We show that they satisfy annulus BPZ equations in Proposition~\ref{prop::annulusBPZ_kappa4}. Unlike the chordal case where the solution space is finite-dimensional, the annulus BPZ equations admit an infinite-dimensional space of solutions. Our explicit construction via regularized Dirichlet energy identifies the particular solutions relevant to the level lines.
\item We derive the law of the level lines in Theorem~\ref{thm::GFF_levellines_annulus}. 
Their law is encoded by the partition function $\LZann{n}$. 
The key idea in deriving the law of the level lines is to introduce an analytic lift of the harmonic observable. 
We calculate the probability $\PP[\cross(\bs{\gamma}), \wind(\bs{\gamma})=(m;\ell)]$ in Theorem~\ref{thm::crossingproba}. This probability is the ratio between $\LZcro{n}^{(\ell)}$ and $\LZann{n}$. We also derive the asymptotics of the partition functions, beyond the usual exponential decay in the modulus $r$ with the rate given by $\SLE$ arm exponents, our explicit formula further reveals a nontrivial polynomial correction $\sqrt r$ in the prefactor.
\end{itemize}

The analysis for level lines of GFF dates back to~\cite{DubedatSLEFreeField, SchrammSheffieldDiscreteGFF, SchrammSheffieldContinuumGFF}.
The properties of multiple level lines of GFF in simply connected domains are widely investigated, see e.g.~\cite{WangWuLevellinesGFFI,AruSepulvedaWernerBTLStwoDGFF}. In particular, explicit formulae for crossing probabilities of multiple level lines in polygons are derived in~\cite{PeltolaWuGlobalMultipleSLEs}. 
In the case of polygons, it is possible to derive all the crossing probabilities, because 
the number of variables is the same as the number of the corresponding chordal BPZ equations and 
we have complete understanding of the solution space to the chordal BPZ equations. 

GFF in multiply connected domains are studied in~\cite{HagendorfBauerBernardGFFDoubly,IzyurovKytolaHadamardSLEFreeField} and its level lines are also analyzed in~\cite{AruSepulvedaWernerBTLStwoDGFF,AruLupuSepulvedaFPSGFF}. 
Certain hitting probabilities of level lines in doubly connected domains are calculated in~\cite{HagendorfBauerBernardGFFDoubly}. 
In general multiply connected domains, it is usually hard to get explicit formula for crossing probabilities, because the number of variables is more than the number of the corresponding BPZ equations. Nevertheless, we find it possible to derive an explicit formula for the crossing probability $\PP[\cross(\bs{\gamma}), \wind(\bs{\gamma})=(m;\ell)]$ in Theorem~\ref{thm::crossingproba}.

\subsection{Partition functions and BPZ equations in annulus}
\label{subsec::intro_pf_BPZ}

\paragraph*{Jacobi theta functions.}
Jacobi theta functions are a family of special functions of two variables $(t, z)$ with $t\in \HH$ and $z\in\C$. This family has four functions, typically denoted by $\vartheta_1, \vartheta_2, \vartheta_3, \vartheta_4$. 
They are fundamental solutions to the heat equation with quasi-periodicity and polynomial growth.  
The Jacobi theta function appears as the one-loop partition function of a free boson compactified on a circle, which is one of the simplest conformal field theories. For the annulus modulus parameter $r > 0$ and a complex variable $z \in \mathbb{C}$, we define the rescaled Jacobi theta functions $\Theta_1$ and $\Theta_3$ by
\begin{align}\label{eqn::JacobiTheta_productexpansion}
\begin{split}
\Theta_1(r;z)=& 2 \ee^{-r/4} \sin\left(\frac{z}{2}\right) \prod_{m=1}^\infty \left(1 - \ee^{-2mr}\right) \left(1 - \ee^{-2mr}\ee^{\ii z}\right) \left(1 - \ee^{-2mr}\ee^{-\ii z}\right);\\
\Theta_3(r;z)=& \prod_{m=1}^\infty \left(1 - \ee^{-2mr}\right) \left(1 - \ee^{-(2m-1)r}\ee^{\ii z}\right) \left(1 - \ee^{-(2m-1)r}\ee^{-\ii z}\right). 
\end{split}
\end{align}
They are the building blocks of the partition functions for GFF in annulus. 

\paragraph*{Partition functions in annulus.}
Fix $n\ge 1$ and $\ell\in\Z$. For $r>0$ and $\bs{\alpha}=(\alpha_1, \ldots, \alpha_n)\in\LX_n$ and $\bs{\beta}=(\beta_1, \ldots, \beta_n)\in\LX_n$, define
\begin{align}\label{eqn::LZann_def}
	\LZann{n}(r; \bs{\alpha}, \bs{\beta})=&\exp\left(\frac{nr}{8}-\frac{1}{8r}\left(\sum_{j=1}^n(-1)^j(\alpha_j-\beta_j)\right)^2\right)\prod_{m=1}^{\infty}(1-\ee^{-2mr})^{\frac{3n}{2}}\\
	&\times \prod_{j=1}^n|\Theta_3(r; \beta_j-\alpha_j)|^{-\frac{1}{2}}\times\prod_{1\le i<k\le n}\frac{|\Theta_1(r; \alpha_k-\alpha_i)|^{\frac{1}{2}(-1)^{k-i}}|\Theta_1(r; \beta_k-\beta_i)|^{\frac{1}{2}(-1)^{k-i}}}{|\Theta_3(r; \beta_k-\alpha_i)|^{\frac{1}{2}(-1)^{k-i}}|\Theta_3(r; \beta_i-\alpha_k)|^{\frac{1}{2}(-1)^{k-i}}};\notag\\
	\label{eqn::LZcroell_def}
		\LZcro{n}^{(\ell)}(r; \bs{\alpha}, \bs{\beta})=&\exp\left(\frac{nr}{8}-\frac{1}{8r}\left(\sum_{j=1}^n(\alpha_j-\beta_j)-2n\pi\ell\right)^2\right)\times \prod_{m=1}^{\infty}(1-\ee^{-2mr})^{\frac{3n}{2}}\\
		&\times\prod_{j=1}^n|\Theta_3(r; \beta_j-\alpha_j)|^{-\frac{1}{2}}\times\prod_{1\le i<k\le n}\frac{|\Theta_1(r; \alpha_k-\alpha_i)|^{\frac{1}{2}} |\Theta_1(r; \beta_k-\beta_i)|^{\frac{1}{2}}}{|\Theta_3(r; \beta_k-\alpha_i)|^{\frac{1}{2}}|\Theta_3(r; \beta_i-\alpha_k)|^{\frac{1}{2}}}.\notag
\end{align}
where $\Theta_1, \Theta_3$ are rescaled Jacobi theta functions defined in~\eqref{eqn::JacobiTheta_productexpansion}. 
We also define
\begin{align}\label{eqn::LZcro_def}
\LZcro{n}(r; \bs{\alpha}, \bs{\beta})=\sum_{\ell\in\Z}\LZcro{n}^{(\ell)}(r; \bs{\alpha}, \bs{\beta}). 
\end{align}

\begin{proposition}[Annulus BPZ]
\label{prop::annulusBPZ_kappa4}
Recall that $\Theta_1, \Theta_3$ are rescaled Jacobi theta functions in~\eqref{eqn::JacobiTheta_productexpansion}. Define
\begin{align} \label{eqn::H1H3_def}
	H_1(r;z) := 2 \partial_z \log \Theta_1(r;z), \qquad
	H_3(r;z) := 2 \partial_z \log \Theta_3(r;z).
\end{align}
The partition functions $\LZann{n}$ and $\LZcro{n}^{(\ell)}$ in~\eqref{eqn::LZann_def} and~\eqref{eqn::LZcroell_def} satisfy the following system of annulus BPZ equations: for all $j\in\{1, \ldots, n\}$, 
\begin{align} \label{eqn::annulus_BPZ_kappa4_alpha}
	\begin{split}
		\frac{\partial_r \LZ}{\LZ}=2\frac{\partial_{\alpha_j}^2 \LZ}{\LZ} & +   \sum_{\ell\neq j} \left( H_1(r; \alpha_\ell-\alpha_j) \frac{\partial_{\alpha_\ell} \LZ}{\LZ} + \frac{1}{4} \partial_z H_1(r; \alpha_\ell-\alpha_j) \right) \\
		& +  \sum_{\ell=1}^n \left( H_3(r; \beta_\ell-\alpha_j) \frac{\partial_{\beta_\ell} \LZ}{\LZ} + \frac{1}{4} \partial_z H_3(r; \beta_\ell-\alpha_j) \right) + \frac{1}{2r} + \frac{3}{4} \LE(r) + \frac{n}{4}; 
	\end{split} \\
 \label{eqn::annulus_BPZ_kappa4_beta}
	\begin{split}
		\frac{\partial_r \LZ}{\LZ}= 2\frac{\partial_{\beta_j}^2 \LZ}{\LZ} &+  \sum_{\ell\neq j} \left( H_1(r; \beta_\ell-\beta_j) \frac{\partial_{\beta_\ell} \LZ}{\LZ} + \frac{1}{4} \partial_z H_1(r; \beta_\ell-\beta_j) \right) \\
		& +  \sum_{\ell=1}^n \left( H_3(r; \alpha_\ell-\beta_j) \frac{\partial_{\alpha_\ell} \LZ}{\LZ} + \frac{1}{4} \partial_z H_3(r; \alpha_\ell-\beta_j) \right) + \frac{1}{2r} + \frac{3}{4} \LE(r) + \frac{n}{4}; 
	\end{split}
\end{align}
where \begin{equation} \label{eqn::Def_LE}
	\LE(r) := -\frac{1}{6} + \sum_{k=1}^{\infty} \frac{1}{\sinh^2(kr)}.
\end{equation}
\end{proposition}

The function $\LE(r)$ is intrinsically connected to the theory of modular forms. Specifically, it is a scaling of the Eisenstein series $E_2$ of weight $2$: 
\begin{equation*}
	\LE(r) = -\frac{1}{6} E_2\left( \frac{\ii r}{\pi} \right).
\end{equation*}
The significance of $E_2$ in the context of annulus SLE and GFF lies in its status as a \textit{quasi-modular} form. Unlike modular forms of higher weight, $E_2$ does not transform purely under the modular inversion $\tau \mapsto -1/\tau$, but picks up an additive term: $E_2(-1/\tau) = \tau^2 E_2(\tau) + \frac{6\tau}{\pi \ii}$. 
In the language of Conformal Field Theory (CFT), this non-modular behavior is a manifestation of the conformal anomaly (or central charge). In our setting, $\LE(r)$ governs the ``drift" or the interaction between the two boundaries of the annulus $\A_r$. As the modulus $r$ tends to infinity, $\LE(r)$ approaches $-1/6$, which corresponds to the limit where the inner boundary of the annulus shrinks to a point, recovering the behavior of radial SLE in the unit disc. The appearance of $\LE(r)$ in the differential equation for the Brownian loop measure (Lemma~\ref{lem::mt_blm_annulus_diff}) is thus consistent with the fact that the partition function of the free boson on a Riemann surface is expressed in terms of these modular-type objects.

Annulus BPZ equations are also considered in the setup of conformal field theory~\cite{ByunKangTakCFTannulusSLE,AlbertsByunKangCFTannulusGFF}. 
The general version of~\eqref{eqn::annulus_BPZ_kappa4_alpha}-\eqref{eqn::annulus_BPZ_kappa4_beta} with $\kappa>0$ will be given in~\eqref{eqn::annulus_BPZ_alpha}-\eqref{eqn::annulus_BPZ_beta} in Section~\ref{sec::multiannulus_SLE}.  
These equations with $\kappa>0$ and $n=1$ appeared in~\cite{lawlerDefiningSLEwithBrownianLoop,ZhanReversibilityWholeplaneSLE}. 
Moreover, we introduce multi-time martingale and show in Proposition~\ref{prop::multitime_mart_annulus} that any solution to the annulus BPZ equations gives a mutli-time martingale in the annulus setup. 
Multi-time martingales in the chordal setting and in the radial setting appeared in earlier literature~\cite{HealeyLawlerNSidedRadialSLE, HuangWuYangMultipleSLEsDysonBM, FengWuYangIsing, HuangPeltolaWuMultiradialSLEResamplingBP}. We extend this idea in the annulus setting. 
In all three settings, one starts from independent single-curve SLEs, inserts the Brownian loop correction and the boundary conformal covariance factors, and obtains a multi-time local martingale exactly when the partition function satisfies the corresponding BPZ system. In the annulus, the difference is that the calculation has to keep track of both boundary components and the moving modulus. 

We will prove Proposition~\ref{prop::annulusBPZ_kappa4} in Section~\ref{subsec::annulusBPZ}. Let us point out the main difficulty comparing to earlier works. 
When $\kappa=4$,  one often has explicit formulae for the partition functions; in the chordal case these formulae are algebraic in the marked points~\cite{PeltolaWuGlobalMultipleSLEs}, while in the radial case they are expressed through trigonometric and hyperbolic functions~\cite{KrusellWangWuCommutationRelation}. The BPZ verification in these settings is therefore mostly reduced to elementary algebraic manipulations. Our annulus verification for $\kappa=4$ is different. In the annulus case, the formulae involve rescaled Jacobi theta functions, making a direct verification much less transparent. The annulus and its covering space carry two periods, and this suggests a different way to handle the complicated identities. After applying the BPZ operator, the remaining identity is reduced to a residual term $G_{\eps}$ (see~\eqref{eqn::LZannLZcro_BPZ_Geps_goal}); we view $G_{\eps}$ as a meromorphic function of the complexified variables $(\bs{\alpha},\bs{\beta})$. The shift identities give double periodicity, and Laurent expansions show that all possible poles cancel. Hence $G_{\eps}$ is constant by Liouville's theorem on the compact complex torus, and the constant is obtained from degeneration $\alpha_j=x_j\delta$, $\beta_j=y_j\delta$ as $\delta\to0$.
Note that the dimension of the solution space of the chordal BPZ equations (with mild extra requirements) is Catalan number. However, the dimension of the solution space of the annulus BPZ is infinite, we have at least solutions $\LZann{n}$ and $\LZcro{n}^{(\ell)}$ with $\ell\in\Z$, see other solutions in Remark~\ref{rem::annulus_otherpatterns}.

\subsection{GFF in annulus and crossing probability}

\begin{theorem}[Crossing probability]
\label{thm::crossingproba}
Fix $r>0$ and an even number $n=2N$. 
For $\bs{\alpha}, \bs{\beta}\in\LX_n$, we write $\bs{x}=\ee^{\ii\bs{\alpha}}$ and $\bs{y}=\ee^{\ii\bs{\beta}-r}$. 
Let $\varphiann{n}$ be the bounded harmonic function in annulus $\A_r$ with alternating boundary data~\eqref{eqn::boundarydata}.
Suppose $\Gamma$ is Dirichlet GFF in annulus $\A_r$. 
Let $\gamma^j$ be the level line of $\Gamma+\varphiann{n}$ starting from $x_{j}$ for $1\le j\le n$ and denote $\bs{\gamma}=(\gamma^1, \ldots, \gamma^n)$.
We define $\cross(\bs{\gamma})$ and $\wind(\bs{\gamma})$ as in~\eqref{eqn::crossingevent_def} and~\eqref{eqn::windingevent_def}. 
We have 
\begin{align}\label{eqn::cross_wind_proba}
&\PP[\cross(\bs{\gamma}), \wind(\bs{\gamma})=(m;\ell)]=\frac{\LZcro{n}^{(\ell)}(r; \bs{\alpha}, \beta_{2m+1}, \ldots, \beta_{2m+n})}{\LZann{n}(r; \bs{\alpha}, \bs{\beta})}, \quad \text{for $m\in\{1, \ldots, N\}$ and $\ell\in\Z$};\\
&\PP[\cross(\bs{\gamma})]=\sum_{m=1}^N\frac{\LZcro{n}(r; \bs{\alpha}, \beta_{2m+1}, \ldots, \beta_{2m+n})}{\LZann{n}(r; \bs{\alpha}, \bs{\beta})}.\label{eqn::cross_proba}
\end{align}
where $\LZann{n}, \LZcro{n}^{(\ell)}, \LZcro{n}$ are partition functions defined in~\eqref{eqn::LZann_def}-\eqref{eqn::LZcro_def}. 
\end{theorem}

\begin{proposition}[Asymptotic]
\label{prop::asymptotics}
For $\bs{\theta}=(\theta_1, \ldots, \theta_n)\in\LX_n$, define 
\begin{align}\label{eqn::LZunitdisc_def}
\LZ_n(\bs{\theta})=\prod_{1\le i<j\le n}\left|\ee^{\ii\theta_i}-\ee^{\ii\theta_j}\right|^{\frac{1}{2}(-1)^{j-i}}, \qquad
\LZradial{n}(\bs{\theta})=\prod_{1\le i<j\le n}\left|\ee^{\ii\theta_i}-\ee^{\ii\theta_j}\right|^{\frac{1}{2}}. 
\end{align}
Assume the same setup as in Theorem~\ref{thm::crossingproba}. The crossing probabilities have the following asymptotic as $r\to\infty$: 
\begin{align}\label{eqn::crossproba_ell_asymp}
&\PP[\cross(\bs{\gamma}), \wind(\bs{\gamma})=(m;\ell)]=\sqrt{2}^{n^2}\frac{\LZradial{n}(\bs{\alpha})\LZradial{n}(\bs{\beta})}{\LZ_n(\bs{\alpha})\LZ_n(\bs{\beta})}\exp\left(-\frac{n^2r}{8}-\frac{1}{8r}\LA(\bs{\alpha}, \bs{\beta}; m; \ell)+O(1)\ee^{-r} \right);\\
&\PP[\cross(\bs{\gamma})]=\sqrt{2}^{n^2}\frac{\LZradial{n}(\bs{\alpha})\LZradial{n}(\bs{\beta})}{\sqrt{2\pi}\LZ_n(\bs{\alpha})\LZ_n(\bs{\beta})}\exp\left(-\frac{n^2r}{8}+\frac{\log r}{2}+\frac{1}{8r}\left(\sum_{j=1}^n(-1)^j(\alpha_j-\beta_j)\right)^2+O(1)\ee^{-r/n^2}\right);\label{eqn::crossproba_asymp}
\end{align}
where 
\begin{align*}
\LA(\bs{\alpha}, \bs{\beta}; m; \ell)=\left(\sum_{j=1}^{n}(\alpha_j-\beta_{j})-2m\pi-2n\pi\ell\right)^2-\left(\sum_{j=1}^{n}(-1)^j(\alpha_j-\beta_j)\right)^2. 
\end{align*}
\end{proposition}

We will complete the proof of Theorem~\ref{thm::crossingproba} and Proposition~\ref{prop::asymptotics} in Section~\ref{sec::GFF_levellines_annulus}. 
Let us explain the key ingredients in the proof.
\begin{itemize}
\item First, we derive the marginal law of the level line $\gamma^1$. This is an annulus $\SLE_4$ process with partition function $\LZann{n}$ in~\eqref{eqn::LZann_def} (see details in Section~\ref{sec::GFF_levellines_annulus}).  The coupling between $\SLE_4$ and the GFF is usually verified through harmonic martingale observables whose quadratic variations match the variation of the Green's function; this point of view goes back to the chordal theory~\cite{SchrammSheffieldContinuumGFF,DubedatSLEFreeField}. In the annulus, the harmonic observable contains theta-function interactions, the modulus derivative, and the interaction between the two boundary components. The new idea of our proof is to introduce an analytic lift whose imaginary part is the harmonic observable, see details in Lemma~\ref{lem::multiannulus_levelline_aux10}. Instead of checking the drift of the harmonic observable by a direct term-by-term cancellation, we apply It\^o's formula to its analytic lift and show that the imaginary part of the complex drift vanishes. The drift of the analytic lift is real on both boundary lines of the universal cover $\S_r$; by periodicity and harmonicity, its imaginary part then vanishes in the whole strip. This turns the main martingale verification into a boundary-value argument for an analytic function and separates it cleanly from the subsequent Green's function calculation.
\item Second, we construct a martingale observable for $\gamma^1$ which is given by the ratio between the two partition functions  $\LZcro{n}^{(\ell)}$ and $\LZann{n}$, see~\eqref{eqn::mtg_partition_ratio}. 
The fact that this process is a local martingale relies on Proposition~\ref{prop::annulusBPZ_kappa4}. 
The boundedness of the process relies the control $\LZcro{n}^{(\ell)}\le \LZann{n}$ proved in Proposition~\ref{prop::LZcro_LZann_mono} in Section~\ref{sec::Regularized Dirichlet energy}. 
We then derive the terminal value of the martingale observable when $\gamma^1$ makes the prescribed crossing. 
When $\gamma^1$ completes the prescribed crossing, this ratio degenerates into a ratio of partition functions in the simply connected domain $\A_r\setminus\gamma^1$, reducing the remaining problem to the probability that the other level lines form the rainbow pattern derived in~\cite{PeltolaWuGlobalMultipleSLEs}. 
\end{itemize}

\medbreak
Let us compare the asymptotic~\eqref{eqn::crossproba_asymp} with earlier result. 
We denote by $\PP_r$ the law of GFF in the annulus $\A_r$ as in Proposition~\ref{prop::asymptotics}.  
Fix $\bs{\alpha}\in\LX_n$ and write $\bs{x}=\ee^{\ii\bs{\alpha}}$. 
We denote by $\PP_*$ the law of GFF in the unit disc $\U$ with alternating boundary data: $\pi$ on $\cup_{j=1}^N(x_{2j-1}x_{2j})$ and $0$ on $\cup_{j=1}^N(x_{2j}x_{2j+1})$. Let $\gamma^j$ be the level line of the field starting from $x_j$ for $1\le j\le n$ and denote $\bs{\gamma}=(\gamma^1, \ldots, \gamma^n)$. We define $\cross(\bs{\gamma})$ to be the event that all $\gamma^1, \ldots, \gamma^n$ hit the centered disc of radius $\ee^{-r}$. It follows from~\cite[Theorem~1.4]{PeltolaWuGlobalMultipleSLEs} and~\cite[Theorem~1.2]{FengWuYangIsing} that the probability for $\cross(\bs{\gamma})$ has the following asymptotic as $r\to\infty$: 
\begin{align}\label{eqn::FWY}
\PP_*[\cross(\bs{\gamma})]=C\frac{\LZradial{n}(\bs{\alpha})}{\LZ_n(\bs{\alpha})}\exp\left(-\frac{n^2}{8}r\right)(1+O(\ee^{-ur})), 
\end{align}
where $C\in (0,\infty)$ and $u>0$ are constants depending on $n$. Comparing with the RHS of~\eqref{eqn::crossproba_asymp}, the leading term $\ee^{\frac{n^2r}{8}}$ are the same and the exponent $\frac{n^2}{8}$ is the arm exponent for $\SLE_4$. However, the subleading terms in the RHS of~\eqref{eqn::crossproba_asymp} is different from the one in the RHS of~\eqref{eqn::FWY}: 
\begin{align*}
\frac{\PP_r[\cross(\bs{\gamma})]}{\PP_*[\cross(\bs{\gamma})]}\sim \sqrt{r}. 
\end{align*} 
The extra $\sqrt{r}$ is due to the influence of the boundary of the inner hole $\ee^{-r}\U$: in the setup of Theorem~\ref{thm::crossingproba}, the boundary value of the GFF is bounded on $\ee^{-r}\partial\U$; while in the setup of GFF in the unit disc, the variance of the average of the field on $\ee^{-r}\partial\U$ is $r$. 

\subsection{Consistence with the previous construction}

We recall the construction of partition function from~\cite{ZhanRestrictionAnnulusSLE} and~\cite{JahangoshahiLawlerMultiplepathsSLE} and point out a byproduct of our analysis on partition functions. 
Fix $\kappa\in (0,4]$ and $n\ge 1$ and $\ell\in\Z$. 
Fix $r>0$ and $\bs{\alpha}, \bs{\beta}\in\LX_n$. For $1\le j\le n$, let $\gamma^j$ be a single annulus $\SLE_{\kappa}(\LFcro{1}^{(\kappa;\ell)})$ in $(\A_r;\ee^{\ii \alpha_j};\ee^{\ii \beta_j-r})$ (see Section~\ref{subsec::multiannulus_pre}). 
Let $\Pind{n}^{(\kappa;\ell)}=\Pind{n}^{(\kappa;\ell)}(\A_r; \ee^{\ii\bs{\alpha}}, \ee^{\ii\bs{\beta}})$ be the probability measure on $\bs{\gamma}=(\gamma^{1}, \ldots, \gamma^{n})$ under which the curves are independent.
Define
\begin{align} \label{eqn::multiannulus_pf_def}
\LFcro{n}^{(\kappa; \ell)}(r; \bs{\alpha}, \bs{\beta})=\prod_{j=1}^n\LFcro{1}^{(\kappa;\ell)}(r; \alpha_j, \beta_j)\times \Eind{n}^{(\kappa;\ell)}\left[\one_{\LE_{\emptyset}(\bs{\gamma})}\exp\left(\frac{\mathfrak{c}}{2}\blm(\A_r; \gamma^1, \ldots, \gamma^n)\right)\right],
\end{align}
where $\mathfrak{c}=\frac{(6-\kappa)(3\kappa-8)}{2\kappa}$ is central charge and 
$\LE_{\emptyset}(\bs{\gamma}) = \{ \gamma^{j} \cap \gamma^{i}=\emptyset, \, \forall i\neq j\}$ is the event that different curves are disjoint, and $\blm$ is the Brownian loop measure defined in~\eqref{eqn::blm_def}. Our partition function $\LZcro{n}^{(\ell)}$ coincides with this partition function when $\kappa=4$. 

\begin{proposition}
\label{prop::two_pf_equal}
Fix $n\ge 1$ and $\ell\in\Z$. 
The partition function $\LZcro{n}^{(\ell)}$ in~\eqref{eqn::LZcroell_def} 
coincides with $\LFcro{n}^{(\kappa; \ell)}$ in~\eqref{eqn::multiannulus_pf_def} when $\kappa=4$: for $r>0$ and $\bs{\alpha}, \bs{\beta}\in\LX_n$, we have
\begin{equation*}
\LFcro{n}^{(4; \ell)}(r; \bs{\alpha}, \bs{\beta})=\LZcro{n}^{(\ell)}(r; \bs{\alpha}, \bs{\beta}).
\end{equation*}
\end{proposition}

Proposition~\ref{prop::two_pf_equal} with $n=1$ is proved in~\cite{ZhanReversibilityWholeplaneSLE}, we will prove the general case $n\ge 2$ in Section~\ref{subsec::LZcron_ac}. 
The proof relies on a cascade relation for the annulus partition function. 
Such cascade relation is immediate for the construction in~\eqref{eqn::multiannulus_pf_def}, see Lemma~\ref{lem::earlierconstruction_marginal_conditional}. 
However, the cascade relation is not clear from the definition~\eqref{eqn::LZcroell_def}. To show that $\LZcro{n}^{(\ell)}$ enjoys the same cascade relation, we use an explicit formula for partition functions derived in~\cite[Theorem~1.5]{PeltolaWuGlobalMultipleSLEs} when $\kappa=4$ (see~\eqref{eqn::rainbow_pf4}), an alternative expression for Jacobi theta functions in Lemma~\ref{lem::JacobiTheta_technical}, and refined analysis of the Radon-Nikodym derivative in Section~\ref{subsec::LZcron_ac}.
\medbreak
Let us point out a byproduct of our construction of partition functions. 
Define
\begin{align}\label{eqn::multiannulus_withoutspiral_def}
\LFcro{n}^{(\kappa)}(r; \bs{\alpha}, \bs{\beta})=\sum_{\ell\in\Z}\LFcro{n}^{(\kappa;\ell)}(r; \bs{\alpha}, \bs{\beta}).
\end{align}
We believe the partition function $\LFcro{n}^{(\kappa)}$ for multi-annulus SLE has the following asymptotic: for $\kappa\in (0,4]$ and $n\ge 2$, as $r\to\infty$, 
\begin{align}\label{eqn::multiannulus_pf_general_asymp}
\LFcro{n}^{(\kappa)}(r; \bs{\alpha}, \bs{\beta})\asymp \LFradial{n}^{(\kappa)}(\bs{\alpha})\LFradial{n}^{(\kappa)}(\bs{\beta})\exp\left(\mathfrak{a}_n r+\frac{\mathfrak{c}}{2}\log r\right), 
\end{align}
where 
\[\mathfrak{b}=\frac{(6-\kappa)}{2\kappa}, \qquad \mathfrak{c}=\frac{(6-\kappa)(3\kappa-8)}{2\kappa}, \qquad \mathfrak{a}_n=n\mathfrak{b}-\frac{4n^2-(4-\kappa)^2}{8\kappa},\]
and 
\[\LFradial{n}^{(\kappa)}(\bs{\theta})=\prod_{1\le i<j\le n}\left|\ee^{\ii\theta_i}-\ee^{\ii\theta_j}\right|^{\frac{2}{\kappa}}, \qquad \text{for }\bs{\theta}\in\LX_n.\]
The authors in~\cite{JahangoshahiLawlerMultiplepathsSLE} derive~\eqref{eqn::multiannulus_pf_general_asymp} for $\kappa\in (0,4]$ and  $n\in \{1,2\}$ in~\cite[Eq.~(13) and Theorem~4]{JahangoshahiLawlerMultiplepathsSLE}. 
Our construction improves this control for the case when $\kappa=4$ (see~\eqref{eqn::LZcro_asymp}): for $\kappa=4$ and $n\ge 1$, as $r\to\infty$,
\begin{align*}
\LFcro{n}^{(4)}(r; \bs{\alpha}, \bs{\beta})=&\frac{\sqrt{2}^{n(n-1)}}{n\sqrt{\pi/2}}\LFradial{n}^{(4)}(\bs{\alpha})\LFradial{n}^{(4)}(\bs{\beta})\exp\left(\frac{n(2-n)}{8}r+\frac{1}{2}\log r+O(1)\ee^{-r/n^2}\right).
\end{align*}
We will prove a weaker bound for $\LFcro{n}^{(\kappa)}$ in~\eqref{eqn::multiannulus_pf_upperbound}, but we cannot prove~\eqref{eqn::multiannulus_pf_general_asymp} with $n\ge 3$ other than for $\kappa=4$. Note that the exponent $\mathfrak{a}_n$ in the exponential part in RHS of~\eqref{eqn::multiannulus_pf_general_asymp} can be predicted from boundary conformal weights $h_{1,2}$ and SLE arm exponents~\cite{WuAlternatingArmIsing, FengWuYangIsing}: 
\[\mathfrak{a}_n= n \underbrace{\frac{(6-\kappa)}{2\kappa}}_{h_{1,2}:=}-\underbrace{\frac{4n^2-(4-\kappa)^2}{8\kappa}}_{\text{SLE arm exponent}}.\]
Whereas, the central charge $\mathfrak{c}$ in the
polynomial part $\sqrt{r}^{\mathfrak{c}}$ is quite mysterious for us. 

\paragraph*{Outline.}
In Section~\ref{sec::Regularized Dirichlet energy}, we introduce regularized Dirichlet energy and calculate four examples of the energy. 
These four examples correspond to the four partition functions $\LZann{n}$ in~\eqref{eqn::LZann_def}, $\LZcro{n}^{(\ell)}$ in~\eqref{eqn::LZcroell_def} and $\LZ_n, \LZradial{n}$ in~\eqref{eqn::LZunitdisc_def}. 
Moreover, we prove $\LZcro{n}^{(\ell)}\le \LZann{n}$ in Section~\ref{sec::Regularized Dirichlet energy}. 
In Section~\ref{sec::multiannulus_SLE}, we introduce annulus BPZ equations for general $\kappa>0$ and provide a general framework for defining and analyzing multi-annulus SLE process. 
In Section~\ref{sec::multiannulu_SLE4}, we prove Propositions~\ref{prop::annulusBPZ_kappa4} and~\ref{prop::two_pf_equal}. 
In Section~\ref{sec::GFF_levellines_annulus}, we describe the law of the level lines of the GFF in annulus and complete the proof of Theorem~\ref{thm::crossingproba} and Proposition~\ref{prop::asymptotics}. 

\paragraph*{Acknowledgment.}
We thank Titus Lupu for helpful discussion about GFF level lines. 
We thank Eveliina Peltola for helpful discussion about BPZ equations. 
H.W. is supported by New Cornerstone Investigator Program 100001127. 
H.W. is partly affiliated at Yanqi Lake Beijing Institute of Mathematical Sciences and Applications, Beijing, China.
This project was initiated when the first and the third authors participated in a program hosted by the Hausdorff Research Institute for Mathematics (HIM), which is supported by the Deutsche Forschungsgemeinschaft (DFG, German Research Foundation) under Germany’s Excellence Strategy EXC-2047/1-390685813.

\section{Regularized Dirichlet energy}
\label{sec::Regularized Dirichlet energy}
\paragraph*{Regularized Dirichlet energy.}
We define the regularized Dirichlet energy of a harmonic function with boundary jumps and interior monodromies following~\cite[Section~5.2]{DubedatSLEFreeField}. Consider a domain $D$ with smooth boundary. 
Suppose $x_1, \ldots, x_n\in\partial D$ are boundary points and $y_1, \ldots, y_m\in D$ are interior points.
Let $u$ be a locally bounded harmonic function in $D \setminus \{y_1, \ldots, y_m\}$ that extends continuously to $\partial D \setminus \{x_1, \ldots, x_n\}$. Assume that $u$ has a jump of $\delta_j$ at $x_j\in\partial D$ for $1\le j\le n$, and has a monodromy of $M_k$ around $y_k \in D$ for $1 \le k \le m$ (i.e., $u$ increases by $M_k$ along a counterclockwise loop around $y_k$). 
We define the regularized Dirichlet energy for $u$ to be
\begin{equation*}
	\| u \|_{\nabla(D),\mathrm{reg}}^2 = \lim_{\substack{\eps_j\to 0, \, 1\le j\le n \\ \rho_k\to 0, \, 1\le k\le m}} \left( \int_{x\in D, \, |x-x_j|\ge \eps_j, \, |x-y_k|\ge \rho_k} |\nabla u |^2 \ud x + \sum_{j=1}^{n} \frac{\delta_j^2}{\pi} \log(\eps_j) + \sum_{k=1}^{m} \frac{M_k^2}{2\pi} \log(\rho_k) \right).
\end{equation*}

For a general domain $\Omega$ such that each component of $\partial\Omega$ is locally connected and that the marked boundary points $x_1, \ldots, x_n$ lie on $C^{1+\eps}$-boundary segments, and $y_1, \ldots, y_m$ are marked points in $\Omega$.
Let $\varphi:\Omega\to D$ be a conformal homeomorphism between $\Omega$ and $D$. Then $u\circ \varphi$ is the corresponding harmonic function in $\Omega$ and we define the regularized Dirichlet energy of $u\circ\varphi$ in $\Omega$ by  
\begin{equation} \label{eqn::reg_Dirichlet_COV}
	\| u\circ \varphi \|_{\nabla(\Omega),\mathrm{reg}}^2=\| u \|_{\nabla(D),\mathrm{reg}}^2 -  \sum_{j=1}^{n} \frac{\delta_j^2}{\pi} \log |\varphi'(x_j)| - \sum_{k=1}^{m} \frac{M_k^2}{2\pi} \log |\varphi'(y_k)|.
\end{equation}
This is well-defined because of the conformal covariance of the regularized Dirichlet energy~\cite[Section~5.2]{DubedatSLEFreeField}.

The goal of this section to calculate four examples for the regularized Dirichlet energy. They are the key players in the statement of our main conclusion. 
In Section~\ref{subsec::energy_unitdisc}, we calculate two examples of the energy in the unit disc. 
These two examples correspond to the two partition functions $\LZ_n$ and $\LZradial{n}$ in~\eqref{eqn::LZunitdisc_def}. 
In Section~\ref{subsec::JacobiTheta}, we derive properties for Jacobi theta functions that will be useful later. 
In Section~\ref{subsec::energy_annulus}, we calculate two examples of the energy in annulus.
These two examples correspond to the two partition functions $\LZann{n}$ and $\LZcro{n}^{(\ell)}$ in~\eqref{eqn::LZann_def} and~\eqref{eqn::LZcroell_def}. 
In Section~\ref{subsec::asymp}, we derive asymptotics of the partition functions. 
In Section~\ref{subsec::LZcro_LZann_mono}, we show that $\LZcro{n}^{(\ell)}$ is smaller than $\LZann{n}$(see Proposition~\ref{prop::LZcro_LZann_mono}). Such control will be important in the proof of Theorem~\ref{thm::crossingproba}.
\begin{proposition}\label{prop::LZcro_LZann_mono}
Fix $r>0$ and an even number $n=2N$ and $\ell\in\Z$. For $\bs{\alpha}, \bs{\beta}\in\LX_n$, we have
\begin{align}\label{eqn::LZcro_LZann_mono}
\LZcro{n}^{(\ell)}(r; \bs{\alpha}, \bs{\beta})\le \LZann{n}(r; \bs{\alpha}, \bs{\beta}). 
\end{align}
\end{proposition}

\subsection{Two examples in the unit disc}
\label{subsec::energy_unitdisc}

\begin{lemma} \label{lem::reg_Dirichlet_polygon}
Fix an even number $n=2N$ and $\bs{\theta}=(\theta_1, \ldots, \theta_{n})\in\LX_{n}$. We write $x_j=\ee^{\ii\theta_j}$ for $1\le j\le n$.  
Let $\varphi_n$ be the bounded harmonic function in $\U$ with alternating boundary data:
\begin{equation} \label{eqn::alternating_boundary_data}
\pi\text{ on }\cup_{j=1}^N (x_{2j-1}x_{2j}), \qquad 0\text{ on }\cup_{j=1}^N (x_{2j}x_{2j+1}). 
\end{equation}
Then
\begin{equation} \label{eqn::varphin_energy}
\| \varphi_n\|_{\nabla(\U),\mathrm{reg}}^2= -2\pi \sum_{1\le i<j\le n} (-1)^{j-i} \log |x_i-x_j|. 
\end{equation}
Moreover, its regularized Dirichlet energy is related to the partition function $\LZ_n$ in~\eqref{eqn::LZunitdisc_def}:
\begin{align*}
\LZ_n(\bs{\theta})=\exp\left(-\frac{1}{4\pi}\| \varphi_n\|_{\nabla(\Omega),\mathrm{reg}}^2\right)=\prod_{1\le i<j\le n}\left|\ee^{\ii\theta_i}-\ee^{\ii\theta_j}\right|^{\frac{1}{2}(-1)^{j-i}}. 
\end{align*}
\end{lemma}

\begin{proof}
The bounded harmonic function $\varphi_n$ in $\U$ with alternating boundary data~\eqref{eqn::alternating_boundary_data} is given by
\begin{equation*}
	\varphi_n(z)=\sum_{j=1}^{2N} (-1)^j (\arg(z-x_j) - \arg(y-x_j)),
\end{equation*}
where $y\in (x_{2N}x_1)$ is fixed. 
Thus,
\begin{align}\label{eqn::reg_Dirichlet_polygon_aux1}
	\| \varphi_n \|_{\nabla(\U),\mathrm{reg}}^2= & \sum_{j=1}^{2N} \| \arg(z-x_j) \|_{\nabla(\U),\mathrm{reg}}^2 + 2 \sum_{1\le i<j\le 2N} (-1)^{j-i} \left( \arg(z-x_i),\arg(z-x_j) \right)_{\nabla(\U), \mathrm{reg}}.
\end{align}
Let us check the two terms in~\eqref{eqn::reg_Dirichlet_polygon_aux1}.
\begin{itemize}
\item It is calculated in the first equation in~\cite[Page~1026]{DubedatSLEFreeField} that 
\begin{align}\label{eqn::Dubedat1026}
\left( \arg(z-x_i),\arg(z-x_j) \right)_{\nabla(\U), \mathrm{reg}}=-\pi\log|x_i-x_j|. 
\end{align}
\item From the rotation symmetry, we have $\| \arg(z-x_j) \|_{\nabla(\U),\mathrm{reg}}^2=\| \arg(z-1) \|_{\nabla(\U),\mathrm{reg}}^2$. Let us calculate $\| \arg(z-1) \|_{\nabla(\U),\mathrm{reg}}^2$. Note that $\arg(z-1)=\Im\log(z-1)$ and $|\nabla\arg(z-1)|^2=\frac{1}{|z-1|^2}$. We have
\begin{align*}
\int_{\U\setminus B(1,\eps)}\frac{\ud z}{|z-1|^2}=&\int_{-\pi/2}^{\pi/2}\one\{2\cos\phi\ge \eps\}\ud \phi\int_{\eps}^{2\cos\phi}\frac{1}{r^2}r\ud r \tag{set $z=1+\ee^{\ii\phi}$}\\
=&\int_{-\pi/2}^{\pi/2}\one\{2\cos\phi\ge \eps\}\ud \phi\left(\log(2\cos\phi)-\log\eps\right)\\
=&o(1)-\pi\log\eps, \quad\text{as }\eps\to 0. 
\end{align*} 
This shows that $\| \arg(z-1) \|_{\nabla(\U),\mathrm{reg}}^2=0$ and 
\begin{align}\label{eqn::reg_Dirichlet_constant}
\| \arg(z-x_j) \|_{\nabla(\U),\mathrm{reg}}^2=\| \arg(z-1) \|_{\nabla(\U),\mathrm{reg}}^2=0.
\end{align}
\end{itemize}
Plugging~\eqref{eqn::Dubedat1026} and~\eqref{eqn::reg_Dirichlet_constant} into~\eqref{eqn::reg_Dirichlet_polygon_aux1} gives~\eqref{eqn::varphin_energy} as desired. 
\end{proof}

\begin{lemma} \label{lem::varphiradial_spiral_energy}
Fix $n\ge 1$ and $\bs{\theta}=(\theta_1, \ldots, \theta_{n})\in\LX_{n}$. Define 
\begin{align}\label{eqn::varphiradial_def}
\varphiradial{n}(z)=-\frac{1}{2}\sum_{j=1}^n\arg\left(\frac{\ee^{\ii\theta_j}-z}{1-\overline{z}\ee^{\ii\theta_j}}\right)+\frac{n}{2}\arg(z), \quad z\in\U\setminus \{0\}. 
\end{align}
Then
\begin{align}\label{eqn::varphiradial_spiral_energy}
\| \varphiradial{n}\|_{\nabla(\U),\mathrm{reg}}^2=-2\pi \sum_{1\le i<j\le n} \log \left| \ee^{\ii\theta_i}-\ee^{\ii\theta_j} \right|.
\end{align}
Moreover, its regularized Dirichlet energy is related to the partition function $\LZradial{n}$ in~\eqref{eqn::LZunitdisc_def}:
\begin{align*}
\LZradial{n}(\bs{\theta})=\exp\left(-\frac{1}{4\pi}\| \varphiradial{n}\|_{\nabla(\U),\mathrm{reg}}^2\right)=\prod_{1\le i<j\le n}\left|\ee^{\ii\theta_i}-\ee^{\ii\theta_j}\right|^{\frac{1}{2}}. 
\end{align*} 
\end{lemma}

\begin{proof}
From~\eqref{eqn::varphiradial_def}, we have
\begin{align*} 
	\varphiradial{n}(z)=\sum_{j=1}^{n}\frac{ \theta_j}{2}-\sum_{j=1}^n\arg\left(\ee^{\ii\theta_j}-z\right)+\frac{n}{2}\arg(z), \quad z\in\U\setminus\{0\}. 
\end{align*}
By Green's formula, we have
\begin{equation} \label{eqn::varphiradial_spiral_energy_aux2}
	\| \varphiradial{n}\|_{\nabla(\U),\mathrm{reg}}^2= \lim_{\eps\to 0} \left( \int_{C_{\mathrm{horizon}}\cup C_{\mathrm{in}}\cup C_{\mathrm{out}} \cup C_{\mathrm{circle}}} \varphiradial{n} \partial_{\nu} \varphiradial{n} + \left( \frac{\pi n^2}{2}+n\pi \right) \log \eps  \right),
\end{equation}
where $\nu$ is the normal vector and
\begin{align*}
\begin{split}
	&C_{\mathrm{horizon}}=(\eps \ee^{0\ii},\ee^{0\ii}) \cup (\eps \ee^{2\pi\ii},\ee^{2\pi \ii}), \quad C_{\mathrm{in}}=\partial B(0,\eps), \\
	&C_{\mathrm{out}}=\bigcup_{j=1}^{n} \left(\ee^{\ii(\theta_j+\eps)},\ee^{\ii(\theta_{j+1}-\eps)}\right), \quad C_{\mathrm{circle}}=\bigcup_{j=1}^{n} \partial B(\ee^{\ii\theta_j},\eps) \cap \U.
\end{split}
\end{align*}
Let us evaluate the integral in the RHS of~\eqref{eqn::varphiradial_spiral_energy_aux2}.
\begin{itemize}
\item First, let us investigate the integral over $C_{\mathrm{circle}}$ as $\eps\to 0$. For $z\in \partial B(\ee^{\ii\theta_j},\eps)\cap \U$, since $\varphiradial{n}(z)$ and $\partial_{\nu} \varphiradial{n}(z)$ is bounded, we have
\begin{align}  \label{eqn::varphiradial_spiral_energy_aux3}
	\int_{C_{\mathrm{circle}}} \varphiradial{n} \partial_{\nu} \varphiradial{n} \ud s= \sum_{j=1}^{n} \int_{\partial B( \ee^{\ii\theta_j},\eps )\cap \U}  \varphiradial{n} \partial_{\nu} \varphiradial{n} \ud s \to 0,\quad \text{as } \eps\to 0.
\end{align}
\item Second, let us investigate the integral over $C_{\mathrm{in}}$. We have
\begin{align} \label{eqn::varphiradial_spiral_energy_aux4}
	& \int_{C_{\mathrm{in}}} \varphiradial{n} \partial_{\nu} \varphiradial{n} \ud s \\
	= & \int_{\partial B(0,\eps)} \left( \sum_{j=1}^{n}\frac{ \theta_j}{2}-\sum_{j=1}^n\arg\left(\ee^{\ii\theta_j}-z\right)+\frac{n}{2}\arg(z) \right) \left( -\sum_{j=1}^n \partial_{\nu} \arg\left(\ee^{\ii\theta_j}-z\right)\right) \ud s \to 0, \quad \text{as }\eps\to 0. 	\notag 
\end{align}
\item Then, let us investigate the integral over $C_{\mathrm{out}}$. Let $\psi(z)$ be the harmonic conjugate of $\varphiradial{n}(z)$: 
\begin{align*}
	\psi(z)=-\sum_{j=1}^n\log\left|\ee^{\ii\theta_j}-z\right|+\frac{n}{2}\log|z|, \quad z\in\U\setminus\{0\}.
\end{align*}
Then $\psi(z)+\ii \varphiradial{n}(z)$ is analytic and we have the Cauchy-Riemann equations
\begin{align*}
	\partial_x \psi(z)=\partial_y \varphiradial{n}(z), \quad \partial_y \psi(z)=-\partial_x \varphiradial{n}(z).
\end{align*}
On the boundary $C_{\mathrm{out}}$, we have $\partial_{\nu} \varphiradial{n} = -\partial_l \psi$ where $l$ is the tangent vector. Since $\varphiradial{n}(z)=-n \pi/2+j\pi$ on the arcs $(\ee^{\ii\theta_j}, \ee^{\ii\theta_{j+1}})$, we have
\begin{align}  \label{eqn::varphiradial_spiral_energy_aux5}
		\int_{C_{\mathrm{out}}} \varphiradial{n} \partial_{\nu} \varphiradial{n} \ud s =&  -\int_{C_{\mathrm{out}}} \varphiradial{n} \partial_l \psi \ud s \notag\\
		= & \sum_{j=1}^{n} \left( (-n\pi/2+j\pi) \psi(\ee^{\ii(\theta_j+\eps)}) - (-n\pi/2+(j-1)\pi) \psi(\ee^{\ii(\theta_j-\eps)}) \right) \notag\\
		& - n\pi/2 \left( \psi(\ee^{0\ii}) + \psi(\ee^{2\pi\ii}) \right) \notag\\
		= & o(1)+\pi \sum_{j=1}^{n} \psi(\ee^{\ii(\theta_j+\eps)}) + n\pi \sum_{j=1}^{n} \log |\ee^{\ii\theta_j}-1|. 
\end{align}
\item Finally, let us investigate the integral over $C_{\mathrm{horizon}}$. On the horizontal segment $(\eps \ee^{0\ii},\ee^{0\ii})$, we have $\partial_{\nu} \varphiradial{n} = -\partial_y \varphiradial{n}=-\partial_x \psi$. On the horizontal segment $(\eps \ee^{2\pi\ii},\ee^{2\pi \ii})$, we have $\partial_{\nu} \varphiradial{n} = \partial_y \varphiradial{n}=\partial_x \psi$. Note that for $x\in (\eps,1)$,
\begin{equation*} 
	\varphiradial{n}(\ee^{2\pi\ii} x) - \varphiradial{n}(\ee^{0\ii} x) =n\pi.
\end{equation*}
Thus we have
\begin{align} \label{eqn::varphiradial_spiral_energy_aux6}
\begin{split}
	\int_{C_{\mathrm{out}}} \varphiradial{n} \partial_{\nu} \varphiradial{n} \ud s = & \int_{\eps}^{1} \left( \varphiradial{n}(\ee^{2\pi\ii} x) - \varphiradial{n}(\ee^{0\ii} x) \right) \partial_x \psi \ud x \\
	= & n\pi \left( \psi(1)-\psi(\eps) \right) \\
	= & -\frac{n^2 \pi \log \eps}{2} - n\pi \sum_{j=1}^{n} \log |\ee^{\ii\theta_j}-1| + o(1).	
\end{split}
\end{align}
\end{itemize}
Plugging~(\ref{eqn::varphiradial_spiral_energy_aux3},\ref{eqn::varphiradial_spiral_energy_aux4},\ref{eqn::varphiradial_spiral_energy_aux5},\ref{eqn::varphiradial_spiral_energy_aux6}) into~\eqref{eqn::varphiradial_spiral_energy_aux2}, we obtain~\eqref{eqn::varphiradial_spiral_energy} as desired.
\end{proof}

\subsection{Properties of Jacobi theta functions}
\label{subsec::JacobiTheta}
\begin{lemma}\label{lem::Theta_basic}
\begin{itemize}
\item The functions $\Theta_1$ and $\Theta_3$ satisfy the following heat equations:
\begin{equation}\label{eqn::Theta_heat}
	\partial_r \Theta_1(r; z) = \partial_z^2 \Theta_1(r; z), \quad 
	\partial_r \Theta_3(r; z) = \partial_z^2 \Theta_3(r; z).
\end{equation}
\item The function $\Theta_1(r; \cdot)$ has period $2\pi$ while the function $\Theta_3(r; \cdot)$ has anti-period $2\pi$:
\begin{equation} \label{eqn::Theta_period}
	\Theta_1(r; z)=-\Theta_1(r; z+2\pi), \quad \Theta_3(r; z)=\Theta_3(r; z+2\pi).
\end{equation}
\item The functions $\Theta_1$ and $\Theta_3$ satisfy the following imaginary shift identities:
\begin{equation}\label{eqn::Theta_shift}
	\Theta_1(r; z + \ii r) = \ii \ee^{\frac{r}{4} - \frac{\ii z}{2}} \Theta_3(r; z), \quad 
	\Theta_3(r; z + \ii r) = \ii \ee^{\frac{r}{4} - \frac{\ii z}{2}} \Theta_1(r; z).
\end{equation}
\item The functions $\Theta_1$ and $\Theta_3$ have the following asymptotics as $r\to \infty$: for $x\in \R$, we have
\begin{align} \label{eqn::Theta_asymptotic}
	\Theta_1(r; x)= 2 \ee^{-r/4} \sin\left(\frac{x}{2}\right) (1+O(1)\ee^{-2r}),\quad
	\Theta_3(r; x)= 1 + O(1)\ee^{-r},
\end{align}
where the $O(1)$ terms have upper bounds independent of $x$.
\end{itemize}
\end{lemma}
\begin{proof}
Equations~(\ref{eqn::Theta_heat},\ref{eqn::Theta_period},\ref{eqn::Theta_shift}) directly follow from~\eqref{eqn::JacobiTheta_productexpansion}. By~\eqref{eqn::JacobiTheta_productexpansion}, we have
\begin{align*}
	\frac{\Theta_1(r; x)}{2 \ee^{-r/4} \sin\left(\frac{x}{2}\right)} = & \prod_{m=1}^{\infty} \left(1-\ee^{-2mr} \right) \left(1 - 2\ee^{-2mr}\cos x +\ee^{-4mr} \right)
	\in  \left(\prod_{m=1}^{\infty} \left(1-\ee^{-2mr} \right)^3,\prod_{m=1}^{\infty}\left(1+\ee^{-2mr} \right)^3\right), \\
	\Theta_3(r; x)= & \prod_{m=1}^\infty \left(1 - \ee^{-2mr}\right) \left(1 - 2 \ee^{-(2m-1)r} \cos(x) + \ee^{-(4m-2)r} \right)\\
	\in & \left(\prod_{m=1}^{\infty} \left(1-\ee^{-2mr} \right) \left(1 - \ee^{-(2m-1)r}\right)^2,\prod_{m=1}^{\infty} \left(1-\ee^{-2mr} \right) \left(1 + \ee^{-(2m-1)r}\right)^2\right),
\end{align*}	
which gives~\eqref{eqn::Theta_asymptotic} as desired.
\end{proof}

\begin{lemma}\label{lem::Theta_bc}
The functions $\Theta_1(r; \cdot)$ and $\Theta_3(r; \cdot)$ are analytic and have no zeros in the open strip $\mathbb{S}_r$. Consequently, one can define single-valued, continuous branch of the argument functions $\arg\Theta_1$ and $\arg\Theta_3$ in $\mathbb{S}_r$. These functions are harmonic in $\mathbb{S}_r$ and can be continuously extended to the boundary with the following boundary data: 
\begin{align} \label{eqn::Theta_bc}
\begin{split}
	&\begin{cases}
		\arg \Theta_1(r; x)= -\ell \pi, \quad x\in (2\ell\pi,2(\ell+1)\pi), \quad \ell\in \Z;\\
		\arg \Theta_1(r; x+\ii r)=\frac{\pi}{2}-\frac{x}{2}, \quad x\in\R;
	\end{cases}\\
	&\begin{cases}
		\arg \Theta_3(r; x)=0, \quad x\in\R;\\
		\arg \Theta_3(r; x+\ii r)=\frac{\pi}{2}-\frac{x}{2}+\ell \pi, \quad x\in (2\ell\pi,2(\ell+1)\pi), \quad \ell\in \Z.
	\end{cases}	
\end{split}
\end{align}
\end{lemma}
\begin{proof}
Let us check~\eqref{eqn::Theta_bc} by tracking the continuous variation of the argument within the strip $\mathbb{S}_r$. 
First, let us consider $\Theta_3(r; \cdot)$ on the boundary at bottom ($y=0$) and $\Theta_1(r; \cdot)$ on the boundary at top ($y=r$). Since $\Theta_3(r; x) > 0$ for all $x  \in \mathbb{R}$, we have $\arg \Theta_3(r; x) = 0$. By the shift identity~\eqref{eqn::Theta_shift}, we have $\arg \Theta_1(r; x+\ii r) = \frac{\pi}{2} - \frac{x}{2}$.
\medbreak
Next, let us consider $\arg \Theta_1(r; \cdot)$ on the boundary at bottom ($y=0$). Tracking the argument along the vertical segment $z = (2\ell+1)\pi + \ii y$ for $y \in (0, r)$:
\begin{equation*}
	\Theta_1(r; (2\ell+1)\pi + \ii y) = 2 (-1)^\ell \ee^{-r/4} \cosh(y/2) \prod_{m=1}^\infty \left(1 - \ee^{-2mr}\right) \left(1 + \ee^{-2mr-y}\right) \left(1 + \ee^{-2mr+y}\right) \in \R\setminus \{0\},
\end{equation*}
we see that the argument $\Theta_1(r; \cdot)$ must be constant along this path. Thus
\begin{equation*}
	\arg \Theta_1(r; (2\ell+1)\pi)=\arg \Theta_1(r; (2\ell+1)\pi+\ii r)=-\ell\pi.
\end{equation*}
Since $\Theta_1(r; x)$ has no zeros in the open interval $(2\ell\pi, 2(\ell+1)\pi)$, its argument remains constantly $-\ell\pi$ throughout this interval.
\medbreak
Finally, let us consider $\arg \Theta_3(r; \cdot)$ on the boundary at top ($y=r$). From the shift identity~\eqref{eqn::Theta_shift}, we have
\begin{equation} \label{eqn::Theta_bc_aux1}
	\arg \Theta_3(r; x+\ii r)\equiv \frac{\pi}{2} - \frac{x}{2} +\arg \Theta_1(r; x)=\frac{\pi}{2} - \frac{x}{2}-\ell\pi \quad (\mathrm{mod} \; 2\pi), \quad x\in (2\ell\pi, 2(\ell+1)\pi).
\end{equation}
To find the exact continuous branch, we evaluate $\Theta_3$ along the vertical line $z = (2\ell+1)\pi + \ii y$:
\begin{equation*}
	\Theta_3(r; (2\ell+1)\pi + \ii y) = \prod_{m=1}^\infty \left(1 - \ee^{-2mr}\right) \left(1 + \ee^{-(2m-1)r-y}\right) \left(1 + \ee^{-(2m-1)r+y}\right) >0. 
\end{equation*}
Hence $\Theta_3$ is strictly positive along this line and we conclude that 
\[
\arg \Theta_3(r; (2\ell+1)\pi + \ii r) = \arg \Theta_3(r; (2\ell+1)\pi) = 0.
\]
Plugging into~\eqref{eqn::Theta_bc_aux1}, we obtain $\arg \Theta_3(r; x+\ii r) = \frac{\pi}{2} - \frac{x}{2} + \ell\pi$ for $x \in (2\ell\pi, 2(\ell+1)\pi)$ as desired.
\end{proof}

In the following lemma, we give equivalent expressions for $\Theta_1$ and $\Theta_3$. 
This will be used in the proof of Proposition~\ref{prop::LZcro_LZann_mono}.
\begin{lemma}\label{lem::JacobiTheta_technical}
For $r>0$ and $z\in\C$, we have
\begin{align}\label{eqn::Jac_theta1}
\Theta_1(r;z)=&-\ii\sqrt{\frac{\pi}{r}}\ee^{-\frac{z^2}{4r}}\Theta_1\left(\frac{\pi^2}{r};\frac{\ii\pi z}{r}\right)\notag\\
=&2\sqrt{\frac{\pi}{r}}\ee^{-\frac{z^2}{4r}} \ee^{-\frac{\pi^2}{4r}} \sinh\left(\frac{\pi z}{2r}\right)\prod_{m=1}^\infty\left(1-\ee^{\frac{-2m\pi^2}{r}}\right)\left(1-\ee^{\frac{-2m\pi^2}{r}-\frac{\pi z}{r}}\right)\left(1-\ee^{\frac{-2m\pi^2}{r}+\frac{\pi z}{r}}\right);\\
\label{eqn::Jac_theta2}
\Theta_3(r;z)=&-\ii\ee^{\frac{\ii z}{2}-\frac{r}{4}}\Theta_1(r;z+\ii r)\notag\\
=&2\sqrt{\frac{\pi}{r}}\ee^{-\frac{z^2}{4r}} \ee^{-\frac{\pi^2}{4r}} \cosh\left(\frac{\pi z}{2r}\right)\prod_{m=1}^\infty\left(1-\ee^{\frac{-2m\pi^2}{r}}\right)\left(1+\ee^{\frac{-2m\pi^2}{r}-\frac{\pi z}{r}}\right)\left(1+\ee^{\frac{-2m\pi^2}{r}+\frac{\pi z}{r}}\right). 
\end{align}
\end{lemma}

\begin{proof}
First, we claim that 
\begin{equation}\label{eqn::eta}
\ee^{-\frac{\pi^2}{12 r}}\prod_{m=1}^\infty\left(1-\ee^{\frac{-2m\pi^2}{r}}\right)=\sqrt{\frac{r}{\pi}}\ee^{-\frac{r}{12}}\prod_{m=1}^\infty\left(1-\ee^{-2mr}\right).
\end{equation}
The identity~\eqref{eqn::eta} is precisely the modular transformation formula for the Dedekind eta function. Indeed, recalling that the Dedekind eta function is defined as
\[\eta(\tau)=\ee^{\frac{\pi\ii\tau}{12}}\prod_{m=1}^{\infty}(1-\ee^{2\pi\ii m\tau}),\quad \tau\in\HH,\]
and satisfies the standard transformation law
$\eta(-1/\tau)=\sqrt{-\ii\tau}\eta(\tau)$.
We obtain~\eqref{eqn::eta} by setting $\tau=\frac{\ii r}{\pi}$ and taking the absolute value. 
\medbreak
The identities~\eqref{eqn::Jac_theta1} and~\eqref{eqn::Jac_theta2} are precisely the Jacobi imaginary transformation applied to our scaled theta functions. For completeness, we give a short proof of~\eqref{eqn::Jac_theta1}; then~\eqref{eqn::Jac_theta2} follows immediately from the shift identity~\eqref{eqn::Theta_shift}. For~\eqref{eqn::Jac_theta1}, define 
\[H(z)=-\ii\sqrt{\frac{\pi}{r}}\ee^{-\frac{z^2}{4r}}\frac{\Theta_1\left(\frac{\pi^2}{r};\frac{\ii\pi z}{r}\right)}{ \Theta_1(r;z)},\quad z\in\C.\]
Note that $H$ has no poles. Moreover, using the quasi-periodicity~\eqref{eqn::Theta_period} and the imaginary shift identity~\eqref{eqn::Theta_shift}, one checks directly that $H$ is doubly periodic:
\[H(z+2\pi)=H(z),\quad H(z+2\ii r)=H(z),\quad \text{ for }z\in\C.\]
This implies that $H$ is a constant due to Liouville's theorem. By~\eqref{eqn::eta} and 
\begin{equation} \label{eqn::varphiann_energy_aux6}
	\lim_{\eps\to 0} \Theta_1(r; \eps)/\eps = \partial_z \Theta_1(r; 0) = \ee^{-\frac{r}{4}} \prod_{m=1}^{\infty} (1-\ee^{-2mr})^3, 
\end{equation}we obtain $H(0)=1$. This gives~\eqref{eqn::Jac_theta1} and completes the proof.
\end{proof}

Recall that $H_1(r;z) = 2 \partial_z \log \Theta_1(r;z), $ is defined in~\eqref{eqn::H1H3_def}. We derive a property for $H_1$ in the following lemma. It will be used in Section~\ref{sec::GFF_levellines_annulus}.
	\begin{lemma} \label{lem::identity_H1}
		For $w_1,w_2\in \S_r$, we have
		\begin{align} \label{eqn::identity_H1}
			\begin{split}	
				& \Im H_1(r;w_1)\Im H_1(r;w_2) \\
				=& \Re\Big(
				-\frac12\left(\partial_zH_1(r;w_1-w_2)-\partial_zH_1(r;w_1-\overline{w_2})\right)
				-\frac14\left(H_1(r;w_1-w_2)^2-H_1(r;w_1-\overline{w_2})^2\right)\\
				& \qquad
				+\frac12H_1(r;w_1-w_2)\left(H_1(r;w_1)-H_1(r;w_2)\right)
				-\frac12H_1(r;w_1-\overline{w_2})\left(H_1(r;w_1)-\overline{H_1(r;w_2)}\right)
				\Big).	
			\end{split}
		\end{align}
	\end{lemma}
	\begin{proof}
		Letting
		\begin{align*}
			A(w_1,w_2)=&
			-\frac12\left(\partial_zH_1(r;w_1-w_2)-\partial_zH_1(r;w_1-\overline{w}_2)\right)-\frac14\left(H_1(r;w_1-w_2)^2-H_1(r;w_1-\overline{w}_2)^2\right)\\
			&+\frac12H_1(r;w_1-w_2)\left(H_1(r;w_1)-H_1(r;w_2)\right)-\frac12H_1(r;w_1-\overline{w}_2)\left(H_1(r;w_1)-\overline{H_1(r;w_2)}\right),
		\end{align*}
		and
		\begin{equation*}
			D(w_1,w_2)=\Re A(w_1,w_2)-\Im H_1(r;w_1)\Im H_1(r;w_2),
		\end{equation*}
		we will show that $D(w_1,w_2)=0$.
		\medbreak
		First, let us check the double-periodicity of $D(w_1,w_2)$ in $w_1$. The $2\pi$-periodicity follows immediately from~\eqref{eqn::H_shift_identities}. We check the imaginary period. When $w_1$ is replaced by $w_1+2\ii r$, by~\eqref{eqn::H_shift_identities}, we have
		\begin{align*}
			&\Re A(w_1+2\ii r,w_2)-\Re A(w_1,w_2)
			=\Re\left(\ii\left(H_1(r;w_2)-\overline{H_1(r;w_2)}\right)\right)=-2\Im H_1(r;w_2), \\
			&\Im H_1(r;w_1+2\ii r)\Im H_1(r;w_2)-\Im H_1(r;w_1)\Im H_1(r;w_2)=-2\Im H_1(r;w_2).
		\end{align*}
		Therefore $D(w_1+2\ii r,w_2)=D(w_1,w_2)$, and $D(\cdot,w_2)$ is doubly periodic.
		
		\medbreak
		Next, let us inspect the possible singularities of $D(\cdot,w_2)$. By the double-periodicity, it suffices to inspect the singularities when $w_1\to 0$, $w_1\to w_2$ and $w_1\to \overline{w}_2$. All unlisted terms are analytic in the corresponding local coordinate. Applying~\eqref{eqn::Laurrent_expand_LE}, we have the following observations.
		
		\begin{itemize}
			\item Let $z=w_1\to 0$.
			The possible singular term in $A(w_1,w_2)$ is
			\begin{align*}
				&\frac12H_1(r;w_1)\left(H_1(r;w_1-w_2)-H_1(r;w_1-\overline{w}_2)\right)=-\frac{2\ii\Im H_1(r;w_2)}{z}+O(1).
			\end{align*}
			Thus the possible singular term in $D(w_1,w_2)$ is
			\[
			\Re\left(-\frac{2\ii\Im H_1(r;w_2)}{z}\right)-\Im\left(\frac{2}{z}\right)\Im H_1(r;w_2)=0.
			\]
			Thus $D(\cdot,w_2)$ has no pole when $w_1\to0$.
			\item Let $z=w_1-w_2\to 0$.
			The possible singular terms in $A(w_1,w_2)$ are
			\begin{align*}
				-\frac12\partial_zH_1(r;z)-\frac14H_1(r;z)^2+\frac12H_1(r;z)\left(H_1(r;w_1)-H_1(r;w_2)\right)=O(1).
			\end{align*}
			Thus $D(\cdot,w_2)$ has no pole when $w_1\to w_2$.
			\item Let $z=w_1-\overline{w}_2\to 0$. The possible singular terms in $A(w_1,w_2)$ are
			\begin{align*}
				\frac12\partial_zH_1(r;z)+\frac14H_1(r;z)^2-\frac12H_1(r;z)\left(H_1(r;w_1)-H_1(r;\overline{w}_2)\right)=O(1),
			\end{align*}
			where we used $H_1(r;w_1)-H_1(r;\overline{w}_2)=O(z)$. Thus $D(\cdot,w_2)$ has no pole when $w_1\to \overline{w}_2$.
		\end{itemize}
		Therefore $D(\cdot,w_2)$ is a doubly periodic harmonic function with no singularities, which implies that it is a constant in $w_1$.
		\medbreak
		Finally, let us take the limit $w_1\to x\in \R$. Note that
		\[
		H_1(r;x-\overline{w}_2)=\overline{H_1(r;x-w_2)},\qquad \partial_zH_1(r;x-\overline{w}_2)=\overline{\partial_zH_1(r;x-w_2)}.
		\]
		We have
		\begin{align*}
			A(x,w_2)=&
			-\frac12\left(\partial_zH_1(r;x-w_2)-\overline{\partial_zH_1(r;x-w_2)}\right)-\frac14\left(H_1(r;x-w_2)^2-\overline{H_1(r;x-w_2)^2}\right)\\
			&+\frac12\left(H_1(r;x-w_2)\left(H_1(r;x)-H_1(r;w_2)\right)-\overline{H_1(r;x-w_2)\left(H_1(r;x)-H_1(r;w_2)\right)}\right),
		\end{align*}
		which is purely imaginary. Therefore $\Re A(x,w_2)=0$, and consequently the constant $D(\cdot,w_2)$ is zero. This gives~\eqref{eqn::identity_H1} as desired.
\end{proof}

\subsection{Two examples in the annulus}
\label{subsec::energy_annulus}

\begin{lemma} \label{lem::reg_Dirichlet_annular}
Fix $r>0$ and an even number $n=2N$. For $\bs{\alpha}, \bs{\beta}\in\LX_n$, we write $\bs{x}=\ee^{\ii\bs{\alpha}}$ and $\bs{y}=\ee^{\ii\bs{\beta}-r}$.
Let $\varphiann{n}$ be the bounded harmonic function in annulus $\A_r$ with alternating boundary data~\eqref{eqn::boundarydata}. Then
\begin{align} \label{eqn::varphiann_energy}
	\begin{split}
		\| \varphiann{n} \|_{\nabla(\A_r),\mathrm{reg}}^2= & \frac{\pi}{2r}\left(\sum_{j=1}^{n}(-1)^j (\alpha_j-\beta_j) \right)^2-6n\pi \sum_{m=1}^{\infty} \log \left(  1-\ee^{-2mr} \right) - \frac{n}{2}\pi r \\
		& - 2\pi \sum_{1\le i<j \le n} (-1)^{j-i} \left( \log \left| \Theta_1(r; \beta_{j}-\beta_i) \right| +\log \left| \Theta_1(r; \alpha_{j}-\alpha_i) \right| \right) \\
		& +2 \pi \sum_{i=1}^{n} \sum_{j=1}^{n} (-1)^{j-i} \log \left| \Theta_3(r; \beta_{j}-\alpha_i) \right| .
	\end{split}
\end{align}
Moreover, its regularized Dirichlet energy is related to the partition function $\LZann{n}$ in~\eqref{eqn::LZann_def}: 
\begin{align}\label{eqn::LZann_varphiann}
\LZann{n}(r; \bs{\alpha}, \bs{\beta})=\exp\left(-\frac{1}{4\pi}\| \varphiann{n} \|_{\nabla(\A_r),\mathrm{reg}}^2\right). 
\end{align}
\end{lemma}

\begin{proof}
Recall the covering map $q$ defined in~\eqref{eqn::annulus_strip_def}. Then $v:=\varphiann{n}\circ q$ is the bounded harmonic function in the infinite strip $\S_r=q^{-1}(\A_r)$ with alternating boundary data:\footnote{with the convention that $\alpha_{2N+j}=\alpha_j+2\pi$ and $\beta_{2N+j}=\beta_j+2\pi$.} 
\begin{align*}
\begin{cases}
	\pi \text{ on } \bigcup_{\ell\in\Z} \bigcup_{j=1}^N \left( (\alpha_{2j-1}+2\pi \ell,\alpha_{2j}+2\pi \ell)\cup(\beta_{2j-1}+2\pi \ell+\ii r,\beta_{2j}+2\pi \ell+\ii r) \right), \\
	0 \text{ on } \bigcup_{\ell\in\Z} \bigcup_{j=1}^N \left( (\alpha_{2j}+2\pi \ell,\alpha_{2j+1}+2\pi \ell)\cup(\beta_{2j}+2\pi \ell+\ii r,\beta_{2j+1}+2\pi \ell+\ii r)\right).
\end{cases}
\end{align*}
By the conformal covariance of regularized Dirichlet energy~\eqref{eqn::reg_Dirichlet_COV}, we have
\begin{align}\label{eqn::varphiann_energy_aux0}
\begin{split}
	\| \varphiann{n} \|_{\nabla(\A_r),\mathrm{reg}}^2= & \| v \|_{\nabla((0,2\pi)\times (0,\ii r)),\mathrm{reg}}^2 + \pi \sum_{j=1}^{2N} \left( \log |q'( \alpha_j)| + \log |q'( \beta_j+\ii r)| \right) \\
	= & \| v \|_{\nabla((0,2\pi)\times (0,\ii r)),\mathrm{reg}}^2 - 2N\pi r.	
\end{split}
\end{align}
\medbreak
Let us calculate the regularized Dirichlet energy of $v$ in the square $(0,2\pi)\times(0,\ii r)$.
By~\eqref{eqn::Theta_bc}, the explicit expression of $v$ is given by:
\begin{equation} \label{eqn::varphiann_energy_harmonic_function}
	v(z)=\sum_{j=1}^{2N} (-1)^j \arg \Theta_1(r; z-\alpha_j) - \sum_{j=1}^{2N} (-1)^j \arg \Theta_3(r; z-\beta_j) - \sum_{j=1}^{2N} (-1)^j (\alpha_j-\beta_j) \frac{\Im(z)}{2r}.
\end{equation}
 By Green's formula, we have
\begin{equation} \label{eqn::varphiann_energy_aux1}
	\| v \|_{\nabla((0,2\pi)\times (0,\ii r)),\mathrm{reg}}^2 = \lim_{\eps\to 0} \left( \int_{C_{\mathrm{top}}\cup C_{\mathrm{bottom}} \cup C_{\mathrm{circle}}}  v \partial_{\nu} v \ud s + 4N \pi \log \eps \right),
\end{equation}
where\footnote{For $z_1,z_2\in \C$, we denote $(z_1,z_2)$ the straight segment from $z_1$ to $z_2$.}
\begin{equation*}
	C_{\mathrm{top}}= \bigcup_{j=1}^{2N} (\beta_{j}+\eps+\ii r,\beta_{j+1}-\eps+\ii r), \quad C_{\mathrm{bottom}}= \bigcup_{j=1}^{2N} (\alpha_j+\eps, \alpha_{j+1}-\eps),
\end{equation*}
and $C_{\mathrm{circle}}$ denotes the union of small semicircles of radius $\eps$ pointing into the domain around each marked point on the boundaries; and we use the fact that the line integrals over the vertical segments at $x=0$ and $x=2\pi$ perfectly cancel each other due to the periodicity.
\medbreak

Let us evaluate the integral in the RHS of~\eqref{eqn::varphiann_energy_aux1}.
\begin{itemize}
\item First, let us investigate the integral over $C_{\mathrm{circle}}$ as $\eps\to 0$. For $z\in \partial B(\alpha_1,\eps)\cap \S_r$, since $v(z)$ and $\partial_{\nu} v(z)=\partial_{\nu} (v(z)-\arg(z-\alpha_1))$ is bounded, we have
\begin{align*} 
	\int_{\partial B(\alpha_1,\eps)\cap \S_r}  v \partial_{\nu} v \ud s \to 0,\quad \text{as } \eps\to 0.
\end{align*}
Similarly, the integrals over other small semicircles vanish as $\eps\to 0$. Thus
\begin{equation} \label{eqn::varphiann_energy_aux2}
	\int_{C_{\mathrm{circle}}}  v \partial_{\nu} v \ud s \to 0,\quad \text{as } \eps\to 0.
\end{equation}
\item Second, let us investigate the integral over $C_{\mathrm{top}}$ and $C_{\mathrm{bottom}}$.
Let $w(z)$ be the harmonic conjugate of $v(z)$: 
\begin{align*}
	w(z)=\sum_{j=1}^{2N} (-1)^j \log |\Theta_1(r; z-\alpha_j)| - \sum_{j=1}^{2N} (-1)^j \log |\Theta_3(r; z-\beta_j)| - \sum_{j=1}^{2N} (-1)^j (\alpha_j-\beta_j) \frac{\Re (z)}{2r}.
\end{align*}
Then $w(z)+\ii v(z)$ is analytic and we have the Cauchy-Riemann equations
\begin{align*}
	\partial_x w(z)=\partial_y v(z), \quad \partial_y w(z)=-\partial_x v(z).
\end{align*}
On the boundary at bottom $C_{\mathrm{bottom}}$ ($y=0$), we have $\partial_{\nu} v = -\partial_y v = -\partial_x w$. Since $v=\pi$ on the intervals $(\alpha_{2j-1}, \alpha_{2j})$ and $0$ elsewhere, we have
\begin{equation} \label{eqn::varphiann_energy_aux3}
	\int_{C_{\mathrm{bottom}}} v \partial_{\nu} v \ud x = -\pi \sum_{j=1}^N \int_{\alpha_{2j-1}+\eps}^{\alpha_{2j}-\eps} \partial_x w \ud x = -\pi \sum_{j=1}^N \Big( w(\alpha_{2j}-\eps) - w(\alpha_{2j-1}+\eps) \Big).
\end{equation}
Similarly, for the boundary at top $C_{\mathrm{top}}$ ($y=r$), we have
\begin{equation} \label{eqn::varphiann_energy_aux4}
	\int_{C_{\mathrm{top}}} v \partial_{\nu} v \ud x = \pi \sum_{j=1}^N \int_{\beta_{2j-1}+\eps}^{\beta_{2j}-\eps} \partial_x w(x+\ii r) \ud x = \pi \sum_{j=1}^N \Big( w(\beta_{2j}+\ii r-\eps) - w(\beta_{2j-1}+\ii r+\eps) \Big).
\end{equation}
\end{itemize}
Plugging~(\ref{eqn::varphiann_energy_aux2},\ref{eqn::varphiann_energy_aux3},\ref{eqn::varphiann_energy_aux4}) into~\eqref{eqn::varphiann_energy_aux1}, we have
\begin{align} \label{eqn::varphiann_energy_aux5}
	\begin{split}
		\| v \|_{\nabla((0,2\pi)\times (0,\ii r)),\mathrm{reg}}^2 
		= &\frac{\pi}{2r}\left(\sum_{j=1}^{2N}(-1)^j (\alpha_j-\beta_j) \right)^2-4N \pi \lim_{\eps\to 0} \log \left| \Theta_1(r; \eps)/\eps \right|\\
		& - 2\pi \sum_{1\le i<j\le 2N} (-1)^{j-i} \left( \log \left| \Theta_1(r; \beta_{j}-\beta_i) \right| +\log \left| \Theta_1(r; \alpha_{j}-\alpha_i) \right| \right) \\
&+2 \pi \sum_{i=1}^{2N} \sum_{j=1}^{2N} (-1)^{j-i} \log \left| \Theta_3(r; \beta_{j}-\alpha_i) \right| 
	\end{split}
\end{align}
Plugging~(\ref{eqn::varphiann_energy_aux5},\ref{eqn::varphiann_energy_aux6}) into~\eqref{eqn::varphiann_energy_aux0}, we obtain~\eqref{eqn::varphiann_energy} as desired.
\end{proof}


\begin{lemma} \label{lem::reg_Dirichlet_annular_cro}
Fix $r>0$ and $n\ge 1$. For $\bs{\alpha}, \bs{\beta}\in\LX_n$, we write $\bs{x}=\ee^{\ii\bs{\alpha}}$ and $\bs{y}=\ee^{\ii\bs{\beta}-r}$.
Define
\begin{equation}\label{eqn::varphicro_def}
	\varphicro{n}(z)=n\pi- \sum_{j=1}^{n} \arg \Theta_1(r; -\ii \log z-\alpha_j) + \sum_{j=1}^{n} \arg \Theta_3(r; -\ii \log z-\beta_j) - \sum_{j=1}^{n} (\alpha_j-\beta_j) \frac{\log |z|}{2r}.
\end{equation}
Then we have
\begin{align} \label{eqn::varphicro_def}
	\begin{split}
	\| \varphicro{n} \|_{\nabla(\A_r),\mathrm{reg}}^2= & \frac{\pi}{2r}\left(\sum_{j=1}^{n} (\alpha_j-\beta_j) \right)^2-6n\pi \sum_{m=1}^{\infty} \log \left(  1-\ee^{-2mr} \right)  - \frac{n}{2}\pi r\\
	& - 2\pi \sum_{1\le i<j\le n} \left( \log \left| \Theta_1(r; \beta_{j}-\beta_i) \right| +\log \left| \Theta_1(r; \alpha_{j}-\alpha_i) \right| \right) \\
	& +2 \pi \sum_{i=1}^{n} \sum_{j=1}^{n} \log \left| \Theta_3(r; \beta_{j}-\alpha_i) \right| .
\end{split}
\end{align}
Moreover, its regularized Dirichlet energy is related to the partition function $\LZcro{n}^{(\ell)}$ in~\eqref{eqn::LZcroell_def}: 
\begin{align}\label{eqn::LZcro_varphicro}
\begin{split}
\LZcro{n}^{(0)}(r; \bs{\alpha}, \bs{\beta})=&\exp\left(-\frac{1}{4\pi}\| \varphicro{n} \|_{\nabla(\A_r),\mathrm{reg}}^2\right), \\
 \LZcro{n}^{(\ell)}(r; \bs{\alpha}, \bs{\beta})=&\LZcro{n}^{(0)}(r; \bs{\alpha}, \beta_{1}+2\pi\ell, \ldots, \beta_{n}+2\pi\ell). 
 \end{split}
\end{align}
\end{lemma}

\begin{proof}
Recall the covering map $q$ defined in~\eqref{eqn::annulus_strip_def}. 
From~\eqref{eqn::Theta_bc}, the function
$v:=\varphicro{n}\circ q$ is the bounded harmonic function in the infinite strip $\S_r=q^{-1}(\A_r)$ with boundary data:\footnote{where we use the convention that $\alpha_{j+n\ell}=\alpha_{j}+2\pi \ell$ and $\beta_{j+n\ell}=\beta_{j}+2\pi \ell$.}
\begin{align*}
	j \pi \text{ on }  (\alpha_{j},\alpha_{j+1})\cup(\beta_{j}+\ii r,\beta_{j+1}+\ii r) \quad j\in \Z.
\end{align*}
By the conformal covariance of regularized Dirichlet energy~\eqref{eqn::reg_Dirichlet_COV}, we have
\begin{align}\label{eqn::varphicro_def_aux0}
	\begin{split}
		\| \varphicro{n} \|_{\nabla(\A_r),\mathrm{reg}}^2= & \| v \|_{\nabla((0,2\pi)\times (0,\ii r)),\mathrm{reg}}^2 + \pi \sum_{j=1}^{n} \left( \log |q'( \alpha_j)| + \log |q'( \beta_j+\ii r)| \right) \\
		= & \| v \|_{\nabla((0,2\pi)\times (0,\ii r)),\mathrm{reg}}^2 - n\pi r.	
	\end{split}
\end{align}
\medbreak
Let us calculate the regularized Dirichlet energy of $v$ in the square $(0,2\pi)\times(0,\ii r)$.
By~\eqref{eqn::Theta_bc}, the explicit expression of $v$ is given by:
\begin{equation*}
	v(z)=n\pi- \sum_{j=1}^{n} \arg \Theta_1(r; z-\alpha_j) + \sum_{j=1}^{n} \arg \Theta_3(r; z-\beta_j) + \sum_{j=1}^{n} (\alpha_j-\beta_j) \frac{\Im(z)}{2r}.
\end{equation*}
By Green's formula, we have
\begin{equation} \label{eqn::varphicro_def_aux1}
	\| v \|_{\nabla((0,2\pi)\times (0,\ii r)),\mathrm{reg}}^2 = \lim_{\eps\to 0} \left( \int_{C_{\mathrm{top}}\cup C_{\mathrm{bottom}} \cup C_{\mathrm{vertical}} \cup C_{\mathrm{circle}}}  v \partial_{\nu} v \ud s + 2n \pi \log \eps \right),
\end{equation}
where
\begin{equation*}
	C_{\mathrm{top}}=(\ii r,2\pi+\ii r) \setminus \bigcup_{j\in \Z} (\beta_j-\eps, \beta_j+\eps) , \quad C_{\mathrm{bottom}}=(0,2\pi)\setminus \bigcup_{j\in \Z} (\alpha_j-\eps, \alpha_j+\eps) , \quad C_{\mathrm{vertical}}=(0,\ii r) \cup (2\pi,2\pi+\ii r).
\end{equation*}
and $C_{\mathrm{circle}}$ denotes the union of small semicircles of radius $\eps$ pointing into the domain around each marked point on the boundaries.

\medbreak

Let us evaluate the integral in the RHS of~\eqref{eqn::varphicro_def_aux1}.
\begin{itemize}
\item First, we have
\begin{equation} \label{eqn::varphicro_def_aux2}
	\int_{C_{\mathrm{circle}}}  v \partial_{\nu} v \ud s \to 0,\quad \text{as } \eps\to 0,
\end{equation}
similarly as~\eqref{eqn::varphiann_energy_aux2}.
\item Second, let us investigate the integral over $C_{\mathrm{top}}$ and $C_{\mathrm{bottom}}$. Let $a,b\in \Z$ satisfying
\begin{align*}
	(0,2\pi) \cap \{\alpha_j:j\in \Z\}=& \{\alpha_j:a+1\le j\le a+n\},\\
	(0,2\pi) \cap \{\beta_j:j\in \Z\}=& \{\beta_j:b+1\le j\le b+n\}.	
\end{align*}
Let $w(z)$ be the harmonic conjugate of $v(z)$: 
\begin{align*}
	w(z)=- \sum_{j=1}^{n} \log |\Theta_1(r; z-\alpha_j)| + \sum_{j=1}^{n} \log |\Theta_3(r; z-\beta_j)| + \sum_{j=1}^{n} (\alpha_j-\beta_j) \frac{\Re (z)}{2r}.
\end{align*}
Then $w(z)+\ii v(z)$ is analytic and we have the Cauchy-Riemann equations
\begin{align*}
	\partial_x w(z)=\partial_y v(z), \quad \partial_y w(z)=-\partial_x v(z).
\end{align*}
On the boundary at bottom $C_{\mathrm{bottom}}$ ($y=0$), we have $\partial_{\nu} v = -\partial_y v = -\partial_x w$. Since $v=j \pi$ on the intervals $(\alpha_{j}, \alpha_{j+1})$, we have
	\begin{equation} \label{eqn::varphicro_def_aux3}
	\begin{split}
		\int_{C_{\mathrm{bottom}}} v \partial_{\nu} v \ud s =&  -\int_{C_{\mathrm{bottom}}} v \partial_x w \ud s \\
		= & -\pi (a+n) w(2\pi)+ \pi a w(0) \\
		&- \pi \sum_{j=a+1}^{a+n} \left( (j-1) w(\alpha_j-\eps)- j w(\alpha_j+\eps) \right).		
	\end{split}
	\end{equation}
	Similarly, for the boundary at top $C_{\mathrm{top}}$ ($y=r$), we have
	\begin{equation} \label{eqn::varphicro_def_aux4}
	\begin{split}
		\int_{C_{\mathrm{top}}} v \partial_{\nu} v \ud s = & \int_{C_{\mathrm{top}}} v \partial_x w \ud s \\
		= & \pi (b+n) w(2\pi+\ii r) - \pi b w(\ii r) \\
		& + \pi \sum_{j=b+1}^{b+n} \left( (j-1) w(\beta_j-\eps+\ii r)- j w(\beta_j+\eps+\ii r) \right).
	\end{split}
	\end{equation}
\item Finally, let us investigate the integral over $C_{\mathrm{vertical}}$. On the boundary at left ($x=0$), we have $\partial_{\nu} v = -\partial_x v = \partial_y w$. On the boundary at right ($x=2\pi$), we have $\partial_{\nu} v = \partial_x v = -\partial_y w$. Thus we have
\begin{align} \label{eqn::varphicro_def_aux5}
	\int_{C_{\mathrm{vertical}}} v \partial_{\nu} v \ud s =  \int_{0}^{r} (v(\ii y)-v(2\pi+\ii y)) \partial_y w(\ii y) \ud y = - n\pi (w(\ii r)- w(0)).
\end{align}
\end{itemize}
Plugging~(\ref{eqn::varphicro_def_aux2},\ref{eqn::varphicro_def_aux3},\ref{eqn::varphicro_def_aux4},\ref{eqn::varphicro_def_aux5}) into~\eqref{eqn::varphiann_energy_aux1}, we have
\begin{align} \label{eqn::varphicro_def_aux6}
	\begin{split}
		\| v \|_{\nabla((0,2\pi)\times (0,\ii r)),\mathrm{reg}}^2 
		= &\frac{\pi}{2r}\left(\sum_{j=1}^{n} (\alpha_j-\beta_j) \right)^2-2n \pi \lim_{\eps\to 0} \log \left| \Theta_1(r; \eps)/\eps \right|\\
		& - 2\pi \sum_{1\le i<j\le n} \left( \log \left| \Theta_1(r; \beta_{j}-\beta_i) \right| +\log \left| \Theta_1(r; \alpha_{j}-\alpha_i) \right| \right) \\
		&+2 \pi \sum_{i=1}^{n} \sum_{j=1}^{n} \log \left| \Theta_3(r; \beta_{j}-\alpha_i) \right|.
	\end{split}
\end{align}
Plugging~(\ref{eqn::varphiann_energy_aux6},\ref{eqn::varphicro_def_aux6}) into~\eqref{eqn::varphicro_def_aux0}, we obtain~\eqref{eqn::varphiann_energy} as desired.
\end{proof}

\subsection{Asymptotic}
\label{subsec::asymp}
\begin{lemma} \label{lem::reg_Dirichlet_annular_expansion}
Assume the same setup as in Lemma~\ref{lem::reg_Dirichlet_annular}. 
The regularized Dirichlet energy and the partition function have the following asymptotics as $r\to\infty$: 
\begin{align} \label{eqn::varphiann_asymptotic}
\begin{split}
	\| \varphiann{n} \|_{\nabla(\A_r),\mathrm{reg}}^2=& -n\pi r -2\pi \sum_{1\le i<j\le n} (-1)^{j-i} \left( \log \sin \left( \frac{\beta_{j}-\beta_i}{2} \right) + \log \sin \left( \frac{\alpha_{j}-\alpha_i}{2} \right) \right) \\
	& + 2\pi n\log 2 +\frac{\pi}{2r}\left(\sum_{j=1}^{n}(-1)^j (\alpha_j-\beta_j) \right)^2+ O(1)\ee^{-r};
\end{split}\\
\label{eqn::LZann_asymp}
\LZann{n}(r; \bs{\alpha}, \bs{\beta})=&\frac{\LZ_n(\bs{\alpha})\LZ_n(\bs{\beta})}{\sqrt{2}^n}\exp\left(\frac{nr}{4}-\frac{1}{8r}\left(\sum_{j=1}^{n}(-1)^j(\alpha_j-\beta_j)\right)^2+O(1)\ee^{-r}\right); 
\end{align}
where $\LZ_n$ is defined in~\eqref{eqn::LZunitdisc_def} and the $O(1)$ terms have upper bounds independent of $(\bs{\alpha}, \bs{\beta})$. 
\end{lemma}
\begin{proof}
From~\eqref{eqn::Theta_asymptotic}, we have 
\begin{align*}
\sum_{1\le i<j\le n} (-1)^{j-i} \log \left| \Theta_1(r; \beta_{j}-\beta_i) \right| =& \frac{nr}{8}-\frac{n\log 2}{2}+\sum_{1\le i<j\le n} (-1)^{j-i}  \log \sin \left( \frac{\beta_{j}-\beta_i}{2} \right) + O(1)\ee^{-2r};\\
\sum_{1\le i<j\le n} (-1)^{j-i} \log \left| \Theta_1(r; \alpha_{j}-\alpha_i) \right| =& \frac{nr}{8}-\frac{n\log 2}{2}+\sum_{1\le i<j\le n} (-1)^{j-i}  \log \sin \left( \frac{\alpha_{j}-\alpha_i}{2} \right) + O(1)\ee^{-2r} ;\\
\sum_{i=1}^{n} \sum_{j=1}^{n} (-1)^{j-i} \log \left| \Theta_3(r; \beta_{j}-\alpha_i) \right|=&O(1)\ee^{-r}. 
\end{align*}
Plugging into~\eqref{eqn::varphiann_energy}, we obtain~\eqref{eqn::varphiann_asymptotic} as desired. Combining~\eqref{eqn::LZann_varphiann} and~\eqref{eqn::varphiann_asymptotic} and noting that 
\[|\ee^{\ii\theta_i}-\ee^{\ii\theta_j}|=2\sin\left(\frac{\theta_j-\theta_i}{2}\right), \quad \text{for }\theta_i<\theta_j<\theta_i+2\pi,\]
we obtain~\eqref{eqn::LZann_asymp}.
\end{proof}

\begin{lemma} \label{lem::reg_Dirichlet_annular_cro_expansion}
Assume the same setup as in Lemma~\ref{lem::reg_Dirichlet_annular_cro}. The regularized Dirichlet energy and the partition function have the following asymptotic as $r\to\infty$: 
\begin{align} \label{eqn::varphicro_asymptotic}
	\| \varphicro{n} \|_{\nabla(\A_r),\mathrm{reg}}^2=& \frac{n(n-2) \pi r}{2} -2\pi \sum_{1\le i<j\le n} \left( \log \sin \left( \frac{\beta_{j}-\beta_i}{2} \right) + \log \sin \left( \frac{\alpha_{j}-\alpha_i}{2} \right) \right) \notag\\
	& -2\pi n(n-1)\log 2 +\frac{\pi}{2r}\left(\sum_{j=1}^{n} (\alpha_j-\beta_j) \right)^2+ O(1)\ee^{-r};
\\
\LZcro{n}^{(\ell)}(r; \bs{\alpha}, \bs{\beta})=&\sqrt{2}^{n(n-1)}\LZradial{n}(\bs{\alpha})\LZradial{n}(\bs{\beta})\notag\\
&\times\exp\left(\frac{n(2-n)r}{8}-\frac{1}{8r}\left(\sum_{j=1}^{n}(\alpha_j-\beta_{j})-2n\pi\ell\right)^2+O(1)\ee^{-r}\right);\label{eqn::LZcroell_asymp}
\end{align}
where $\LZradial{n}$ is defined in~\eqref{eqn::LZunitdisc_def} and the $O(1)$ terms have upper bounds independent of $(\bs{\alpha}, \bs{\beta})$. 
\end{lemma}
\begin{proof}
From~\eqref{eqn::Theta_asymptotic}, we have 
\begin{align*}
	\sum_{1\le i<j\le n} \log \left| \Theta_1(r; \beta_{j}-\beta_i) \right| =& -\frac{n(n-1)r}{8}+\frac{n(n-1)\log 2}{2}+\sum_{1\le i<j\le n}  \log \sin \left( \frac{\beta_{j}-\beta_i}{2} \right) + O(1)\ee^{-2r};\\
	\sum_{1\le i<j\le n} \log \left| \Theta_1(r; \alpha_{j}-\alpha_i) \right| =& -\frac{n(n-1)r}{8}+\frac{n(n-1)\log 2}{2}+\sum_{1\le i<j\le n}  \log \sin \left( \frac{\alpha_{j}-\alpha_i}{2} \right) + O(1)\ee^{-2r} ;\\
	\sum_{i=1}^{n} \sum_{j=1}^{n} \log \left| \Theta_3(r; \beta_{j}-\alpha_i) \right|=&O(1)\ee^{-r}. 
\end{align*}
Plugging into~\eqref{eqn::varphicro_def}, we obtain~\eqref{eqn::varphicro_asymptotic} as desired.
Combining~\eqref{eqn::varphicro_asymptotic} and~\eqref{eqn::LZcro_varphicro}, we obtain~\eqref{eqn::LZcroell_asymp} as desired.
\end{proof}

\begin{lemma} \label{lem::LZcro_asymp}
Fix $n\ge 1$ and $\bs{\alpha}, \bs{\beta}\in\LX_n$. The partition function $\LZcro{n}$~\eqref{eqn::LZcro_def} has the following asymptotic as $r\to\infty$: 
\begin{align}\label{eqn::LZcro_asymp}
\LZcro{n}(r; \bs{\alpha}, \bs{\beta})=&\sqrt{2}^{n(n-1)}\frac{\LZradial{n}(\bs{\alpha})\LZradial{n}(\bs{\beta})}{n\sqrt{\pi/2}}\exp\left(\frac{n(2-n)r}{8}+\frac{\log r}{2}+O(1)\left(\ee^{-r}+\ee^{-2r/n^2}\right)\right);
\end{align}
where $\LZradial{n}$ is defined in~\eqref{eqn::LZunitdisc_def}; moreover, when $n=2N$ is even, we have
\begin{align} \label{eqn::LZcro_sum_asymp}
\sum_{m=1}^{N}\LZcro{n}(r; \bs{\alpha}, \beta_{2m+1}, \ldots, \beta_{2m+n})
=&\sqrt{2}^{n(n-1)} \frac{\LZradial{n}(\bs{\alpha})\LZradial{n}(\bs{\beta})}{\sqrt{2\pi}}\notag\\
&\times \exp\left(\frac{n(2-n)r}{8}+\frac{\log r}{2}+O(1)\left(\ee^{-r}+\ee^{-2r/n^2}\right)\right).
\end{align}
\end{lemma}

\begin{proof}
Plugging~\eqref{eqn::LZcroell_asymp} into~\eqref{eqn::LZcro_def}, we have
\begin{align} \label{eqn::LZcro_asymp_aux1}
	\begin{split}
		\LZcro{n}(r; \bs{\alpha}, \bs{\beta})=& \sqrt{2}^{n(n-1)}\LZradial{n}(\bs{\alpha})\LZradial{n}(\bs{\beta}) \\
		&\times\sum_{\ell\in\Z} \exp\left(\frac{n(2-n)r}{8}-\frac{1}{8r}\left(\sum_{j=1}^{n}(\alpha_j-\beta_{j})-2n\pi\ell\right)^2+O(1)\ee^{-r}\right),
	\end{split}
\end{align}
where the $O(1)$ terms have upper bounds independent of $(\bs{\alpha}, \bs{\beta}, \ell)$.
By Poisson summation formula, we have
\begin{equation}\label{eqn::LZcro_asymp_aux2}
	\sum_{\ell\in\Z} \exp\left( -\frac{1}{8r}\left(x-2n\pi\ell\right)^2 \right) = \frac{1}{n}\sqrt{\frac{2r}{\pi}} \sum_{k\in\Z} \exp\left( -\frac{2rk^2}{n^2}+ \frac{\ii k x}{n} \right) = \frac{1}{n}\sqrt{\frac{2r}{\pi}} \left(1+O(1)\ee^{-2r/n^2}\right),
\end{equation}
where $x=\sum_{j=1}^{n}(\alpha_j-\beta_{j})$ and the $O(1)$ term has upper bound independent of $x$. Plugging~\eqref{eqn::LZcro_asymp_aux2} into~\eqref{eqn::LZcro_asymp_aux1}, we obtain~\eqref{eqn::LZcro_asymp} as desired.
The same asymptotic holds for $(\beta_{2m+1}, \ldots, \beta_{2m+n})$. 
Summing~\eqref{eqn::LZcro_asymp} over $1\le m\le N$, we obtain~\eqref{eqn::LZcro_sum_asymp}.
\end{proof}

\subsection{Proof of Proposition~\ref{prop::LZcro_LZann_mono}}
\label{subsec::LZcro_LZann_mono}

To check~\eqref{eqn::LZcro_LZann_mono}, and also to facilitate our future analysis, it is more convenient to introduce boundary Poisson kernel for strips. 

\paragraph*{Boundary Poisson kernel.}
For the $2$-polygon $(\U; x_1, x_2)$, the \emph{boundary Poisson kernel} is defined as
\begin{align}\label{eqn::bPoisson_U}
	\Poisson(\U; x_1, x_2)=\frac{1}{|x_1-x_2|^2}, \qquad \text{for }x_1,x_2\in \partial\U. 
\end{align}
and for a general nice 2-polygon $(\Omega;x,y)$, we extend its definition via conformal covariance:
\begin{equation} \label{eqn::bPoisson_cov}
	\Poisson(\Omega;x,y):= |\varphi'(x)| |\varphi'(y)| \Poisson(\U;\varphi(x),\varphi(y)),
\end{equation}
where $\varphi$ is any conformal map from $\Omega$ to $\U$. Later, we will use the boundary Poisson kernel in the strip $\S_r=\{z\in\C: 0<\Im{z}<r\}$:  $\text{for } x,y\in\R,$
\begin{align}\label{eqn::bPoisson_strip}
\Poisson(\S_r; x, y+\ii r)=\frac{\pi^2}{4r^2} \left( \cosh \left(\frac{\pi(y-x)}{2r}\right) \right)^{-2}\quad\text{and}\quad \Poisson(\S_r; x, y)=\frac{\pi^2}{4r^2} \left( \sinh \left(\frac{\pi(y-x)}{2r}\right) \right)^{-2}. 
\end{align}
In addition, if $U\subset \Omega$ is a subdomain agreeing in neighborhoods of $x$ and $y$, then we have
	\begin{equation} \label{eqn::bPoisson_monotone}
		\Poisson(U;x,y)\le \Poisson(\Omega;x,y).
	\end{equation}

\begin{proof}[Proof of Proposition~\ref{prop::LZcro_LZann_mono}]
We first show the conclusion for $\ell=0$: 
\begin{align}\label{eqn::LZcro_LZann_mono_aux0}
	\LZcro{n}^{(0)}(r;\bs{\alpha},\bs{\beta})\le \LZann{n}(r;\bs{\alpha},\bs{\beta}). 
\end{align}
From~(\ref{eqn::LZann_def},\ref{eqn::LZcroell_def}), we have
\begin{align}\label{eqn::LZcro_LZann_mono_aux01}
	\frac{\LZcro{n}^{(0)}(r;\bs{\alpha},\bs{\beta})}{\LZann{n}(r;\bs{\alpha},\bs{\beta})}=&\exp\left(\frac{1}{8r}\left(\sum_{j=1}^n(-1)^j(\alpha_j-\beta_j)\right)^2-\frac{1}{8r}\left(\sum_{j=1}^n(\alpha_j-\beta_j)\right)^2\right)
	\notag\\
	&\times\prod_{\substack{1\le i<k\le n\\ 2\nmid (k-i)}}\frac{|\Theta_1(r; \alpha_k-\alpha_i)| |\Theta_1(r; \beta_k-\beta_i)|}{|\Theta_3(r; \beta_k-\alpha_i)||\Theta_3(r; \beta_i-\alpha_k)|}.
\end{align}
Combining~(\ref{eqn::bPoisson_strip},\ref{eqn::Jac_theta1},\ref{eqn::Jac_theta2}), we have, for $1\le i<k\le n$, 
\begin{align}\label{eqn::LZcro_LZann_mono_aux2}
	&\frac{|\Theta_1(r; \alpha_k-\alpha_i)| |\Theta_1(r; \beta_k-\beta_i)| }{|\Theta_3(r; \beta_k-\alpha_i)| |\Theta_3(r; \beta_i-\alpha_k)|}\notag\\=&\left(\frac{\Poisson(\S_r; \alpha_i, \beta_k+\ii r)\Poisson(\S_r; \alpha_k, \beta_i+\ii r)}{\Poisson(\S_r; \alpha_i, \alpha_k)\Poisson(\S_r; \beta_i+\ii r, \beta_k+\ii r)}\right)^{\frac{1}{2}}\ee^{\frac{(\alpha_i-\beta_i)(\alpha_k-\beta_k)}{2r}}\\
	&\times\prod_{m=1}^\infty \frac{\left(1-\ee^{\frac{-2m\pi^2}{r}-\frac{\pi (\alpha_k-\alpha_i)}{r}}\right)\left(1-\ee^{\frac{-2m\pi^2}{r}+\frac{\pi (\alpha_k-\alpha_i)}{r}}\right)\left(1-\ee^{\frac{-2m\pi^2}{r}-\frac{\pi (\beta_k-\beta_i)}{r}}\right)\left(1-\ee^{\frac{-2m\pi^2}{r}+\frac{\pi (\beta_k-\beta_i)}{r}}\right)}{\left(1+\ee^{\frac{-2m\pi^2}{r}-\frac{\pi (\beta_k-\alpha_i)}{r}}\right)\left(1+\ee^{\frac{-2m\pi^2}{r}+\frac{\pi (\beta_k-\alpha_i)}{r}}\right)\left(1+\ee^{\frac{-2m\pi^2}{r}-\frac{\pi (\beta_i-\alpha_k)}{r}}\right)\left(1+\ee^{\frac{-2m\pi^2}{r}+\frac{\pi (\beta_i-\alpha_k)}{r}}\right)}.\notag
\end{align}
Plugging~\eqref{eqn::LZcro_LZann_mono_aux2} into~\eqref{eqn::LZcro_LZann_mono_aux01}, we obtain 
\begin{align} \label{eqn::LZcro_LZann_mono_aux3}
	&\frac{\LZcro{n}^{(0)}(r;\bs{\alpha},\bs{\beta})}{\LZann{n}(r;\bs{\alpha},\bs{\beta})} \notag \\
	=&\prod_{\substack{1\le i<k\le n\\ 2\nmid (k-i)\\ m\ge 1}}
	\frac{\left(1-\ee^{-\frac{2m\pi^2}{r}-\frac{\pi(\alpha_k-\alpha_i)}{r}}\right) \left(1-\ee^{-\frac{2m\pi^2}{r}+\frac{\pi(\alpha_k-\alpha_i)}{r}}\right)}
	{\left(1+\ee^{-\frac{2m\pi^2}{r}-\frac{\pi(\beta_k-\alpha_i)}{r}}\right)
	\left(1+\ee^{-\frac{2m\pi^2}{r}+\frac{\pi(\beta_k-\alpha_i)}{r}}\right)}
	\frac{\left(1-\ee^{-\frac{2m\pi^2}{r}-\frac{\pi(\beta_k-\beta_i)}{r}}\right)
	\left(1-\ee^{-\frac{2m\pi^2}{r}+\frac{\pi(\beta_k-\beta_i)}{r}}\right)}
	{\left(1+\ee^{-\frac{2m\pi^2}{r}-\frac{\pi(\beta_i-\alpha_k)}{r}}\right)
	\left(1+\ee^{-\frac{2m\pi^2}{r}+\frac{\pi(\beta_i-\alpha_k)}{r}}\right)}\notag \\
	&\times \prod_{\substack{1\le i<k\le n\\ 2\nmid (k-i)}}
	\left(\frac{\Poisson(\S_r;\alpha_i,\beta_k+\ii r)\Poisson(\S_r;\alpha_k,\beta_i+\ii r)}
	{\Poisson(\S_r;\alpha_i,\alpha_k)\Poisson(\S_r;\beta_i+\ii r,\beta_k+\ii r)}\right)^{\frac12}.
\end{align}
Since the second line of~\eqref{eqn::LZcro_LZann_mono_aux3} is smaller than one (the numerator is smaller than one and the denominator is bigger that one), and the third line of~\eqref{eqn::LZcro_LZann_mono_aux3} is also smaller than one due to
\begin{equation} \label{eqn::LZcro_LZann_mono_aux4}
	\left( \frac{\Poisson(\S_r;\alpha_i,\beta_k+\ii r)\Poisson(\S_r;\alpha_k,\beta_i+\ii r)}
	{\Poisson(\S_r;\alpha_i,\alpha_k)\Poisson(\S_r;\beta_i+\ii r,\beta_k+\ii r)}\right)^{\frac{1}{2}}=\frac{\left( 1-\ee^{-\frac{\pi (\alpha_k-\alpha_i)}{r}} \right) \left( 1-\ee^{-\frac{\pi (\beta_k-\beta_i)}{r}} \right)}{\left( 1+\ee^{-\frac{\pi (\beta_k-\alpha_i)}{r}} \right) \left( 1+\ee^{-\frac{\pi (\alpha_k-\beta_i)}{r}} \right)} \le 1,
\end{equation}
we obtain~\eqref{eqn::LZcro_LZann_mono_aux0}.
\medbreak
Next, for general $\ell\in\Z$, from~\eqref{eqn::LZcro_LZann_mono_aux0} and~\eqref{eqn::Theta_period}, we have 
\begin{align*}
	\LZcro{n}^{(\ell)}(r; \bs{\alpha}, \bs{\beta})
	=\LZcro{n}^{(0)}(r; \bs{\alpha}, \bs{\beta}+2\pi\ell)
	\le \LZann{n}(r; \bs{\alpha}, \bs{\beta}+2\pi\ell)
	= \LZann{n}(r; \bs{\alpha}, \bs{\beta}).
\end{align*}
This completes the proof for~\eqref{eqn::LZcro_LZann_mono}.
\end{proof}

\section{Multi-annulus SLE: general framework}
\label{sec::multiannulus_SLE}
In this section, we will relate annulus BPZ equations with multi-time martingale and provide a general framework on defining and analyzing multi-annulus SLE process. 
Fix $\kappa>0$ and we use the following CFT/SLE parameters indexed by $\kappa > 0$:
\begin{align}\label{eqn::universal_parameters}
	\mathfrak{b} := \frac{6-\kappa}{2\kappa} ,\qquad 
	\textnormal{and} \qquad 
	\tilde{\mathfrak{b}} := \frac{(6-\kappa)(\kappa-2)}{8\kappa} , 
	\qquad \textnormal{and} \qquad 
	\mathfrak{c} := \frac{(6-\kappa)(3\kappa-8)}{2\kappa}.
\end{align}

\paragraph*{Annulus BPZ equations.}
Fix $\kappa>0$, and $n\ge 1$, and $\bs{\alpha}, \bs{\beta} \in \LX_{n}$. Let $\LZ=\LZ(r; \bs{\alpha}, \bs{\beta}):\R_{>0} \times \LX_{n} \times\LX_{n}\to \R$ be a positive $C^{1,2}$ differentiable function and let $F(r):\R_{>0}\to \R$. 
Recall that $H_1, H_3$ are derivatives of rescaled Jacobi theta functions defined in~\eqref{eqn::H1H3_def}. 
The annulus BPZ equations for the coordinates $\alpha_j$ and $\beta_j$ are given by: for all $j\in\{1, \ldots, n\}$, 
\begin{align} \label{eqn::annulus_BPZ_alpha} 
\begin{split}
(\partial_r-F(r))\LZ=\LD_{\alpha_j}^{(n)}\LZ,\quad
\text{where }
\LD_{\alpha_j}^{(n)} = \frac{\kappa}{2} \partial_{\alpha_j}^2 + &\sum_{\ell\neq j} \left( H_1(r;\alpha_\ell-\alpha_j) \partial_{\alpha_\ell}  + \mathfrak{b}\partial_z H_1(r;\alpha_\ell-\alpha_j) \right) \\
+&\sum_{\ell=1}^n \left( H_3(r;\beta_\ell-\alpha_j) \partial_{\beta_\ell} + \mathfrak{b} \partial_z H_3(r;\beta_\ell-\alpha_j) \right) ;
\end{split}
\end{align}
\begin{align} \label{eqn::annulus_BPZ_beta}
\begin{split}		(\partial_r-F(r))\LZ=\LD_{\beta_j}^{(n)}\LZ,\quad\text{where }\LD_{\beta_j}^{(n)}= \frac{\kappa}{2} \partial_{\beta_j}^2 +   &\sum_{\ell\neq j} \left( H_1(r;\beta_\ell-\beta_j) \partial_{\beta_\ell} + \mathfrak{b} \partial_z H_1(r;\beta_\ell-\beta_j) \right) \\
		 +&  \sum_{\ell=1}^n \left( H_3(r;\alpha_\ell-\beta_j) \partial_{\alpha_\ell}  + \mathfrak{b} \partial_z H_3(r;\alpha_\ell-\beta_j) \right). 
		 \end{split}
\end{align}
The function $F(r)$ represents a degree of freedom in the system. For the case $n=1$, explicit solutions for various $\kappa$ values were investigated in~\cite[Section~8]{ZhanReversibilityWholeplaneSLE}. The multi-annulus SLE partition function from~\cite{ZhanRestrictionAnnulusSLE, JahangoshahiLawlerMultiplepathsSLE} satisfies annulus BPZ equations and has the following properties. 

\begin{proposition}\label{prop::multiannulu_pf_construction}
Fix $\kappa\in (0,4]$ and $n\ge 1$ and $\ell\in\Z$. 
Fix $r>0$ and $\bs{\alpha}, \bs{\beta}\in\LX_n$. For $1\le j\le n$, let $\gamma^j$ be annulus $\SLE_{\kappa}(\LFcro{1}^{(\kappa;\ell)})$ in $(\A_r;\ee^{\ii \alpha_j};\ee^{\ii \beta_j-r})$. 
Let $\Pind{n}^{(\kappa;\ell)}=\Pind{n}^{(\kappa;\ell)}(\A_r; \ee^{\ii\bs{\alpha}}, \ee^{\ii\bs{\beta}})$ be the probability measure on $\bs{\gamma}=(\gamma^{1}, \ldots, \gamma^{n})$ under which the curves are independent.
Define $\LFcro{n}^{(\kappa; \ell)}(r; \bs{\alpha}, \bs{\beta})$ as in~\eqref{eqn::multiannulus_pf_def}: 
\begin{align*}
\LFcro{n}^{(\kappa; \ell)}(r; \bs{\alpha}, \bs{\beta})=\prod_{j=1}^n\LFcro{1}^{(\kappa;\ell)}(r; \alpha_j, \beta_j)\times \Eind{n}^{(\kappa;\ell)}\left[\one_{\LE_{\emptyset}(\bs{\gamma})}\exp\left(\frac{\mathfrak{c}}{2}\blm(\A_r; \gamma^1, \ldots, \gamma^n)\right)\right],
\end{align*}
where $\LE_{\emptyset}(\bs{\gamma}) = \{ \gamma^{j} \cap \gamma^{i}=\emptyset, \, \forall i\neq j\}$ is the event that different curves are disjoint, and $\blm$ is the Brownian loop measure defined in~\eqref{eqn::blm_def}. 
\begin{itemize}
\item The partition function $\LFcro{n}^{(\kappa; \ell)}$ is smooth and it satisfies annulus BPZ equations~\eqref{eqn::annulus_BPZ_alpha}-\eqref{eqn::annulus_BPZ_beta} with
\begin{equation} \label{eqn::F(r)inBPZ}
F(r)=	F_n(r):=\frac{6\tilde{\mathfrak{b}}-\mathfrak{b}}{r} + 6 \tilde{\mathfrak{b}} \LE(r) + \mathfrak{b} n. 
\end{equation} 
\item Define $\LFcro{n}^{(\kappa)}(r; \bs{\alpha}, \bs{\beta})$ as in~\eqref{eqn::multiannulus_withoutspiral_def}:
\begin{align*}
\LFcro{n}^{(\kappa)}(r; \bs{\alpha}, \bs{\beta})=\sum_{\ell\in\Z}\LFcro{n}^{(\kappa;\ell)}(r; \bs{\alpha}, \bs{\beta}).
\end{align*}
Then there exists a constant $C=C(n)\in (0,\infty)$ independent of $(r; \bs{\alpha}, \bs{\beta})$ such that
\begin{align}\label{eqn::multiannulus_pf_upperbound}
\LFcro{n}^{(\kappa)}(r; \bs{\alpha}, \bs{\beta})\le C\exp\left(\left(n\mathfrak{b}-\frac{(8-\kappa)}{\kappa}\right)r+\left(\frac{3\kappa}{2}-2n\right)\mathfrak{b}\log r\right). 
\end{align}
\end{itemize}
\end{proposition}

We will show in Proposition~\ref{prop::multitime_mart_annulus} that any solution to the annulus BPZ equations gives a mutli-time martingale in the annulus setup. As $\LFcro{n}^{(\kappa; \ell)}$ satisfies annulus BPZ equations, it provides us with a multi-time martingale due to Proposition~\ref{prop::multitime_mart_annulus}. 
We define $n$-annulus $\SLE_{\kappa}(\LFcro{n}^{(\kappa;\ell)})$ in Definition~\ref{def::multiannulus_SLE}, using multi-time martingale. 
We prefer to define multi-annulus SLE via multi-time martingale, instead of the construction in~\cite{ZhanRestrictionAnnulusSLE} and~\cite{JahangoshahiLawlerMultiplepathsSLE}, for the following reason. The construction in~\cite{ZhanRestrictionAnnulusSLE} and~\cite{JahangoshahiLawlerMultiplepathsSLE} gives one candidate for multi-annulus SLE, but this is not the only reasonable one, see Remark~\ref{rem::annulus_otherpatterns}. In the case of $\kappa=4$, we have at least two reasonable candidates for multi-annulus $\SLE_4$: the two given by the two partition functions $\LZcro{n}^{(\ell)}$ and $\LZann{n}$. We would like to provide a unified framework to treat all reasonable multi-annulus SLE with commutation relation. With this purpose, we relate annulus BPZ equations with multi-time martingale and define multi-annulus SLE using the martingale.

The multi-time martingale provides us with a powerful tool in analyzing the multiple SLE process, see e.g. Theorem~\ref{thm::GFF_levellines_annulus}. 
But it is not clear that the process defined in this way coincides with previous construction of multi-annulus SLE~\cite[Section~2.3]{JahangoshahiLawlerMultiplepathsSLE} and~\cite[Section~8.2]{ZhanRestrictionAnnulusSLE}. 
We show in Proposition~\ref{prop::ac_general} that multi-annulus SLE in Definition~\ref{def::multiannulus_SLE} is absolutely continuous with respect to independent annulus SLEs, which implies that our construction coincides with previous construction. 
\begin{proposition}\label{prop::ac_general}
Assume the same setup as in Proposition~\ref{prop::multiannulu_pf_construction}. The law of $n$-annulus $\SLE_{\kappa}(\LFcro{n}^{(\kappa;\ell)})$ in $(\A_r; \ee^{\ii\bs{\alpha}}, \ee^{\ii\bs{\beta}-r})$ is the same as $\Pind{n}^{(\kappa;\ell)}$ weighted by the Radon-Nikodym derivative 
\begin{align}\label{eqn::ac_RN}
\frac{\prod_{j=1}^n\LFcro{1}^{(\kappa;\ell)}(r; \alpha_j, \beta_j)}{\LFcro{n}^{(\kappa;\ell)}(r; \bs{\alpha}, \bs{\beta})}\one_{\LE_{\emptyset}(\bs{\gamma})}\exp\left(\frac{\mathfrak{c}}{2}\blm(\A_r; \gamma^1, \ldots, \gamma^n)\right), 
\end{align}
where $\LE_{\emptyset}(\bs{\gamma}) = \{ \gamma^{j} \cap \gamma^{i}=\emptyset, \, \forall i\neq j\}$ is the event that different curves are disjoint, and $\blm$ is the Brownian loop measure defined in~\eqref{eqn::blm_def}.
\end{proposition}

This section is organized as follows. In Section~\ref{subsec::polyon_pre}, we give preliminaries on rainbow SLE and its partition functions. 
In Section~\ref{subsec::multiannulus_pre}, we give preliminaries on annulus Loewner chain  and summarize known results for a single annulus SLE. 
In Section~\ref{subsec::multitime_mart}, we introduce multi-time martingale. 
In Sections~\ref{subsec::earlierconstruction} and~\ref{subsec::multiannulus_pf_proof}, we complete the proof of Propositions~\ref{prop::multiannulu_pf_construction} and~\ref{prop::ac_general}. The first item in Proposition~\ref{prop::multiannulu_pf_construction} was argued in~\cite[Proposition~4 and Proof of Theorem~1]{JahangoshahiLawlerMultiplepathsSLE}, see also Lemma~\ref{lem::multiannulus_pf_BPZ}. 
The proof of these two propositions relies heavily on earlier result for a single annulus SLE from~\cite{ZhanRestrictionAnnulusSLE} and~\cite{JahangoshahiLawlerMultiplepathsSLE} and earlier result for multi-chordal SLE partition functions~\cite{PeltolaWuGlobalMultipleSLEs}. 

\subsection{Preliminaries: partition functions in polygons}
\label{subsec::polyon_pre}

\paragraph*{Brownian loop measure.}
Brownian loop measure $\blm^\mathrm{loop}$ is a $\sigma$-finite measure on planar unrooted Brownian loops
--- see~\cite{LawlerWernerBrownianLoopsoup} for its definition and properties. 
While the total mass of $\blm^\mathrm{loop}$ is infinite, the mass on macroscopic loops is finite: 
if $\Omega$ is a domain and $K_1, K_2 \subset \overline{\Omega}$ are two disjoint compact subsets, 
then the total mass $\blm(\Omega; K_1, K_2)$ of Brownian loops that stay in $\Omega$ and intersect both $K_1$ and $K_2$ is finite.
In general, for $n\ge 2$ disjoint compact subsets $K_1, \ldots, K_n$ of $\overline{\Omega}$, we denote
\begin{align}\label{eqn::blm_def}
	\blm(\Omega; K_1, \ldots, K_n) := \sum_{j=2}^n \blm^{\mathrm{loop}} \big[ \ell\subset\Omega: \ell\cap K_i\neq \emptyset\textnormal{ for at least }j \textnormal{ of the }i\in\{1, \ldots, n\} \big].
\end{align}
See~\cite{LawlerPartitionFunctionsSLE,PeltolaWangSLELDP} for more properties and~\cite{DubedatEulerIntegralsCommutingSLEs, DubedatCommutationSLE, KozdronLawlerMultipleSLEs, PeltolaWuGlobalMultipleSLEs} 
for alternative forms for~\eqref{eqn::blm_def}.

\paragraph*{Polygon.} We say that $(\Omega; \bs{x})=(\Omega; x_1, \ldots, x_n)$ 
is a (topological) $n$-\emph{polygon} if $\Omega\subsetneq\C$ is simply connected, $\partial\Omega$ is locally connected, and $x_1, x_2, \ldots, x_n \in \partial\Omega$ are distinct points lying counterclockwise along the boundary. 
We say that $(\Omega; \bs{x})=(\Omega; x_1, \ldots, x_n)$ 
is a \emph{nice} polygon if we assume further that the marked boundary points $x_1, x_2, \ldots, x_n$ lie on $C^{1+\eps}$-boundary segments, for some $\eps>0$, so that derivatives of conformal maps on $\Omega$ are defined there.

\paragraph*{Chordal SLE.}
For a $2$-polygon $(\Omega; x_1, x_2)$, denote by $\chamber(\Omega; x_1, x_2)$ the set of continuous simple unparameterized curves in $\Omega$ connecting $x_1$ and $x_2$ such that they only touch the boundary $\partial\Omega$ in $\{x_1,x_2\}$. Fix $\kappa\in (0,4]$. Chordal $\SLE_{\kappa}$ in $(\Omega; x_1, x_2)$ is a probability measure on $\chamber(\Omega; x_1, x_2)$ that satisfies conformal invariance and domain Markov property. Its definition is usually given in the upper-half plane via chordal Loewner chain.
As we mainly focus on the annulus setting in this article, we do not plan to introduce chordal Loewner chain. Readers may check~\cite{WernerRandomPlanarcurves}.
We denote by $\Ptwo^{(\kappa)}(\Omega; x_1, x_2)$ the law of chordal $\SLE_{\kappa}$ in $(\Omega; x_1, x_2)$. Its partition function is given by
\begin{align*}
\LZtwo(\Omega; x_1, x_2):=\Poisson(\Omega; x_1, x_2)^{\mathfrak{b}}, 
\end{align*}
where $\Poisson(\Omega; x_1, x_2)$ is boundary Poisson kernel~\eqref{eqn::bPoisson_U}-\eqref{eqn::bPoisson_cov}. 

\paragraph*{Rainbow SLE.}
Fix $\kappa\in (0,4]$ and a $2N$-polygon $(\Omega; x_1, \ldots, x_{2N})$, denote by $\chamber_{\rainbow}(\Omega; x_1, \ldots, x_{2N})$ the set of collections of $N$ curves $(\gamma^1, \ldots, \gamma^N)$ where $\gamma^j$ is a continuous simple curve in $\Omega$ connecting $x_j$ and $x_{2N+1-j}$ for $1\le j\le N$ and $\gamma^1, \ldots, \gamma^N$ are disjoint. 
We call a probability measure on $(\gamma^1, \ldots, \gamma^N)\in \chamber_{\rainbow}(\Omega; x_1, \ldots, x_{2N})$ an $N$-rainbow $\SLE_{\kappa}$ if, for each $j\in\{1, \ldots, N\}$, the conditional law of the curve $\gamma^j$ given the other curves $(\gamma^1, \ldots, \gamma^{j-1}, \gamma^{j+1}, \ldots, \gamma^N)$ is the chordal $\SLE_{\kappa}$ connecting $x_j$ and $x_{2N+1-j}$ in the component of the domain $\Omega\setminus\cup_{i\neq j}\gamma^i$ that has $x_j$ and $x_{2N+1-j}$ on the boundary. 
For $\kappa\in (0,4]$, the existence and uniqueness of $N$-rainbow $\SLE_{\kappa}$ is proved in~\cite{PeltolaWuGlobalMultipleSLEs, BeffaraPeltolaWuUniqueness}. Moreover, the law of rainbow $\SLE_{\kappa}$ is encoded by the following rainbow partition functions. 

\paragraph*{Rainbow partition functions.}
In the upper half-plane $\HH$, rainbow partition functions are functions 
	\begin{align*}
		\LZrainbow{N}^{(\kappa)}(\HH; \cdot) \colon \{ \bs{x} = (x_1,\ldots,x_{2N})\in\R^{2N} \,|\, x_1<\cdots<x_{2N} \} \to (0,\infty) ,
	\end{align*}
	which are defined recursively 
	via the following four properties, motivated by CFT (see~\cite{FengLiuPeltolaWu2024} and references therein):
	\begin{itemize}
		\item[(PDE)] {\bf\emph{Chordal BPZ equations}}\textnormal{:} 
		\begin{align*}
			\bigg(
			\frac{\kappa}{2} \partial_j^2
			+  \underset{1\leq i\neq j \leq 2N}{\sum} \,  \bigg( \frac{2\partial_{i}}{x_i-x_j}
			- \frac{2\mathfrak{b}}{(x_i-x_j)^{2}} \bigg) \bigg)
			\LZrainbow{N}^{(\kappa)}(\HH; \bs{x}) =  0 , \qquad \textnormal{for all }j\in \{1,\ldots,2N\} .
		\end{align*}
		
		\item[(COV)] {\bf\emph{M\"{o}bius covariance}}\textnormal{:} 
		for all M\"obius maps $\varphi \colon \HH \to \HH$ such that $\varphi(x_1) < \cdots < \varphi(x_{2N})$, we have
		\begin{align*}
			\LZrainbow{N}^{(\kappa)}(\HH; x_1,\ldots,x_{2N}) = 
			\Big( \prod_{j=1}^{2N} \varphi'(x_j) \Big)^{\mathfrak{b}} 
			\, \LZrainbow{N}^{(\kappa)}(\HH; \varphi(x_1),\ldots,\varphi(x_{2N})).
		\end{align*}
		
		\item[(ASY)] {\bf\emph{Asymptotics}}\textnormal{:} 
		with $\LZrainbow{0}^{(\kappa)} \equiv 1$, the collection $\{\LZrainbow{N}^{(\kappa)}\}_{N\ge 0} $ satisfies the following recursive asymptotics property. Fix $N\ge 1$ and $j \in \{1,2, \ldots, 2N-1 \}$. 
		Then, we have
		\begin{align*}
			\lim_{x_j,x_{j+1}\to\xi} \frac{\LZrainbow{N}^{(\kappa)}(\HH; x_1,\ldots, x_{2N})}{ (x_{j+1}-x_j)^{-2 \mathfrak{b} } }
			= 
			\begin{cases}
				\LZrainbow{N-1}^{(\kappa)}(\HH; x_1, \ldots, x_{j-1}, x_{j+2}, \ldots, x_{2N}), 
				& \quad j = N, \\
				0 ,
				& \quad j \neq N ,
			\end{cases}
		\end{align*}
		where $\xi \in (x_{j-1}, x_{j+2})$ (with the convention that $x_0 = -\infty$ and $x_{2N+1} = +\infty$). 
		
		\item[(PLB)] {\bf\emph{Power-law bound}}\textnormal{:} 
		there exist constants $C>0$ and $r>0$ such that for all $N \geq 1$, we have
		\begin{align*}
			\LZrainbow{N}^{(\kappa)}(\HH; x_1,\ldots,x_{2N}) \le  \; & C\prod_{1\le i<j\le 2N}(x_j-x_i)^{\nu^{ij}(r)}, 
			\qquad \textnormal{for all } x_1<\cdots<x_{2N} ,
			\\
			\nonumber
			\textnormal{where } \quad
			\nu^{ij}(r) := \; &
			\begin{cases}
				r , & \textnormal{if }|x_j-x_i|>1,\\
				-r , & \textnormal{if }|x_j-x_i|\le 1.
			\end{cases}
		\end{align*}
	\end{itemize}
As the number of boundary points is usually clear from the context, we also denote $\LZrainbow{N}^{(\kappa)}(\HH; x_1,\ldots,x_{2N})$ by $\LZ_{\rainbow}^{(\kappa)}(\HH; x_1,\ldots,x_{2N})$.
	We extend the rainbow partition function to general nice polygons by conformal covariance.
	For a nice $2N$-polygon $(\Omega; x_1,\ldots,x_{2N})$, letting $\varphi:\Omega\to\HH$ be any conformal map such that $\varphi(x_1)<\cdots<\varphi(x_{2N})$, we define
	\begin{align}\label{eqn::rainbow_pf_polygon_def}
		\LZ_{\rainbow}^{(\kappa)}(\Omega; x_1,\ldots,x_{2N})
		:= \prod_{j=1}^{2N} |\varphi'(x_j)|^{\mathfrak{b}} \times
		\LZ_{\rainbow}^{(\kappa)}(\HH; \varphi(x_1),\ldots,\varphi(x_{2N})).
	\end{align}
We have the following refined power-law bound for $\LZ_{\rainbow}^{(\kappa)}$ by~\cite[Eq.~(1.4) and Theorem~1.5]{PeltolaWuGlobalMultipleSLEs}:
	\begin{align}\label{eqn::rainbow_pf_upperbound}
		\LZ_{\rainbow}^{(\kappa)}(\Omega; x_1, \ldots, x_{2N}) \le \prod_{j=1}^{N} \Poisson(\Omega;x_{j},x_{2N+1-j})^{\mathfrak{b}}. 
\end{align}

Rainbow partition functions have the following equivalent description, which is a key ingredient for the proof of Proposition~\ref{prop::ac_general}. 
\begin{lemma}[{\cite[Eq.~(3.7, 3.8)]{PeltolaWuGlobalMultipleSLEs}}]
\label{lem::rainbow_pf_blm_rep}
Fix $\kappa\in (0,4]$ and $2N$-polygon $(\Omega; x_1, \ldots, x_{2N})$. 
For $1\le j\le N$, let $\gamma^j$ be chordal $\SLE_{\kappa}$ in $(\Omega; x_j, x_{2N+1-j})$. Let $\PP_N^{(\kappa)}$ be the probability measure on $\bs{\gamma}=(\gamma^1, \ldots, \gamma^N)$ under which the curves are independent. Then the law of rainbow $\SLE_{\kappa}$ in $(\Omega; x_1, \ldots, x_{2N})$ has the same law as $\PP_N^{(\kappa)}$ weighted by the Radon-Nikodym derivative 
\begin{align}\label{eqn::rainbow_blm_rep}
\frac{\prod_{j=1}^{N}\LZtwo(\Omega;x_j,x_{2N+1-j})}{\LZ_{\rainbow}^{(\kappa)}(\Omega;x_1,\ldots,x_{2N})}\one_{\LE_{\emptyset}(\bs{\gamma})}
		\exp\left(\frac{\mathfrak{c}}{2}\blm(\Omega;\gamma^1,\ldots,\gamma^N)\right), 
\end{align}
where $\LE_{\emptyset}(\bs{\gamma}) = \{ \gamma^{j} \cap \gamma^{i}=\emptyset, \, \forall i\neq j\}$ is the event that different curves are disjoint, and $\blm$ is the Brownian loop measure defined in~\eqref{eqn::blm_def}.
In particular, we have 
\begin{align}\label{eqn::rainbow_pf_blm_rep}
	\begin{split}
		\LZ_{\rainbow}^{(\kappa)}(\Omega;x_1,\ldots,x_{2N})
		= \prod_{j=1}^{N}\LZtwo(\Omega;x_j,x_{2N+1-j}) \times
		\E_N^{(\kappa)} \left[
		\one_{\LE_{\emptyset}(\bs{\gamma})}
		\exp\left(\frac{\mathfrak{c}}{2}\blm(\Omega;\gamma^1,\ldots,\gamma^N)\right)
		\right].
	\end{split}
\end{align}
\end{lemma}

\subsection{Preliminaries: annulus SLE}
\label{subsec::multiannulus_pre}

\paragraph*{Schwartz kernel.}
For a modulus parameter $r > 0$, the Schwartz kernel $S_r(z)$ is defined as
\begin{align*}
	S_r(z) := \lim_{N \to \infty} \sum_{k=-N}^N \frac{\ee^{2kr} + z}{\ee^{2kr} - z}.
\end{align*}
The derivatives of the rescaled Jacobi theta functions defined in~\eqref{eqn::JacobiTheta_productexpansion} is related to the Schwartz kernel: 
\begin{align} \label{eqn::H1H3}
	H_1(r;z) := 2 \partial_z \log \Theta_1(r;z) = -\ii S_r(\ee^{\ii z}), \quad
	H_3(r;z) := 2 \partial_z \log \Theta_3(r;z)= -\ii S_r(\ee^{\ii z - r}) + \ii.
\end{align}
\begin{lemma}\label{lem::H1H3}
We have the asymptotic as $z\to 0$: 
\begin{equation} \label{eqn::Laurrent_expand_LE}
	H_1(t;z) = \frac{2}{z} + \LE(t) z + O(z^3), \quad \text{as } z \to 0, 
\end{equation}
where $\LE$ is defined in~\eqref{eqn::Def_LE}. 
\end{lemma}
\begin{proof}
For the Laurent expansion of $H_1(r; \cdot)$, differentiating the product formula for $\Theta_1$ gives
\[
H_1(r;z)=\cot(z/2)+2\sum_{m=1}^{\infty}\left(
-\frac{\ii \ee^{-2mr}\ee^{\ii z}}{1-\ee^{-2mr}\ee^{\ii z}}
+\frac{\ii \ee^{-2mr}\ee^{-\ii z}}{1-\ee^{-2mr}\ee^{-\ii z}}
\right).
\]
Using $\cot(z/2)=2/z-z/6+O(z^3)$ and expanding the summands at $z=0$, we obtain
\[
H_1(r;z)=\frac{2}{z}+\left(-\frac{1}{6}+\sum_{m=1}^{\infty}\frac{4\ee^{-2mr}}{(1-\ee^{-2mr})^2}\right)z+O(z^3).
\]
Since $4\ee^{-2mr}/(1-\ee^{-2mr})^2=1/\sinh^2(mr)$, the coefficient of $z$ is precisely $\LE(r)$.
\end{proof}

\paragraph*{Annulus Loewner chain.}
Every doubly connected domain $D$ with non-degenerate boundary is conformally equivalent to a unique standard annulus $\A_r$ for some $r > 0$. We call $\Mod(D) := r$ the modulus of $D$. 
An $\A_r$-hull is a relatively closed subset $K \subsetneq \A_r$ such that $\A_r \setminus K$ is a doubly connected domain that has the inner boundary of $\A_r$ as one of its boundary components. The mapping-out function from $\A_r \setminus K$ to $\A_{s}$ with $s=\Mod(\A_r\setminus K)$, denoted by $\mathfrak{g}_K$, is unique up to rotation.

The annulus Loewner evolution involves the Schwartz kernel for annuli. 
Fix $r > 0$ and $T \in (0, r)$. Let $\zeta: [0, T) \to \R$ be a continuous function. 
An annulus Loewner chain driven by $\zeta$ is a family of $\A_r$-hulls $(K_t)_{t \in [0, T)}$ such that the mapping-out function $\mathfrak{g}_t$ from $\A_r \setminus K_t$ onto $\A_{r-t}$ satisfies the annulus Loewner equation (see~\cite{ZhanSLEannulus}):
\begin{align}\label{eqn::annulus_Loewner_equation}
	\partial_t \mathfrak{g}_t(z) = \mathfrak{g}_t(z) S_{r-t}( \mathfrak{g}_t(z)/\ee^{\ii \zeta_t} ), \qquad \mathfrak{g}_0(z) = z, \qquad z \in \overline{\A}_r.
\end{align}
In the universal cover $\S_r$, let $\mathfrak{\covmap}_t$ be the lifting of $\mathfrak{g}_t$ such that $\mathfrak{\covmap}_t \colon \S_r \setminus q^{-1}(K_t) \to \S_{r-t}$. The covering map $\mathfrak{\covmap}_t$ satisfies
\begin{align}\label{eqn::annulus_Loewner_equation_cov}
	\partial_t \mathfrak{\covmap}_t(z) = H_1(r-t;\mathfrak{\covmap}_t(z) - \zeta_t), \qquad \mathfrak{\covmap}_0(z) = z.
\end{align}
Furthermore, the evolution of the points on the inner boundary (i.e., $z = x + \ii r$) is governed by $H_3$. Specifically, for $z \in \R + \ii r$, the real part of the map satisfies 
\begin{equation}\label{eqn::annulus_Loewner_equation_cov_H3}
	\partial_t \Re(\mathfrak{\covmap}_t(z)) = H_3(r-t; \Re(\mathfrak{\covmap}_t(z)) - \zeta_t).
\end{equation}

For each $z \in \overline{\A}_r$, the flow $t \mapsto \mathfrak{g}_t(z)$ is well-defined up to the swallowing time
\[
\sigma_z := \sup\left\{ t \in [0, T) : \inf_{s \in [0, t]} |\mathfrak{g}_s(z) - \exp(\ii \zeta_s)| > 0 \right\}.
\]
The hulls are determined by $K_t = \{z \in \A_r : \sigma_z \le t\}$. The family $(K_t)_{t \in [0, T)}$ is parameterized by its modulus: $\Mod(\A_r \setminus K_t) = r - t$. The chain $(K_t)$ satisfies the local growth property if for $0 \le s < t < T$, the diameter of $g_s(K_t \setminus K_s)$ tends to $0$ as $t \downarrow s$ uniformly over $s$. 
Conversely, any increasing family of $\A_r$-hulls $(K_t)_{t \in [0, T)}$ satisfying the local growth property can be represented as an annulus Loewner chain driven by some continuous $\zeta$. For a simple curve $\gamma$ in $\A_r$ starting at $\ee^{\ii \theta}$ on the outer boundary, the driving function is given by $\zeta_t = \arg \mathfrak{g}_t(\gamma_t)$.

\paragraph*{Annulus SLE process.} 
Fix $\alpha,\beta\in \R$.
Let $\LZ=\LZ(r; \alpha,\beta): \R_{>0}\times \R^2\to \R$ be a positive $C^{1,2}$ differentiable function, the annulus $\SLE_{\kappa}(\LZ)$ in $(\A_r;\ee^{\ii\alpha};\ee^{\ii\beta-r})$ is the  annulus Loewner chain $(K_t)_{0\le t<r}$ driven by a continuous function $\zeta:[0,r)\to \R$ satisfying the SDE system
\begin{equation}\label{eqn::annuluSLEkappa}
	\begin{cases}
		\displaystyle \ud \zeta_t = \sqrt{\kappa} \ud B_t + \kappa \partial_{\alpha} \log \LZ(r-t; \zeta_t, V_t) \ud t, \qquad \zeta_0=\alpha, \\
		\displaystyle \ud V_t = H_3(r-t; V_t - \zeta_t) \ud t, \qquad V_0=\beta,
	\end{cases}
\end{equation}
where $\{B_t\}_{t\ge 0}$ is a standard Brownian motion. 
The construction of the partition function for annulus SLE appears in both~\cite[Section~5]{lawlerDefiningSLEwithBrownianLoop} and~\cite[Section~6]{ZhanReversibilityWholeplaneSLE}, albeit with different emphases. We briefly summarize the two constructions below.
\medbreak
Lawler~\cite{lawlerDefiningSLEwithBrownianLoop} constructs the partition function for annulus SLE using the boundary perturbation property.
Suppose $\Omega\subset\A_r$ is simply connected and it agrees with $\A_r$ in neighborhoods of $\ee^{\ii\alpha}$ and $\ee^{\ii\beta-r}$.
Recall that $\Ptwo^{(\kappa)}(\Omega; \ee^{\ii\alpha}, \ee^{\ii\beta-r})$ denotes the law of chordal $\SLE_{\kappa}$ in $(\Omega; \ee^{\ii\alpha}, \ee^{\ii\beta-r})$.
Lawler proves in~\cite[Sections~4-5]{lawlerDefiningSLEwithBrownianLoop} that there exists a unique probability measure $\Pcro{1}^{(\kappa)}(\A_r;\ee^{\ii \alpha},\ee^{\ii \beta-r})$ and a partition function $\LFcro{1}^{(\kappa)}(r; \alpha, \beta)$ such that the following boundary perturbation holds:
for any simply connected test domain $\Omega\subset \A_r$ that agrees with $\A_r$ in neighborhoods of $\ee^{\ii \alpha}$ and $\ee^{\ii \beta-r}$,
\begin{equation} \label{eqn::bp_annulus_Def}
	\frac{\ud \Ptwo^{(\kappa)}(\Omega;\ee^{\ii \alpha},\ee^{\ii \beta-r})}{\ud \Pcro{1}^{(\kappa)}(\A_r;\ee^{\ii \alpha},\ee^{\ii \beta-r})}(\gamma)=\frac{\LFcro{1}^{(\kappa)}(r;\alpha,\beta)}{\LZtwo(\Omega;\ee^{\ii \alpha},\ee^{\ii \beta-r})} \one\{\gamma\subset \Omega\} \exp\left(\frac{\mathfrak{c}}{2} \blm(\A_r;\gamma,\A_r\setminus \Omega)\right).
\end{equation}
For a curve $\gamma$ in $\A_r$ from $\ee^{\ii\alpha}$ to $\ee^{\ii\beta-r}$, we say that $\gamma$ has winding $\ell$ if $q^{-1}(\gamma)$ connects $\alpha$ to $\beta+2\pi\ell+\ii r$ and we denote $\wind(\gamma)=\ell$.
The measure $\Pcro{1}^{(\kappa)}(\A_r;\ee^{\ii \alpha},\ee^{\ii \beta-r})$ can be decomposed according to the winding numbers of the curves: 
\begin{equation} \label{eqn::winding_decomposition}
	\Pcro{1}^{(\kappa)}(\A_r;\ee^{\ii \alpha},\ee^{\ii \beta-r}) = \frac{1}{\LFcro{1}^{(\kappa)}(r;\alpha,\beta)} \sum_{\ell \in \Z} \LFcro{1}^{(\kappa;\ell)}(r;\alpha,\beta) \Pcro{1}^{(\kappa;\ell)}(\A_r;\ee^{\ii \alpha},\ee^{\ii \beta-r}),
\end{equation}
where $\Pcro{1}^{(\kappa;\ell)}(\A_r;\ee^{\ii \alpha},\ee^{\ii \beta-r})$ is the conditional measure of curves with winding $\ell$:
\begin{equation} \label{eqn::condition_winding}
	\Pcro{1}^{(\kappa;\ell)}(\A_r;\ee^{\ii \alpha},\ee^{\ii \beta-r})[\cdot]=\Pcro{1}^{(\kappa)}(\A_r;\ee^{\ii \alpha},\ee^{\ii \beta-r})[\cdot | \wind(\gamma)=\ell],
\end{equation}
and $\LFcro{1}^{(\kappa;\ell)}(r;\alpha,\beta)$ is the associated partition function and $\LFcro{1}^{(\kappa)}(r;\alpha,\beta)=\sum_{\ell\in \Z} \LFcro{1}^{(\kappa;\ell)}(r;\alpha,\beta)$.
The partition function $\LFcro{1}^{(\kappa;\ell)}(r;\alpha,\beta)$ satisfies
\begin{align*}
	\LFcro{1}^{(\kappa;\ell)}(r;\alpha,\beta)=\LFcro{1}^{(\kappa;\ell)}(r;\alpha+c,\beta+c), \quad \text{for } c\in \R;\qquad
	\LFcro{1}^{(\kappa;\ell)}(r;\alpha,\beta)=\LFcro{1}^{(\kappa;0)}(r;\alpha,\beta+2\ell \pi),\quad\text{for }\ell\in\Z;
\end{align*}
and the annulus BPZ equation (see~\cite[Proposition~7.1 and Eq.~(73)]{lawlerDefiningSLEwithBrownianLoop}):
\begin{equation}\label{eqn::annulusBPZ_single}
\begin{split}
	\left(\partial_r-F_1(r)\right)\LZ=& \LD_{\alpha}^{(2)}\LZ,\quad\text{where }\LD_{\alpha}^{(2)}=\frac{\kappa}{2} \partial_{\alpha}^2 +  H_3(r;\beta-\alpha) \partial_{\beta}  + \mathfrak{b} \partial_z H_3(r;\beta-\alpha);\\
	\left(\partial_r-F_1(r)\right)\LZ=& \LD_{\beta}^{(2)}\LZ,\quad\text{where }\LD_{\beta}^{(2)}=\frac{\kappa}{2} \partial_{\beta}^2 +  H_3(r;\alpha-\beta) \partial_{\alpha}  + \mathfrak{b} \partial_z H_3(r;\alpha-\beta);
\end{split}
\end{equation}
and
\begin{equation*}
	F_1(r)=\frac{6\tilde{\mathfrak{b}}-\mathfrak{b}}{r} + 6 \tilde{\mathfrak{b}} \LE(r) + \mathfrak{b},
\end{equation*}
and $H_3$ is derivative of rescaled Jacobi theta functions defined in~\eqref{eqn::H1H3_def}, constants $\mathfrak{b}, \tilde{\mathfrak{b}}$ are defined in~\eqref{eqn::universal_parameters} and $\LE(r)$ is defined in~\eqref{eqn::Def_LE}.

We extend the annulus SLE measure to general doubly connected domains. Let $D$ be a doubly connected domain, and $x,y\in \partial D$ lie on different boundary components. 
Let $\varphi:D\to\A_r$ be a conformal map that sends the boundary component containing $x$ to $\partial\U$ and the boundary component containing $y$ to $\ee^{-r}\partial\U$, with $\varphi(x)=\ee^{\ii\alpha}$ and $\varphi(y)=\ee^{\ii\beta-r}$.
We define $\Pcro{1}^{(\kappa;\ell)}(D;x,y)$ as the image of $\Pcro{1}^{(\kappa;\ell)}(\A_r;\ee^{\ii\alpha},\ee^{\ii\beta-r})$ under $\varphi^{-1}$. Assuming $\partial D$ is $C^{1+\eps}$ in neighborhoods of $x$ and $y$ for some $\eps>0$, we define
\begin{equation*}
	\LFcro{1}^{(\kappa;\ell)}(D;x,y)
	:=|\varphi'(x)|^{\mathfrak{b}}|\varphi'(y)|^{\mathfrak{b}}
	\LFcro{1}^{(\kappa;\ell)}(r;\alpha,\beta).
\end{equation*}

\medbreak
Zhan~\cite{ZhanReversibilityWholeplaneSLE} constructs the partition function for annulus SLE by constructing a particular solution $\LFcro{1}^{(\kappa;\ell)}$ to annulus BPZ equations~\eqref{eqn::annulusBPZ_single}. 
As shown in~\cite[Theorems~1.1 and~1.2]{ZhanRestrictionAnnulusSLE}, the annulus $\SLE_{\kappa}$ defined via the Loewner chain~\eqref{eqn::annuluSLEkappa} with $\LZ=\LFcro{1}^{(\kappa)}$ satisfies the boundary perturbation property~\eqref{eqn::bp_annulus_Def}, ensuring that the two approaches describe the same path measure. Note that the partition functions constructed in these works may differ by a multiplicative factor depending solely on the modulus $r$. As the drift of the driving process in the annulus Loewner chain depends only on the spatial derivatives (w.r.t. $\alpha$ or $\beta$), any $r$-dependent factor cancels out. This freedom corresponds to the $F(r)$ term in the system of annulus BPZ equations~(\ref{eqn::annulus_BPZ_alpha},\ref{eqn::annulus_BPZ_beta}). The partition function $\LFcro{1}^{(\kappa;\ell)}(r;\alpha,\beta)$ does not have explicit formula in general (see~\cite[Section~6.1-6.4]{ZhanReversibilityWholeplaneSLE} or~\cite[Section~4.2]{ZhanRestrictionAnnulusSLE}), but it coincides with $\LZcro{1}^{(\ell)}$ in~\eqref{eqn::LZcroell_def} when $\kappa=4$, see~\eqref{eqn::single_annulusSLE4_pf} and~\cite[Eq.~(4.7,4.9)]{ZhanRestrictionAnnulusSLE}.
\medbreak
We summarize in the following lemmas the properties of a single annulus $\SLE_{\kappa}$ and its partition function that will be used later. 
\begin{lemma}\label{lem::singleannulusSLE}
Fix $\kappa\in (0,4]$ and $\ell\in\Z$. 
\begin{itemize}
\item Transience: the process $\gamma$ driven by the driving function~\eqref{eqn::annuluSLEkappa} with $\LZ=\LFcro{1}^{(\kappa;\ell)}(r;\alpha,\beta)$ is a continuous curve that almost surely terminates at the target point $\ee^{\ii \beta-r}$ as $t \to r$ and it has winding $\ell$, i.e. $q^{-1}(\gamma)$ connects $\alpha$ to $\beta+2\pi\ell+\ii r$.
\item Boundary perturbation (in simply connected domain): for any simply connected test domain $\Omega\subset \A_r$ that agrees with $\A_r$ in neighborhoods of $\ee^{\ii \alpha}$ and $\ee^{\ii \beta-r}$, we denote by $\Ptwo^{(\kappa)}(\Omega;\ee^{\ii \alpha},\ee^{\ii \beta-r})$ the law of chordal $\SLE_{\kappa}$ in $(\Omega;\ee^{\ii \alpha};\ee^{\ii \beta-r})$, then 
	\begin{equation} \label{eqn::bp_annulus}
		\frac{\ud \Ptwo^{(\kappa)}(\Omega;\ee^{\ii \alpha},\ee^{\ii \beta-r})}{\ud \Pcro{1}^{(\kappa;\ell)}(\A_r;\ee^{\ii \alpha},\ee^{\ii \beta-r})}(\gamma)=\frac{\LFcro{1}^{(\kappa;\ell)}(r;\alpha,\beta)}{\LZtwo(\Omega;\ee^{\ii \alpha},\ee^{\ii \beta-r})}  \one\{ \gamma\subset \Omega\} \exp\left(\frac{\mathfrak{c}}{2} \blm(\A_r;\gamma,\A_r\setminus \Omega)\right). 
\end{equation}
\item Boundary perturbation (in doubly connected domain): 
	for any doubly connected test domain $D \subset \A_r$ that agrees with $\A_r$ in neighborhoods of $\ee^{\ii \alpha}$ and $\ee^{\ii \beta-r}$, then
	\begin{equation}\label{eqn::bp_annulus_doubly_connected}
		\frac{\ud \Pcro{1}^{(\kappa; \ell)}(D;\ee^{\ii \alpha},\ee^{\ii \beta-r})}{\ud \Pcro{1}^{(\kappa;\ell)}(\A_r;\ee^{\ii \alpha},\ee^{\ii \beta-r})}(\gamma)=
		\frac{\LFcro{1}^{(\kappa;\ell)}(r;\alpha,\beta)}{\LFcro{1}^{(\kappa;\ell)}(D;\ee^{\ii \alpha},\ee^{\ii \beta-r})}
		\one\{\gamma\subset D\}
		\exp\left(\frac{\mathfrak{c}}{2}\blm(\A_r;\gamma,\A_r \setminus D)\right).
	\end{equation}
\end{itemize}
\end{lemma}
\begin{proof}
It is shown in~\cite[Proposition~6.4 and Theorem~7.3]{ZhanReversibilityWholeplaneSLE} that $\gamma$ driven by the driving function~\eqref{eqn::annuluSLEkappa} with $\LZ=\LFcro{1}^{(\kappa;\ell)}(r;\alpha,\beta)$ is a continuous curve that almost surely terminates at the target point $\ee^{\ii \beta-r}$ as $t \to r$, and~\eqref{eqn::condition_winding} implies it has winding $\ell$. The relation~\eqref{eqn::bp_annulus} follows from~\eqref{eqn::bp_annulus_Def} and~\eqref{eqn::winding_decomposition}. 
Moreover, for any simply connected test domain $\Omega\subset D$ that agrees with $D$ in neighborhoods of $x$ and $y$, we have
\begin{equation*}
	\blm(\A_r;\gamma,\A_r\setminus \Omega)=\blm(\A_r;\gamma,\A_r\setminus D)+\blm(D;\gamma, D\setminus\Omega).
\end{equation*}	
Plugging into~\eqref{eqn::bp_annulus}, we obtain that~\eqref{eqn::bp_annulus_doubly_connected} holds for $\gamma$ in any test domain $\Omega$.
Exhausting $\Omega$ by such simply connected test domains gives~\eqref{eqn::bp_annulus_doubly_connected}.
\end{proof}
\begin{lemma}
When $\kappa=4$, the partition function $\LFcro{1}^{(4;\ell)}$ coincides with $\LZcro{1}^{(\ell)}$ defined in~\eqref{eqn::LZcroell_def}:
\begin{equation}\label{eqn::single_annulusSLE4_pf}
	\LFcro{1}^{(4;\ell)}(r;\alpha,\beta)=\LZcro{1}^{(\ell)}(r;\alpha,\beta).
\end{equation}
\end{lemma}
\begin{proof}
When $\kappa=4$, the partition function constructed in~\cite[Section~6.1-6.4]{ZhanReversibilityWholeplaneSLE} (see also~\cite[Section~4.2]{ZhanRestrictionAnnulusSLE}) is \[ \LZcro{1}^{(\ell)}(r;\alpha,\beta)\times \sqrt{\frac{\pi}{r}} \ee^{-\frac{nr}{8}} \times \prod_{k=1}^{\infty}(1-\ee^{-2kr})^{-\frac{3n}{2}}, \] which generates the same driving function as $\LZcro{1}^{(\ell)}$ in~\eqref{eqn::annuluSLEkappa} and implies~\eqref{eqn::single_annulusSLE4_pf}. See also~\cite[Corollary~2.15]{AruBordereauSLEpartitionfunctionmultiplyGFF}, where the partition function satisfying the boundary perturbation property~\eqref{eqn::bp_annulus} is also identified.
\end{proof}

\subsection{Multi-time martingale}
\label{subsec::multitime_mart}


\paragraph*{Multi-time parameter.}
Fix $n \ge 1$, initial points $\bs{\theta} = (\theta_1, \ldots, \theta_n) \in \LX_n$, and an initial modulus $r > 0$. 
To define the multi-time framework in the annulus, we first fix the normalization of individual slit evolution. 
For $1\le j\le n$, let $(\zeta_t^j)_{0 \le t <r}$ be a continuous driving function with $\zeta_0^j = \theta_j$. 
We define $\mathfrak{g}_{t_j}^j$ as the unique solution to the annulus Loewner equation~\eqref{eqn::annulus_Loewner_equation}: 
\begin{equation}\label{eqn::multitime_aux1}
	\partial_{t_j} \mathfrak{g}_{t_j}^j(z) = \mathfrak{g}_{t_j}^j(z) S_{r-t_j}( \mathfrak{g}_{t_j}^j(z)/\ee^{\ii \zeta_{t_j}^j} ), \qquad \mathfrak{g}_0^j(z) = z, \qquad z \in \overline{\A}_r.
\end{equation}
Let $\mathfrak{\covmap}_{t_j}^j$ be the covering map of $\mathfrak{g}_{t_j}^j$. This fixes the normalization for each individual slit flow. Let $\gamma^j$ be the curve generated by the driving function $\zeta^j$ for $1\le j\le n$. Assume $\gamma_{[0,t_i]}^i\cap \gamma_{[0,t_j]}^j=\emptyset$ for $1\le i<j\le n$.
We define the following conformal transformations:
\begin{itemize}
	\item Let $\mathfrak{g}_{\bs{t}}$ be the conformal map from the component of $\mathbb{A}_r \setminus \cup_{j=1}^n \gamma^j_{[0, t_j]}$ containing the inner boundary to $\mathbb{A}_{\Mod(\A_r\setminus \bs{\gamma}_{[\bs{0},\bs{t}]})}$ with \[\mathfrak{g}_{\bs{t}} (\partial(\exp(-r)\U))=\partial( \exp(-\Mod(\A_r\setminus \bs{\gamma}_{[\bs{0},\bs{t}]})) \U).\] Note that $\mathfrak{g}_{\bs{t}}$ is unique up to rotation.
	\item Let $\mathfrak{g}_{\bs{t}, j}=\mathfrak{g}_{\bs{t}} \circ \mathfrak{g}_{t_j}^{-1}$. Then $\mathfrak{g}_{\bs{t}, j}$ is the conformal map from the component of $\mathbb{A}_{r-t_j} \setminus \mathfrak{g}_{t_j}^j \left( \cup_{i \neq j} \gamma^i_{[0, t_i]} \right)$ containing the inner boundary to $\mathbb{A}_{\Mod(\A_r\setminus \bs{\gamma}_{[\bs{0},\bs{t}]})}$ with \[\mathfrak{g}_{\bs{t},j} (\partial(\exp(t_j-r)\U))=\partial( \exp(-\Mod(\A_r\setminus \bs{\gamma}_{[\bs{0},\bs{t}]})) \U).\]
\end{itemize}
Let $\mathfrak{\covmap}_{\bs{t}}, \mathfrak{\covmap}_{\bs{t}, j}$ be the covering maps of $\mathfrak{g}_{\bs{t}}, \mathfrak{g}_{\bs{t}, j}$ respectively. Then $\mathfrak{\covmap}_{\bs{t}}, \mathfrak{\covmap}_{\bs{t}, j}$ are unique up to a translation $c \in \R$.
We define the multi-slit driving function $\bs{W}_{\bs{t}} = (W^1_{\bs{t}}, \ldots, W^n_{\bs{t}})$ by
\begin{equation}\label{eqn::multitime_driving_annulus}
	W^j_{\bs{t}} = \mathfrak{\covmap}_{\bs{t}, j}(\zeta^j_{t_j}), \qquad \text{for } 1 \le j \le n.
\end{equation}
Note that while $\mathfrak{\covmap}_{\bs{t}, j}$ is unique up to a translation $c \in \R$, the value $W^j_{\bs{t}}$ and the map $\mathfrak{\covmap}_{\bs{t}}$ shift by the same $c$, leaving the term $W^i_{\bs{t}} - W^j_{\bs{t}}$, $\mathfrak{\covmap}_{\bs{t}}(z) - W^j_{\bs{t}}$, $\mathfrak{\covmap}_{\bs{t}}(z_1) - \mathfrak{\covmap}_{\bs{t}}(z_2)$ and the derivative $\mathfrak{\covmap}'_{\bs{t}, j}$ well-defined.

\begin{lemma} \label{lem::mt_blm_annulus_diff}
Assume the same notation as in the paragraph of~\eqref{eqn::multitime_aux1}-\eqref{eqn::multitime_driving_annulus}.
The Brownian loop measure
\begin{equation}\label{eqn::mt_blm_annulus}
	\blm_{r,\bs{t}}:=\blm \left(\A_r;\gamma_{[0,t_1]}^{1},\ldots,\gamma_{[0,t_n]}^{n} \right)
\end{equation} 
solves the following differential equations: for all $j\in\{1, \ldots, n\}$, 
\begin{equation} \label{eqn::mt_blm_annulus_diff}
	\partial_{t_j} \blm_{r,\bs{t}} = -\frac{1}{3} \LS \mathfrak{\covmap}_{\bs{t},j} (\zeta_{t_j}^j) +  \mathfrak{\covmap}_{\bs{t},j} (\zeta_{t_j}^j)^2 \left( \LE (\Mod(\A_r \setminus \bs{\gamma}_{[\bs{0},\bs{t}]})) +\frac{1}{\Mod(\A_r \setminus \bs{\gamma}_{[\bs{0},\bs{t}]})} \right) - \left( \LE (r-t_j) + \frac{1}{r-t_j} \right),
\end{equation}
where $\LE(t)$ is defined in~\eqref{eqn::Def_LE}. 
\end{lemma}

\begin{proof}
From~\eqref{eqn::blm_def}, we have
\begin{align} \label{eqn::mt_blm_annulus_diff_aux1}
\begin{split}
	\; &\blm \big( \A_r;\gamma_{[0,t_1]}^{1},\ldots,\gamma_{[0,t_n]}^{n} \big) \\
	= \; & \sum_{j=1}^{n-1} \blm^{\mathrm{loop}}\big[ \ell\subset\A_r \colon \ell\cap \gamma_{[0,t_1]}^{1}\neq \emptyset , \textnormal{ and } \ell\cap \gamma_{[0,t_i]}^{i}\neq \emptyset\textnormal{ for at least }j \textnormal{ of the }i\in\{2, \ldots, n\}\big] \\
	\; &+ \sum_{j=2}^{n-1} \blm^{\mathrm{loop}}\big[\ell \subset \A_r \setminus \gamma_{[0,t_1]}^{1} \colon \ell\cap \gamma_{[0,t_i]}^{i} \neq \emptyset\textnormal{ for at least }j \textnormal{ of the }i\in\{2, \ldots, n\}\big] \\
	=\; &\blm \big( \A_r;\gamma_{[0,t_1]}^{1}, \cup_{j=2}^n \gamma_{[0,t_j]}^{j} \big)
	\, + \, \blm \big( \A_r;\gamma_{[0,t_2]}^{2},\ldots,\gamma_{[0,t_n]}^{n} \big) .	
\end{split}
\end{align}
It is proved in~\cite[Lemma~7.2]{ZhanRestrictionAnnulusSLE} that
\begin{equation} \label{eqn::mt_blm_annulus_diff_aux2}
\begin{split}
	& \blm \big( \A_r;\gamma_{[0,t_1]}^{1}, \cup_{j=2}^n \gamma_{[0,t_j]}^{j} \big)
	\\
	=& - \frac{1}{3} \int_{0}^{t_1} \LS \mathfrak{\covmap}_{\bs{t},1} (\zeta_{s_1}^1) \ud s_1 + \int_{r}^{\Mod(\A_r\setminus \cup_{j=2}^n \gamma_{[0,t_j]}^j)} \left( \LE(s)+\frac{1}{s} \right) \ud s - \int_{r-t_1}^{\Mod(\A_r\setminus \bs{\gamma}_{[\bs{0},\bs{t}]})} \left( \LE(s)+\frac{1}{s} \right)\ud s.
\end{split}
\end{equation}
Combining~(\ref{eqn::Annulus_basic1},\ref{eqn::mt_blm_annulus_diff_aux1},\ref{eqn::mt_blm_annulus_diff_aux2}), we prove~\eqref{eqn::mt_blm_annulus_diff} for $j=1$. The variations $\partial_{t_j} \blm_{r,\bs{t}}$ for $2\le j\le n$ can be calculated in the same way and we obtain~\eqref{eqn::mt_blm_annulus_diff} as desired.
\end{proof}

\paragraph*{Multi-time local martingale.}
Fix $n\ge 1$. 
Assume $\LZ_n(r;\bs{\alpha}, \bs{\beta}): \R_{>0}\times\LX_n\times\LX_n\to \R $ is positive, smooth and rotation-invariant, i.e.
\begin{equation}
	\LZ_n(r;\bs{\alpha}+c, \bs{\beta}+c)=\LZ_n(r;\bs{\alpha},\bs{\beta}),\qquad\text{for all }c\in\R, 
\end{equation} 
where $\bs{\theta}+c=(\theta_1+c, \ldots, \theta_n+c)$ for $\bs{\theta}\in\LX_n$. 
Suppose further that $\LZ_n$ satisfies annulus BPZ equations~\eqref{eqn::annulus_BPZ_alpha}. 
We construct local martingales associated with $\LZ_n$ below.

\begin{proposition}\label{prop::multitime_mart_annulus}
Fix $n\ge 1$ and $\kappa\in (0,4]$ and $\bs{\alpha}, \bs{\beta}\in\LX_n$. 
\begin{itemize}
\item Assume $\LF_1: \R_{>0}\times\R\times\R\to \R$ is positive, smooth and rotation-invariant.  
Assume further that $\LF_1$ satisfies annulus BPZ equations~\eqref{eqn::annulusBPZ_single}. 
\item Assume $\LF_n: \R_{>0}\times\LX_n\times\LX_n\to \R $ is positive, smooth and rotation-invariant. Assume further that $\LF_n$ satisfies the system of annulus BPZ equations~\eqref{eqn::annulus_BPZ_alpha} with $F(r)=F_n(r)$ defined in~\eqref{eqn::F(r)inBPZ}. 
\end{itemize}
For $1\le j\le n$, let $\gamma^j$ be the annulus $\SLE_{\kappa}(\LF_1)$ in $(\A_r; \ee^{\ii\alpha_j}, \ee^{\ii\beta_j-r})$. Let $\Pind{n}^{(\kappa)}(\LF_1)$ be the probability measure on $\bs{\gamma}=(\gamma^{1}, \ldots, \gamma^{n})$ under which the curves are independent. 
We parameterize $\bs{\gamma}$ by $n$-time parameter $\bs{t}$, and let $\bs{W}_{\bs{t}} = (W_{\bs{t}}^{1},\ldots,W_{\bs{t}}^{n})$ be the multi-slit driving function (it is unique up to rotation).
Then the following process is $n$-time local martingale under $\Pind{n}^{(\kappa)}(\LF_1)$: 
\begin{align}\label{eqn::multitime_mart_annulus}
\begin{split}
	M_{\bs{t}}(\LF_1; \LF_n)
	:= \; & \one_{\LE_{\emptyset}(\bs{\gamma}_{\bs{t}})} \, \exp\bigg(\frac{\mathfrak{c}}{2}\blm_{r,\bs{t}} + \mathfrak{b} \sum_{j=1}^{n} (r-t_j) - n \mathfrak{b} \Mod(\A_r\setminus \bs{\gamma}_{[\bs{0},\bs{t}]}) \bigg) 
	\; \times \prod_{j=1}^{n} \mathfrak{\covmap}_{\bs{t},j}'(\zeta_{t_j}^{j})^{\mathfrak{b}}  \\
	& \times \prod_{j=1}^{n} \frac{\mathfrak{\covmap}_{\bs{t}}'(\beta_j+\ii r)^{\mathfrak{b}}}{(\mathfrak{\covmap}_{t_j}^j)'( \beta_j + \ii r)^{\mathfrak{b}}}  \times \frac{\LF_n(\Mod(\A_r\setminus \bs{\gamma}_{[\bs{0},\bs{t}]});\bs{W}_{\bs{t}},\Re \mathfrak{\covmap}_{\bs{t}} (\bs{\beta}+\ii r)) }{\prod_{j=1}^{n} \LF_1(r-t_j;\zeta_{t_j}^j,\Re \mathfrak{\covmap}_{t_j}^j(\beta_j+\ii r))};
\end{split}
\end{align}
where $\LE_{\emptyset}(\bs{\gamma}_{\bs{t}}) = \{ \gamma_{[0,t_j]}^{j} \cap \gamma_{[0,t_i]}^{i}=\emptyset, \, \forall i\neq j\}$ is the event that different curves are disjoint, and 
$\blm_{r,\bs{t}}$ is the Brownian loop term defined in~\eqref{eqn::mt_blm_annulus}: 
\[	\blm_{r,\bs{t}}:=\blm \left(\A_r;\gamma_{[0,t_1]}^{1},\ldots,\gamma_{[0,t_n]}^{n} \right).\]
\end{proposition}

As we mentioned earlier, Proposition~\ref{prop::multitime_mart_annulus} can be viewed as the annulus counterpart of the chordal and radial multi-time martingales in~\cite[Proposition~2.1]{HuangWuYangMultipleSLEsDysonBM} and~\cite[Proposition~2.4]{HuangPeltolaWuMultiradialSLEResamplingBP}. It is also related to Zhan's two-time annulus commutation martingale in~\cite[Sections~4.2--4.3]{ZhanReversibilityWholeplaneSLE}, but the geometric configuration is different. In the present proposition the curves are treated from the same side: the starting points lie on one boundary component and the targets lie on the other, and all curves are encoded in a single $n$-time martingale. Since Zhan's goal there is to study reversibility, his two-curve martingale is designed for a commutation coupling in which, after applying the annulus inversion, the two growing curves are viewed from different sides of the annulus. The same idea should also allow a mixed-side generalization of Proposition~\ref{prop::multitime_mart_annulus}, with some curves grown from one boundary component and some from the other, but this extension is not needed in the present article.

To prove Proposition~\ref{prop::multitime_mart_annulus}, let us calculate the variation of terms in RHS of~\eqref{eqn::multitime_mart_annulus}. 
By~\eqref{eqn::annuluSLEkappa}, the driving function of $\gamma^j$ satisfies
\begin{equation*}
	\ud \zeta_{t_j}^j = \sqrt{\kappa} \ud B_{t_j}^j + \kappa \partial_{\alpha} \log \LF_1(r-t; \zeta_{t_j}, \Re \mathfrak{\covmap}_{\bs{t}}(\beta_j+\ii r)) \ud t_j,
\end{equation*}
where $B^1,\ldots,B^n$ are independent Brownian motions. 
To ease our notation, we simply write
\begin{equation*}
	R_{r;\bs{t}}=\Mod(\A_r\setminus \bs{\gamma}_{[\bs{0},\bs{t}]}).
\end{equation*}
We have the following standard calculations (see~\cite{ZhanSLEannulus,ZhanRestrictionAnnulusSLE,ZhanReversibilityWholeplaneSLE}):
\begin{align}
	\ud \; R_{r;\bs{t}} = & -\sum_{j=1}^{n} \mathfrak{\covmap}'_{\bs{t}, j}(\zeta_{t_j}^j)^2 \ud t_j \label{eqn::Annulus_basic1}\\
	\ud \left( \Re \mathfrak{\covmap}_{\bs{t}} (\beta+\ii r) - W^j_{\bs{t}} \right)  = & \sum_{i=1}^{n}  H_3 \left( R_{r;\bs{t}} ; 
	\Re \mathfrak{\covmap}_{\bs{t}}(\beta+\ii r) - W_{\bs{t}}^i 
	\right) (\mathfrak{\covmap}_{\bs{t},i}' (\zeta_{t_i}^{i}))^2 \ud t_i, \label{eqn::Annulus_basic2} \\
	& - \mathfrak{\covmap}'_{\bs{t},j}(\zeta_{t_j}^j) \ud \zeta_{t_j}^j +\kappa \mathfrak{b} \mathfrak{\covmap}''_{\bs{t},j}(\zeta_{t_j}^j) \ud t_j - \sum_{i\neq j} H_1 \left( R_{r;\bs{t}} ; W_{\bs{t}}^j -W_{\bs{t}}^i \right) (\mathfrak{\covmap}_{\bs{t},i}' (\zeta_{t_i}^{i}))^2 \ud t_i, \notag \\
	\frac{\ud \mathfrak{\covmap}'_{\bs{t}, j}(\zeta_{t_j}^j)}{\mathfrak{\covmap}'_{\bs{t}, j}(\zeta_{t_j}^j)}  = & 	\frac{ \mathfrak{\covmap}_{\bs{t},j}''(\zeta_{t_j}^{j})}{\mathfrak{\covmap}_{\bs{t},j}'(\zeta_{t_j}^{j})} \, \ud \zeta_{t_j}^{j} 
	+ \sum_{i \neq j} \partial_z H_1\left( R_{r;\bs{t}}; W_{\bs{t}}^{j}-W_{\bs{t}}^i  \right) (\mathfrak{\covmap}_{\bs{t},i}' (\zeta_{t_i}^{i}))^2 \ud t_i \label{eqn::Annulus_basic3} \\
	&+ \bigg( \bigg( \frac{3\kappa-8}{6} \bigg) \frac{ \mathfrak{\covmap}_{\bs{t},j}'''(\zeta_{t_j}^{j})}{\mathfrak{\covmap}_{\bs{t},j}'(\zeta_{t_j}^{j})} + \frac{1}{2} \bigg( \frac{ \mathfrak{\covmap}_{\bs{t},j}''(\zeta_{t_j}^{j})}{\mathfrak{\covmap}_{\bs{t},j}'(\zeta_{t_j}^{j})}\bigg)^2 - \LE(r-t_j) + \LE(R_{r;\bs{t}})(\mathfrak{\covmap}_{\bs{t},j}' (\zeta_{t_j}^{j}))^2 \bigg) \ud t_j,	\notag \\
	\frac{\ud \mathfrak{\covmap}'_{\bs{t}} (\beta+\ii r)}{\mathfrak{\covmap}'_{\bs{t}} (\beta+\ii r)} = & \sum_{i=1}^{n} \partial_z H_3 \left( R_{r;\bs{t}}; 
	\Re \mathfrak{\covmap}_{\bs{t}}(\beta+\ii r) - W_{\bs{t}}^i 
	\right) (\mathfrak{\covmap}_{\bs{t},i}' (\zeta_{t_i}^{i}))^2 \ud t_i. \label{eqn::Annulus_basic4}
\end{align}

\begin{proof}[Proof of Propositon~\ref{prop::multitime_mart_annulus}]
Applying It\^{o}'s formula and combining
~(\ref{eqn::annulus_Loewner_equation_cov}, \ref{eqn::annulus_Loewner_equation_cov_H3}, \ref{eqn::annuluSLEkappa}, \ref{eqn::Annulus_basic1}, \ref{eqn::Annulus_basic2}, \ref{eqn::Annulus_basic3}, \ref{eqn::Annulus_basic4}, \ref{eqn::mt_blm_annulus_diff}), 
we have (see details in Appendix~\ref{appendix::multitime_mart})
\begin{align}
 \label{eqn::multitime_mart_annulus_aux1}
	&\frac{\ud M_{\bs{t}}(\LF_1; \LF_n)}{M_{\bs{t}}(\LF_1; \LF_n)} \notag\\
	=  &\sum_{j=1}^{n} \sqrt{\kappa} \Bigg( \mathfrak{\covmap}_{\bs{t},j}' (\zeta_{t_j}^{j}) \frac{\partial_{\alpha_j} \LF_n (	R_{r;\bs{t}};\bs{W}_{\bs{t}},\Re \mathfrak{\covmap}_{\bs{t}} (\bs{\beta}+\ii r))}{\LF_n ( R_{r;\bs{t}};\bs{W}_{\bs{t}},\Re \mathfrak{\covmap}_{\bs{t}} (\bs{\beta}+\ii r))} + \mathfrak{b} \frac{\mathfrak{\covmap}''_{\bs{t},j}(\zeta_{t_j}^j)}{\mathfrak{\covmap}'_{\bs{t},j}(\zeta_{t_j}^j)} -	\frac{\partial_{\alpha} \LF_1 (r-t_j;\zeta_{t_j}^j,\Re \mathfrak{\covmap}_{t_j}^j(\beta_j+\ii r))  }{\LF_1 (r-t_j;\zeta_{t_j}^j,\Re \mathfrak{\covmap}_{t_j}^j(\beta_j+\ii r)) }  \Bigg) \ud B_{t_j}^j  \notag\\
	& + \sum_{j=1}^{n}  \frac{\left(F_n(r)-\partial_r+\LD^{(n)}_{\alpha_j}\right) \LF_n (	R_{r;\bs{t}};\bs{W}_{\bs{t}},\Re \mathfrak{\covmap}_{\bs{t}} (\bs{\beta}+\ii r)) }{\LF_n (	R_{r;\bs{t}};\bs{W}_{\bs{t}},\Re \mathfrak{\covmap}_{\bs{t}} (\bs{\beta}+\ii r)) } \mathfrak{\covmap}_{\bs{t},j}' (\zeta_{t_j}^{j})^2 \ud t_j\notag\\
	&- \sum_{j=1}^{n} \frac{\left(F_1(r)-\partial_r+\LD^{(2)}_{\alpha}\right) \LF_1(r-t_j;\zeta_{t_j}^j,\Re \mathfrak{\covmap}_{t_j}^j(\beta_j+\ii r)) }{\LF_1(r-t_j;\zeta_{t_j}^j,\Re \mathfrak{\covmap}_{t_j}^j(\beta_j+\ii r))} \mathfrak{\covmap}_{\bs{t},j}' (\zeta_{t_j}^{j})^2 \ud t_j,
\end{align}
where $\LD^{(2)}_{\alpha}, \LD^{(n)}_{\alpha_j}$ are the differential operators in the annulus BPZ equations~\eqref{eqn::annulusBPZ_single} and~\eqref{eqn::annulus_BPZ_alpha} and $F_1(r), F_n(r)$ are the same as in~\eqref{eqn::annulusBPZ_single} and~\eqref{eqn::F(r)inBPZ}: 
\begin{align*}
F_1(r)=\frac{6\tilde{\mathfrak{b}}-\mathfrak{b}}{r} + 6 \tilde{\mathfrak{b}} \LE(r) + \mathfrak{b},\qquad F_n(r)=\frac{6\tilde{\mathfrak{b}}-\mathfrak{b}}{r} + 6 \tilde{\mathfrak{b}} \LE(r) + \mathfrak{b} n.
\end{align*}
Note that by~\eqref{eqn::annulusBPZ_single}, we have $(F_1(r)-\partial_r+\LD^{(2)}_{\alpha}) \LF_1=0$. Plugging into~\eqref{eqn::multitime_mart_annulus_aux1}, we see that $M_{\bs{t}}(\LF_1; \LF_n)$ is $n$-time local martingale under $\Pind{n}^{(\kappa)}(\LF_1)$ if and only if all the drift terms vanish, which happens  when $\LF_n$ satisfies the system of annulus BPZ equations~\eqref{eqn::annulus_BPZ_alpha} with $F(r)=F_n(r)$ given by~\eqref{eqn::F(r)inBPZ}. This finishes the proof.
\end{proof}

\subsection{Construction from Brownian loop measure}
\label{subsec::earlierconstruction}
In this section, we work on the construction of multi-annulus SLE and its partition function in~\cite{ZhanRestrictionAnnulusSLE} and~\cite{JahangoshahiLawlerMultiplepathsSLE}, and derive their properties that will be used in the proof of Propositions~\ref{prop::multiannulu_pf_construction} and~\ref{prop::ac_general}. 
\begin{lemma}\label{lem::earlierconstruction_marginal_conditional}
Assume the same setup as in Proposition~\ref{prop::multiannulu_pf_construction}. 
We denote by $\widehat{\mathsf{P}}=\widehat{\mathsf{P}}(\A_r; \ee^{\ii\bs{\alpha}}, \ee^{\ii\bs{\beta}})$ the measure $\Pind{n}^{(\kappa;\ell)}$ weighted by~\eqref{eqn::ac_RN}:
\begin{align} \label{eqn::earlierconstruction}
\frac{\prod_{j=1}^n\LFcro{1}^{(\kappa;\ell)}(r; \alpha_j, \beta_j)}{\LFcro{n}^{(\kappa;\ell)}(r; \bs{\alpha}, \bs{\beta})}\one_{\LE_{\emptyset}(\bs{\gamma})}\exp\left(\frac{\mathfrak{c}}{2}\blm(\A_r; \gamma^1, \ldots, \gamma^n)\right). 
\end{align}
We denote $\dot{\bs{\alpha}}_1=(\alpha_2, \ldots, \alpha_n)$ and $\dot{\bs{\beta}}_1=(\beta_n, \ldots, \beta_2)$. 
Then the law of $\bs{\gamma}=(\gamma^1, \ldots, \gamma^n)$ under $\widehat{\mathsf{P}}$ has the following characterization.
\begin{itemize}
\item The marginal law of $\gamma^1$ under $\widehat{\mathsf{P}}$ is the same as $\Pcro{1}^{(\kappa; \ell)}(\A_r; \ee^{\ii\alpha_1}, \ee^{\ii\beta_1-r})$ weighted by Radon-Nikodym derivative
\begin{align}\label{eqn::rainbow_RN}
\frac{\LFcro{1}^{(\kappa;\ell)}(\A_r; \ee^{\ii\alpha_1}, \ee^{\ii\beta_1-r})}{\LFcro{n}^{(\kappa;\ell)}(\A_r; \ee^{\ii\bs{\alpha}}, \ee^{\ii\bs{\beta}-r})}\LZ_{\rainbow}^{(\kappa)}(\A_r\setminus\gamma^1; \ee^{\ii\dot{\bs{\alpha}}_1}, \ee^{\ii\dot{\bs{\beta}}_1-r}). 
\end{align} 
In particular, we have
\begin{align}\label{eqn::multiannulus_pf_2nd}
\LFcro{n}^{(\kappa; \ell)}(r; \bs{\alpha}, \bs{\beta})=\LFcro{1}^{(\kappa;\ell)}(r; \alpha_1, \beta_1)\Ecro{1}^{(\kappa;\ell)}\left[\LZ_{\rainbow}^{(\kappa)}(\A_r\setminus\gamma^1; \ee^{\ii\dot{\bs{\alpha}}_1}, \ee^{\ii\dot{\bs{\beta}}_1-r})\right]. 
\end{align}
\item The conditional law of $(\gamma^2, \ldots, \gamma^n)$ given $\gamma^1$ is rainbow $\SLE_{\kappa}$ in $(\A_r\setminus\gamma^1; \ee^{\ii\dot{\bs{\alpha}}_1}, \ee^{\ii\dot{\bs{\beta}}_1-r})$. 
\end{itemize}
\end{lemma}
\begin{proof}
Let us first derive the marginal law of $\gamma^1$ under $\widehat{\mathsf{P}}$. We apply~\eqref{eqn::rainbow_pf_blm_rep} for the rainbow partition function $\LZ_{\rainbow}^{(\kappa)}(\A_r\setminus\gamma^1; \ee^{\ii\dot{\bs{\alpha}}_1}, \ee^{\ii\dot{\bs{\beta}}_1-r})$. Denote by $\PP_{n-1}^{(\kappa)}$ the law of independent chordal $\SLE_{\kappa}$ measure $\otimes_{j=2}^{n}\Ptwo^{(\kappa)}(\A_r\setminus\gamma^1;\ee^{\ii\alpha_j},\ee^{\ii\beta_j-r})$. 
From~\eqref{eqn::rainbow_pf_blm_rep}, we have
\begin{align}\label{eqn::multiannulus_pf_equiv_aux1}
\begin{split}
&\LZ_{\rainbow}^{(\kappa)}(\A_r\setminus \gamma^1;\ee^{\ii\dot{\bs{\alpha}}_1}, \ee^{\ii\dot{\bs{\beta}}_1-r})\\
=&\prod_{j=2}^n \LZtwo(\A_r\setminus\gamma^1; \ee^{\ii\alpha_j}, \ee^{\ii\beta_j-r})\times \E_{n-1}^{(\kappa)}\left[\one_{\LE_{\emptyset}(\bs{\gamma})}\exp\left(\frac{\mathfrak{c}}{2}\blm(\A_r\setminus\gamma^1; \gamma^2, \ldots, \gamma^n)\right)\right]. 
\end{split}
\end{align}
From the boundary perturbation property of annulus $\SLE_{\kappa}(\LFcro{1}^{(\kappa;\ell)})$ in~\eqref{eqn::bp_annulus}, for $2\le j\le n$, the law of chordal $\SLE_{\kappa}$ in $(\A_r\setminus\gamma^1; \ee^{\ii\alpha_j}, \ee^{\ii\beta_j-r})$ is the same as the law of annulus $\SLE_{\kappa}(\LFcro{1}^{(\kappa;\ell)})$ in $(\A_r; \ee^{\ii\alpha_j}, \ee^{\ii\beta_j-r})$ weighted by 
\begin{align}\label{eqn::multiannulus_pf_equiv_aux12}
\frac{\LFcro{1}^{(\kappa;\ell)}(r;\alpha_j,\beta_j)}{\LZtwo(\A_r\setminus\gamma^1;\ee^{\ii \alpha_j},\ee^{\ii \beta_j-r})}  \one\{ \gamma^j\subset \A_r\setminus\gamma^1\} \exp\left(\frac{\mathfrak{c}}{2} \blm(\A_r;\gamma^j, \gamma^1)\right). 
\end{align}
Plugging into~\eqref{eqn::multiannulus_pf_equiv_aux1}, we obtain
\begin{align}\label{eqn::multiannulus_pf_equiv_aux2}
\begin{split}
	&\LZ_{\rainbow}^{(\kappa)}(\A_r\setminus \gamma^1;\ee^{\ii\dot{\bs{\alpha}}_1}, \ee^{\ii\dot{\bs{\beta}}_1-r})\\
	=\;&\prod_{j=2}^{n}\LFcro{1}^{(\kappa;\ell)}(r;\alpha_j,\beta_j)\,
	\times \mathsf{E}_{n-1}^{(\kappa;\ell)}\left[
	\one_{\LE_{\emptyset}(\bs{\gamma})}
	\exp\left(\frac{\mathfrak{c}}{2}\left(\blm(\A_r\setminus \gamma^1;\gamma^2,\ldots,\gamma^n)
	+\sum_{j=2}^{n}\blm(\A_r;\gamma^j,\gamma^1)\right)\right)
	\right],\end{split}
\end{align}
where $\mathsf{P}_{n-1}^{(\kappa;\ell)}$ denotes the law of independent annulus $\SLE_{\kappa}$ measure $\otimes_{j=2}^n\Pcro{1}^{(\kappa;\ell)}(\A_r; \ee^{\ii\alpha_j}, \ee^{\ii\beta_j-r})$. For the Brownian loop measure, on the event $\LE_{\emptyset}(\bs{\gamma})$, we have 
\begin{align}\label{eqn::blm_split_gamma1}
	\blm(\A_r\setminus \gamma^1;\gamma^2,\ldots,\gamma^n)
	+\sum_{j=2}^{n}\blm(\A_r;\gamma^j,\gamma^1)
	=\blm(\A_r;\gamma^1,\ldots,\gamma^n).
\end{align}
Plugging into~\eqref{eqn::multiannulus_pf_equiv_aux2}, we have
\begin{align}\label{eqn::multiannulus_pf_equiv_aux3}
	\LZ_{\rainbow}^{(\kappa)}(\A_r\setminus \gamma^1;\ee^{\ii\dot{\bs{\alpha}}_1}, \ee^{\ii\dot{\bs{\beta}}_1-r})
	=\prod_{j=2}^{n}\LFcro{1}^{(\kappa;\ell)}(r;\alpha_j,\beta_j)\,
	\times \mathsf{E}_{n-1}^{(\kappa;\ell)}\left[
	\one_{\LE_{\emptyset}(\bs{\gamma})}
	\exp\left(\frac{\mathfrak{c}}{2}\blm(\A_r;\gamma^1,\ldots,\gamma^n)\right)
	\right]. 
\end{align}
Integrating over $\gamma^1$, we obtain~\eqref{eqn::multiannulus_pf_2nd} as desired. Moreover, the Radon-Nikodym derivative of the marginal law of $\gamma^1$ under $\widehat{\mathsf{P}}$ with respect to $\Pcro{1}^{(\kappa;\ell)}(\A_r;\ee^{\ii\alpha_1},\ee^{\ii\beta_1-r})$ is given by
\begin{align*}
&\frac{\prod_{j=1}^{n}\LFcro{1}^{(\kappa;\ell)}(r;\alpha_j,\beta_j)}
		{\LFcro{n}^{(\kappa;\ell)}(r;\bs{\alpha},\bs{\beta})}
		\mathsf{E}_{n-1}^{(\kappa;\ell)}\left[
		\one_{\LE_{\emptyset}(\bs{\gamma})}
		\exp\left(\frac{\mathfrak{c}}{2}\blm(\A_r;\gamma^1,\ldots,\gamma^n)\right)
		\right]\\
		=&\frac{\LFcro{1}^{(\kappa;\ell)}(r;\alpha_1,\beta_1)}
		{\LFcro{n}^{(\kappa;\ell)}(r;\bs{\alpha},\bs{\beta})} \LZ_{\rainbow}^{(\kappa)}(\A_r\setminus\gamma^1;\ee^{\ii\dot{\bs{\alpha}}_1},\ee^{\ii\dot{\bs{\beta}}_1-r}). \tag{due to~\eqref{eqn::multiannulus_pf_equiv_aux3}}
\end{align*}
This gives the desired marginal law of $\gamma^1$ under $\widehat{\mathsf{P}}$. 
\medbreak
Next, we derive the conditional law of $(\gamma^2, \ldots, \gamma^n)$ given $\gamma^1$ under $\widehat{\mathsf{P}}$. 
This conditional law is the law of independent annulus $\SLE_{\kappa}$ measure $\otimes_{j=2}^n\Pcro{1}^{(\kappa; \ell)}(\A_r; \ee^{\ii\alpha_j}, \ee^{\ii\beta_j-r})$ weighted by
\begin{equation*}
	\frac{\prod_{j=2}^{n}\LFcro{1}^{(\kappa;\ell)}(r;\alpha_j,\beta_j)}
	{\LZ_{\rainbow}^{(\kappa)}(\A_r \setminus \gamma^1;\ee^{\ii\dot{\bs{\alpha}}_1},\ee^{\ii\dot{\bs{\beta}}_1-r})}
	\one_{\LE_{\emptyset}(\bs{\gamma})}
	\exp\left(\frac{\mathfrak{c}}{2}\blm(\A_r;\gamma^1,\ldots,\gamma^n)\right). 
\end{equation*}
From~\eqref{eqn::multiannulus_pf_equiv_aux12} and~\eqref{eqn::blm_split_gamma1}, this is the same as $\PP_{n-1}^{(\kappa)}$ weighted by 
\[
\frac{\prod_{j=2}^{n}\LZtwo(\A_r \setminus \gamma^1;\ee^{\ii\alpha_j},\ee^{\ii\beta_j-r})}
{\LZ_{\rainbow}^{(\kappa)}(\A_r \setminus \gamma^1;\ee^{\ii\dot{\bs{\alpha}}_1},\ee^{\ii\dot{\bs{\beta}}_1-r})}
\one_{\LE_{\emptyset}(\bs{\gamma})}
\exp\left(\frac{\mathfrak{c}}{2}\blm(\A_r \setminus \gamma^1;\gamma^2,\ldots,\gamma^n)\right).
\]
By~\eqref{eqn::rainbow_blm_rep}, this is exactly the law of rainbow $\SLE_{\kappa}$ in $(\A_r \setminus \gamma^1;\ee^{\ii\dot{\bs{\alpha}}_1},\ee^{\ii\dot{\bs{\beta}}_1-r})$.
\end{proof}

\begin{lemma}\label{lem::rainbow_RN_bounded}
Fix $\kappa\in (0,4]$ and $r>0$. 
Fix $n\ge 1$ and $\bs{\alpha}, \bs{\beta}\in\LX_n$. 
Suppose $\gamma$ is a continuous simple curve in $\A_r$ from $\ee^{\ii\alpha_1}$ to $\ee^{\ii\beta_1-r}$.
We denote $\dot{\bs{\alpha}}_1=(\alpha_2, \ldots, \alpha_n)$ and $\dot{\bs{\beta}}_1=(\beta_n, \ldots, \beta_2)$. 
Then the rainbow partition function of the domain $\A_r\setminus\gamma$ is bounded:
\begin{equation}\label{eqn::rainbow_RN_bounded}
	\LZ_{\rainbow}^{(\kappa)}(\A_r\setminus\gamma; \ee^{\ii\dot{\bs{\alpha}}_1}, \ee^{\ii\dot{\bs{\beta}}_1-r}) \le \left(\ee^{r}\frac{\pi^2}{4r^2}\right)^{(n-1)\mathfrak{b}}.
\end{equation}
\end{lemma}

\begin{proof}
From~\eqref{eqn::rainbow_pf_upperbound}, we have 
\begin{align}\label{eqn::rainbow_RN_aux1}
\LZ_{\rainbow}^{(\kappa)}(\A_r\setminus\gamma; \ee^{\ii\dot{\bs{\alpha}}_1}, \ee^{\ii\dot{\bs{\beta}}_1-r})\le \prod_{j=2}^n \Poisson(\A_r\setminus\gamma; \ee^{\ii\alpha_j}, \ee^{\ii\beta_j-r})^{\mathfrak{b}}.
\end{align} 
Note that $\gamma$ is a continuous simple curve in $\A_r$ from $\ee^{\ii\alpha_1}$ to $\ee^{\ii\beta_1-r}$. 
Suppose $\gamma$ has winding $\ell\in\Z$, i.e. $q^{-1}(\gamma)$ connects $\alpha_1$ to $\beta_1+2\pi\ell+\ii r$. 
For $2\le j\le n$, we have
\begin{align*}
	\Poisson(\A_r\setminus\gamma; \ee^{\ii\alpha_j},\ee^{\ii\beta_j-r})
	= & \ee^{r}\Poisson(q^{-1}(\A_r\setminus\gamma);\alpha_j,\beta_j+2\pi\ell+\ii r) \tag{due to~\eqref{eqn::bPoisson_cov}}\\
	\le & \ee^{r}\Poisson(\S_r;\alpha_j,\beta_j+2\pi\ell+\ii r) \tag{due to~\eqref{eqn::bPoisson_monotone}} \\
	= & \ee^{r}\frac{\pi^2}{4r^2} \left( \cosh \left(\frac{\pi(\beta_j+2\pi\ell-\alpha_j)}{2r}\right) \right)^{-2}. \tag{due to~\eqref{eqn::bPoisson_strip}}
\end{align*}
Thus,
\begin{align}
\Poisson(\A_r\setminus\gamma; \ee^{\ii\alpha_j},\ee^{\ii\beta_j-r})\le & \ee^{r}\frac{\pi^2}{4r^2}. \label{eqn::rainbow_RN_aux2}
\end{align}
Plugging into~\eqref{eqn::rainbow_RN_aux1}, we obtain~\eqref{eqn::rainbow_RN_bounded} as desired. 
\end{proof}

\begin{lemma}\label{lem::earlierconstruction_DMP}
Assume the same setup as in Lemma~\ref{lem::earlierconstruction_marginal_conditional}.
The law of $\bs{\gamma}=(\gamma^1,\ldots,\gamma^n)$ under $\widehat{\mathsf{P}}$ enjoys the domain Markov property: for $t_1,\ldots,t_n<r$, given $(\gamma_{[0,t_1]}^1,\ldots,\gamma_{[0,t_n]}^n)$, the law of $(\gamma_{[t_1,r]}^1,\ldots,\gamma_{[t_n,r]}^n)$ is 
\begin{align*}
	\widehat{\mathsf{P}}(\A_r\setminus \cup_{j=1}^n \gamma_{[0,t_j]}^j ;\gamma_{t_1}^1,\ldots,\gamma_{t_n}^n,\ee^{\ii \bs{\beta}-r}).
\end{align*}
\end{lemma}
\begin{proof}
We first show that the conditional law of $(\gamma_{[t_1,r]}^1,\gamma^2,\ldots,\gamma^n)$ given $\gamma_{[0,t_1]}^1$ is $\widehat{\mathsf{P}}(\A_r\setminus \gamma_{[0,t_1]}^1 ;\gamma_{t_1}^1,\ee^{\ii \dot{\bs{\alpha}}_1},\ee^{\ii \bs{\beta}-r})$. 
We have the following observation:
\begin{itemize}
	\item By the boundary perturbation property~\eqref{eqn::bp_annulus_doubly_connected}, for $2\le j\le n$, we have
	\begin{align} \label{eqn::earlierconstruction_DMP_aux2}
	\begin{split}
		\frac{\ud \Pcro{1}^{(\kappa;\ell)}(\A_r\setminus \gamma^1_{[0,t_1]};\ee^{\ii\alpha_j},\ee^{\ii\beta_j-r})}{\ud \Pcro{1}^{(\kappa;\ell)}(\A_r;\ee^{\ii\alpha_j},\ee^{\ii\beta_j-r})}(\gamma^j) = &\frac{\LFcro{1}^{(\kappa;\ell)}(r; \alpha_j, \beta_j)}{\LFcro{1}^{(\kappa;\ell)}(\A_r\setminus \gamma^1_{[0,t_1]};\ee^{\ii\alpha_j},\ee^{\ii\beta_j-r})}\\
		&\one\{ \gamma^j\subset \A_r\setminus \gamma^1_{[0,t_1]} \}
		\exp\left(\frac{\mathfrak{c}}{2}
		\blm(\A_r;\gamma^j,\gamma^1_{[0,t_1]})\right).
		\end{split}
	\end{align}
	\item We have the following Brownian loop measure decomposition: on the event $\{\gamma^2,\ldots,\gamma^n\subset \A_r\setminus\gamma^1_{[0,t_1]}\}$, we have
	\begin{equation} \label{eqn::earlierconstruction_DMP_aux3}
		\blm(\A_r;\gamma^1_{[0,t_1]}\cup \gamma^1_{[t_1,r]},\gamma^2,\ldots,\gamma^n)
		=\sum_{j=2}^n\blm(\A_r;\gamma^j,\gamma^1_{[0,t_1]})
		+\blm(\A_r\setminus \gamma^1_{[0,t_1]};\gamma^1_{[t_1,r]},\gamma^2,\ldots,\gamma^n).
	\end{equation}
\end{itemize}
Under $\Pind{n}^{(\kappa;\ell)}$, conditionally on $\gamma^1_{[0,t_1]}$, the curve $\gamma^1_{[t_1,r]}$ has law $\Pcro{1}^{(\kappa;\ell)}(\A_r\setminus \gamma^1_{[0,t_1]};\gamma^1_{t_1},\ee^{\ii\beta_1-r})$ by the domain Markov property of the single annulus $\SLE_{\kappa}$, and it is independent of $(\gamma^2,\ldots,\gamma^n)$. Combining with~(\ref{eqn::earlierconstruction},\ref{eqn::earlierconstruction_DMP_aux2},\ref{eqn::earlierconstruction_DMP_aux3}), under $\widehat{\mathsf{P}}(\A_r ;\ee^{\ii \bs{\alpha}},\ee^{\ii \bs{\beta}-r})$, the conditional law of the $(\gamma^1_{[t_1,r]},\gamma^2,\ldots,\gamma^n)$ given $\gamma^1_{[0,t_1]}$ is the same as 
\[
\Pcro{1}^{(\kappa;\ell)}(\A_r\setminus \gamma^1_{[0,t_1]};\gamma^1_{t_1},\ee^{\ii\beta_1-r})
\otimes\left(\otimes_{j=2}^n
\Pcro{1}^{(\kappa;\ell)}(\A_r\setminus \gamma^1_{[0,t_1]};\ee^{\ii\alpha_j},\ee^{\ii\beta_j-r})\right),
\]
weighted by
\[
\one_{\LE_{\emptyset}(\gamma^1,\ldots,\gamma^n)}
\exp\left(\frac{\mathfrak{c}}{2}\blm(\A_r\setminus \gamma^1_{[0,t_1]};\gamma_{[t_1,r]}^1,\gamma^2,\ldots,\gamma^n)\right). 
\]
This confirms that the conditional law of $(\gamma_{[t_1,r]}^1,\gamma^2,\ldots,\gamma^n)$ given $\gamma_{[0,t_1]}^1$ is $\widehat{\mathsf{P}}(\A_r\setminus \gamma_{[0,t_1]}^1 ;\gamma_{t_1}^1,\ee^{\ii \dot{\bs{\alpha}}_1},\ee^{\ii \bs{\beta}-r})$. 
Then, iterating the first step to $\gamma_{[0,t_j]}^j$ in the domain $\A_r\setminus \cup_{i=1}^{j-1} \gamma_{[0,t_i]}^i$ for $2\le j\le n$, we obtain the desired domain Markov property. 
\end{proof}

\subsection{Construction from multi-time martingale: proof of Propositions~\ref{prop::multiannulu_pf_construction} and~\ref{prop::ac_general}}
\label{subsec::multiannulus_pf_proof}

\begin{lemma}\label{lem::multiannulus_pf_BPZ}
The partition function $\LFcro{n}^{(\kappa; \ell)}$ is smooth and it satisfies annulus BPZ equations~\eqref{eqn::annulus_BPZ_alpha} with $F(r)=F_n(r)$ given by~\eqref{eqn::F(r)inBPZ}.
\end{lemma}
\begin{proof}
Assume the same setup as in Lemma~\ref{lem::earlierconstruction_marginal_conditional}. 
The smoothness of $\LFcro{n}^{(\kappa;\ell)}$ is proved in~\cite{JahangoshahiLawlerMultiplepathsSLE, KarrilaViitasaari}. 
Let us show that it satisfies annulus BPZ equations~\eqref{eqn::annulus_BPZ_alpha} with $F(r)=F_n(r)$ given by~\eqref{eqn::F(r)inBPZ}.
By symmetry, it suffices to prove the equation~\eqref{eqn::annulus_BPZ_alpha} for $\alpha_1$.	From~\eqref{eqn::rainbow_pf_polygon_def} and~\eqref{eqn::multiannulus_pf_2nd}, we have
\begin{align} \label{eqn::multiannulus_pf_general_BPZ_aux1}
	N_t:=& \Ecro{1}^{(\kappa;\ell)}\left[ \LZ_{\rainbow}^{(\kappa)}(\A_r\setminus\gamma^1; \ee^{\ii\dot{\bs{\alpha}}_1},\ee^{\ii\dot{\bs{\beta}}_1-r}) \;\middle|\;\gamma^1_{[0,t]}\right] \notag\\
	= & \exp\left((n-1)\mathfrak{b}t\right)
	\prod_{i=2}^{n}
	\left((\mathfrak{\covmap}_t^1)'(\alpha_i)
	(\mathfrak{\covmap}_t^1)'(\beta_i+\ii r)\right)^{\mathfrak{b}} \Ecro{1}^{(\kappa;\ell)}\left[ \LZ_{\rainbow}^{(\kappa)}(\A_{r-t}\setminus\mathfrak{g}_t^1(\gamma^1); \mathfrak{g}_t^1(\ee^{\ii\dot{\bs{\alpha}}_1}),\mathfrak{g}_t^1(\ee^{\ii\dot{\bs{\beta}}_1-r})) \right] \notag\\
	= & \exp\left((n-1)\mathfrak{b}t\right)
	\prod_{i=2}^{n}
	\left((\mathfrak{\covmap}_t^1)'(\alpha_i)
	(\mathfrak{\covmap}_t^1)'(\beta_i+\ii r)\right)^{\mathfrak{b}} \frac{\LFcro{n}^{(\kappa;\ell)}
	\left(r-t;\zeta_t^1,\mathfrak{\covmap}_t^1(\dot{\bs{\alpha}}_1),
	\Re\mathfrak{\covmap}_t^1(\bs{\beta}+\ii r)\right)}
	{\LFcro{1}^{(\kappa;\ell)}(r-t;\zeta_t^1,\Re\mathfrak{\covmap}_t^1(\beta_1+\ii r))}. 
\end{align}
The process $N_t$ 
is a local martingale under $\Pcro{1}^{(\kappa;\ell)}$. 
Applying Proposition~\ref{prop::multitime_mart_annulus} and~\eqref{eqn::multitime_mart_annulus_aux1}, we conclude that $\LFcro{n}^{(\kappa; \ell)}$ satisfies the annulus BPZ equations~\eqref{eqn::annulus_BPZ_alpha} with $F(r)=F_n(r)$ given by~\eqref{eqn::F(r)inBPZ}.
\end{proof}
\begin{definition}\label{def::multiannulus_SLE}
Fix $\kappa\in (0,4]$ and $n\ge 1$ and $\ell\in\Z$. 
Fix $\bs{\alpha}, \bs{\beta}\in\LX_n$. 
For $1\le j\le n$, let $\gamma^j$ be annulus $\SLE_{\kappa}(\LFcro{1}^{(\kappa;\ell)})$ in $(\A_r;\ee^{\ii \alpha_j};\ee^{\ii \beta_j})$. 
Let $\Pind{n}^{(\kappa;\ell)}$ be the probability measure on $\bs{\gamma}=(\gamma^{1}, \ldots, \gamma^{n})$ under which the curves are independent.
As $\LFcro{n}^{(\kappa; \ell)}$ satisfies annulus BPZ equations~\eqref{eqn::annulus_BPZ_alpha} with $F(r)=F_n(r)$ given by~\eqref{eqn::F(r)inBPZ}, 
the process $M_{\bs{t}}(\LFcro{1}^{(\kappa;\ell)}; \LFcro{n}^{(\kappa;\ell)})$ is $n$-time local martingale under $\Pind{n}^{(\kappa;\ell)}$ due to Proposition~\ref{prop::multitime_mart_annulus}. 
We define $n$-annulus $\SLE_{\kappa}(\LFcro{n}^{(\kappa;\ell)})$ in $(\A_r; \ee^{\ii\bs{\alpha}}, \ee^{\ii\bs{\beta}-r})$ as the probability measure obtained by tilting $\Pind{n}^{(\kappa;\ell)}$ by the $n$-time local martingale $M_{\bs{t}}(\LFcro{1}^{(\kappa;\ell)}; \LFcro{n}^{(\kappa;\ell)})$. 
We denote by $\Pcro{n}^{(\kappa; \ell)}=\Pcro{n}^{(\kappa; \ell)}(\A_r; \ee^{\ii\bs{\alpha}}, \ee^{\ii\bs{\beta}})$ the law of $n$-annulus $\SLE_{\kappa}(\LFcro{n}^{(\kappa;\ell)})$. 
\end{definition}

\begin{lemma} \label{lem::nannulus_marginal_gamma1}
Assume the same setup as in Definition~\ref{def::multiannulus_SLE}. 
Suppose $\bs{\gamma}=(\gamma^1, \ldots, \gamma^n)$ is $n$-annulus $\SLE_{\kappa}(\LFcro{n}^{(\kappa; \ell)})$. 
We denote $\dot{\bs{\alpha}}_1=(\alpha_2, \ldots, \alpha_n)$ and $\dot{\bs{\beta}}_1=(\beta_n, \ldots, \beta_2)$. 
Then the marginal law of $\gamma^1$ under $\Pcro{n}^{(\kappa; \ell)}$ is absolutely continuous with respect to $\Pcro{1}^{(\kappa; \ell)}$ in $(\A_r; \ee^{\ii\alpha_1}, \ee^{\ii\beta_1-r})$ with Radon-Nikodym derivative given by~\eqref{eqn::rainbow_RN}. 
In particular, $\gamma^1$ is a continuous transient curve in $\A_r$ from $\ee^{\ii\alpha_1}$ to $\ee^{\ii\beta_1-r}$. 
\end{lemma}
\begin{proof}
By~\eqref{eqn::multitime_mart_annulus}, the marginal law of $\gamma^1_{[0,t]}$ under $\Pcro{n}^{(\kappa; \ell)}$ is the same as  $\Pcro{1}^{(\kappa;\ell)}$ weighted by $N_t/N_0$ where $N_t$ is defined in~\eqref{eqn::multiannulus_pf_general_BPZ_aux1}. Note that the curve is a transient curve from $\ee^{\ii\alpha_1}$ to $\ee^{\ii\beta_1-r}$ under $\Pcro{1}^{(\kappa;\ell)}$ due to Lemma~\ref{lem::singleannulusSLE}, thus
\begin{equation*}
\begin{split}
	N_t/N_0=& \frac{\LFcro{1}^{(\kappa;\ell)}(r;\alpha_1,\beta_1)}{\LFcro{n}^{(\kappa;\ell)}(r;\bs{\alpha},\bs{\beta})} \Ecro{1}^{(\kappa;\ell)}\left[ \LZ_{\rainbow}^{(\kappa)}(\A_r\setminus\gamma^1; \ee^{\ii\dot{\bs{\alpha}}_1},\ee^{\ii\dot{\bs{\beta}}_1-r}) \;\middle|\;\gamma^1_{[0,t]}\right] \\
	\to & \frac{\LFcro{1}^{(\kappa;\ell)}(r;\alpha_1,\beta_1)}{\LFcro{n}^{(\kappa;\ell)}(r;\bs{\alpha},\bs{\beta})} \LZ_{\rainbow}^{(\kappa)}(\A_r\setminus\gamma^1;
	\ee^{\ii\dot{\bs{\alpha}}_1},\ee^{\ii\dot{\bs{\beta}}_1-r}) \quad\text{as }t\to r,	
\end{split}
\end{equation*}
which gives~\eqref{eqn::rainbow_RN} as desired.
\end{proof}

\begin{proof}[Proof of Proposition~\ref{prop::ac_general}]
We assume the same setup as in Lemma~\ref{lem::earlierconstruction_marginal_conditional}. 
It suffices to show that $\widehat{\mathsf{P}}$ is the same as $n$-annulus $\SLE_{\kappa}(\LFcro{n}^{(\kappa; \ell)})$.  
We parameterize $\bs{\gamma}$ by $n$-time parameter $\bs{t}$, and let $\bs{W}_{\bs{t}} = (W_{\bs{t}}^{1},\ldots,W_{\bs{t}}^{n})$ be the multi-slit driving function. 
We write $\bs{\gamma}_{[0,\bs{t}]}=(\gamma^1_{[0,t_1]}, \ldots, \gamma^n_{[0,t_n]})$ and we also write $\bs{\gamma}_{\bs{t}}=\cup_{j=1}^n \gamma^j_{[0,t_j]}$. 
The event that different curves are disjoint is denoted by $\LE_{\emptyset}(\bs{\gamma}_{\bs{t}}) = \{ \gamma_{[0,t_j]}^{j} \cap \gamma_{[0,t_i]}^{i}=\emptyset, \, \forall i\neq j\}$.

We claim that on the event $\LE_{\emptyset}(\bs{\gamma}_{\bs{t}})$, for $1\le i\le n$, the joint law of $(\gamma^1_{[0,t_1]}, \ldots, \gamma^i_{[0,t_i]})$ under $\widehat{\mathsf{P}}$ is the same as it is under $\Pcro{n}^{(\kappa; \ell)}$. The claim holds for $i=1$ due to Lemma~\ref{lem::earlierconstruction_marginal_conditional} and Lemma~\ref{lem::nannulus_marginal_gamma1}. Assume the claim holds for $i$, let us consider the joint law of $(\gamma^1_{[0,t_1]}, \ldots, \gamma^i_{[0,t_i]}, \gamma^{i+1}_{[0,t_{i+1}]})$. From the induction hypothesis, the joint law of $(\gamma^1_{[0,t_1]}, \ldots, \gamma^i_{[0,t_i]})$ under $\widehat{\mathsf{P}}$ is the same as it is under $\Pcro{n}^{(\kappa; \ell)}$. It remains to check the conditional law of $\gamma^{i+1}_{[0,t_{i+1}]}$ given $(\gamma^1_{[0,t_1]}, \ldots, \gamma^i_{[0,t_i]})$. 
\begin{itemize}
\item From the domain Markov property of $\widehat{\mathsf{P}}$ in Lemma~\ref{lem::earlierconstruction_DMP}, the conditional law of $\gamma^{i+1}_{[0,t_{i+1}]}$ given $(\gamma^1_{[0,t_1]}, \ldots, \gamma^i_{[0,t_i]})$ under $\widehat{\mathsf{P}}(\A_r; \ee^{\ii\bs{\alpha}}, \ee^{\ii\bs{\beta}-r})$ is the same as the law of $\gamma^{i+1}_{[0,t_{i+1}]}$ under \[\widehat{\mathsf{P}}(\A_r\setminus \cup_{j=1}^i \gamma_{[0,t_j]}^j;\gamma_{t_1}^1,\ldots,\gamma_{t_i}^i,\ee^{\ii\alpha_{i+1}},\ldots,\ee^{\ii\alpha_{n}},\ee^{\ii\bs{\beta}-r}).\]
\item From the domain Markov property of $\Pcro{n}^{(\kappa; \ell)}$, the conditional law of $\gamma^{i+1}_{[0,t_{i+1}]}$ given $(\gamma^1_{[0,t_1]}, \ldots, \gamma^i_{[0,t_i]})$ under $\Pcro{n}^{(\kappa; \ell)}(\A_r; \ee^{\ii\bs{\alpha}}, \ee^{\ii\bs{\beta}-r})$ is the same as the law of $\gamma^{i+1}_{[0,t_{i+1}]}$ under \[\Pcro{n}^{(\kappa; \ell)}(\A_r\setminus \cup_{j=1}^i \gamma_{[0,t_j]}^j;\gamma_{t_1}^1,\ldots,\gamma_{t_i}^i,\ee^{\ii\alpha_{i+1}},\ldots,\ee^{\ii\alpha_{n}},\ee^{\ii\bs{\beta}-r}).\]
\end{itemize}
Combining with Lemma~\ref{lem::earlierconstruction_marginal_conditional} and Lemma~\ref{lem::nannulus_marginal_gamma1}, the above two conditional laws are the same. This completes the proof the claim.

From the above analysis, on the event $\LE_{\emptyset}(\bs{\gamma}_{\bs{t}})$, the joint law of $\bs{\gamma}_{[0,\bs{t}]}$ under $\widehat{\mathsf{P}}$ is the same as it is under $\Pcro{n}^{(\kappa; \ell)}$. The conclusion will follow once we know $\widehat{\mathsf{P}}[\LE_{\emptyset}(\bs{\gamma})]=1$ and $\Pcro{n}^{(\kappa; \ell)}[\LE_{\emptyset}(\bs{\gamma})]=1$.
The fact $\widehat{\mathsf{P}}[\LE_{\emptyset}(\bs{\gamma})]=1$ follows from the construction of $\widehat{\mathsf{P}}$ in Lemma~\ref{lem::earlierconstruction_marginal_conditional}. It remains to check $\Pcro{n}^{(\kappa; \ell)}[\LE_{\emptyset}(\bs{\gamma})]=1$.
From Lemma~\ref{lem::nannulus_marginal_gamma1}, the curve $\gamma^1$ under $\Pcro{n}^{(\kappa; \ell)}$ is a continuous simple curve in $\A_r$ from $\ee^{\ii\alpha_1}$ to $\ee^{\ii\beta_1-r}$ and it does not hit any other point on the boundary. For $\bs{t}=(t_1, \ldots, t_n)$, given $\bs{\gamma}_{\bs{t}}$ and on the event $\LE_{\emptyset}(\bs{\gamma}_{\bs{t}})$, from the domain Markov property of $\Pcro{n}^{(\kappa; \ell)}$ and Lemma~\ref{lem::nannulus_marginal_gamma1}, the curve segment $\gamma^i_{[t_i, r]}$ is a continuous simple curve in $\A_r\setminus\bs{\gamma}_{\bs{t}}$ from $\gamma^i_{t_i}$ to $\ee^{\ii\beta_i-r}$ and it does not hit any other point on the boundary. In particular, the curve segment $\gamma^i_{[t_i, r]}$ does not hit $\cup_{j}\gamma^j_{[0,t_j]}$ except at the starting point $\gamma^i_{t_i}$.  This holds for all $\bs{t}$ and thus $\Pcro{n}^{(\kappa; \ell)}[\LE_{\emptyset}(\bs{\gamma})]=1$ as desired. This completes the proof. 
\end{proof}

\begin{proof}[Proof of Proposition~\ref{prop::multiannulu_pf_construction}]
The annulus BPZ equations are checked in Lemma~\ref{lem::multiannulus_pf_BPZ}. It remains to derive the bound in~\eqref{eqn::multiannulus_pf_upperbound}. 
Fix $\kappa\in (0,4]$ and $n\ge 2$ and $\ell\in\Z$. 
Fix $r>0$ and $\bs{\alpha}, \bs{\beta}\in\LX_n$. 
Suppose $(\gamma^1, \gamma^2)$ is annulus $\SLE_{\kappa}(\LFcro{2}^{(\kappa;\ell)})$ in $(\A_r; \ee^{\ii\alpha_1}, \ee^{\ii\alpha_2}, \ee^{\ii\beta_1-r}, \ee^{\ii\beta_2-r})$ and we denote $\ddot{\bs{\alpha}}_{12}=(\alpha_3, \ldots, \alpha_n)$ and $\ddot{\bs{\beta}}_{12}=(\beta_n, \ldots, \beta_3)$. 
We claim that
\begin{align}\label{eqn::multiannulus_pf_3rd_aux1}
\LFcro{n}^{(\kappa; \ell)}(r; \bs{\alpha}, \bs{\beta})=\LFcro{2}^{(\kappa;\ell)}(r; \alpha_1, \alpha_2, \beta_1, \beta_2)\Ecro{2}^{(\kappa;\ell)}\left[\LZ_{\rainbow}^{(\kappa)}(\A_r\setminus(\gamma^1\cup\gamma^2); \ee^{\ii\ddot{\bs{\alpha}}_{12}}, \ee^{\ii\ddot{\bs{\beta}}_{12}-r})\right]. 
\end{align}
Let us check the two terms in RHS of~\eqref{eqn::multiannulus_pf_3rd_aux1}. 
\begin{itemize}
\item Using a similar idea as in the proof of Lemma~\ref{lem::earlierconstruction_marginal_conditional}, by applying~\eqref{eqn::rainbow_pf_blm_rep} and~\eqref{eqn::bp_annulus} for the rainbow partition function,
we obtain
\begin{align*}
&\LZ_{\rainbow}^{(\kappa)}(\A_r\setminus(\gamma^1\cup\gamma^2); \ee^{\ii\ddot{\bs{\alpha}}_{12}}, \ee^{\ii\ddot{\bs{\beta}}_{12}-r})\\
	=\;& \prod_{j=3}^{n}\LFcro{1}^{(\kappa;\ell)}(r;\alpha_j,\beta_j)\,
	\times \mathsf{E}_{n-2}^{(\kappa;\ell)}\left[
	\one\{(\gamma^1\cup\gamma^2)\cap\gamma^j=\emptyset, \forall 3\le j\le n\}
	\exp\left(\frac{\mathfrak{c}}{2}\blm(\A_r; (\gamma^1\cup\gamma^2), \gamma^3,\ldots,\gamma^n)\right)
	\right],
\end{align*}
where $\mathsf{P}_{n-2}^{(\kappa;\ell)}$ denotes the law of independent annulus $\SLE_{\kappa}$ measure $\otimes_{j=3}^n \Pcro{1}^{(\kappa; \ell)}(\A_r; \ee^{\ii\alpha_j}, \ee^{\ii\beta_j-r})$. Integrating over $(\gamma^1, \gamma^2)$, we have
\begin{align}
&\Ecro{2}^{(\kappa; \ell)}\left[\LZ_{\rainbow}^{(\kappa)}(\A_r\setminus(\gamma^1\cup\gamma^2); \ee^{\ii\ddot{\bs{\alpha}}_{12}}, \ee^{\ii\ddot{\bs{\beta}}_{12}-r})\right]\label{eqn::multiannulus_pf_3rd_aux2}\\
=&\prod_{j=3}^{n}\LFcro{1}^{(\kappa;\ell)}(r;\alpha_j,\beta_j)\,
	\notag\\
	&\times \Ecro{2}^{(\kappa;\ell)}\left[\mathsf{E}_{n-2}^{(\kappa;\ell)}\left[
	\one\{(\gamma^1\cup\gamma^2)\cap\gamma^j=\emptyset, \forall 3\le j\le n\}
	\exp\left(\frac{\mathfrak{c}}{2}\blm(\A_r; (\gamma^1\cup\gamma^2), \gamma^3,\ldots,\gamma^n)\right)
	\right]\right]. \notag
\end{align}
\item 
Applying Proposition~\ref{prop::ac_general} for $(\gamma^1, \gamma^2)$, the law of $(\gamma^1, \gamma^2)$ under $\Pcro{2}^{(\kappa; \ell)}$ is the same as the law of $\Pind{2}^{(\kappa; \ell)}$ weighted by the Radon-Nikodym derivative 
\begin{align}\label{eqn::multiannulus_pf_3rd_aux3}
\frac{\prod_{j=1}^2\LFcro{1}^{(\kappa;\ell)}(r; \alpha_j, \beta_j)}{\LFcro{2}^{(\kappa;\ell)}(r; \alpha_1, \alpha_2, \beta_1, \beta_2)}\one\{\gamma^1\cap\gamma^2=\emptyset\}\exp\left(\frac{\mathfrak{c}}{2}\blm(\A_r; \gamma^1, \gamma^2)\right). 
\end{align}
\end{itemize}
Combining~\eqref{eqn::multiannulus_pf_3rd_aux2} and~\eqref{eqn::multiannulus_pf_3rd_aux3}, we have
\begin{align*}
\text{RHS of~\eqref{eqn::multiannulus_pf_3rd_aux1}}
=&\prod_{j=1}^{n}\LFcro{1}^{(\kappa;\ell)}(r;\alpha_j,\beta_j)\,
	\times \Eind{n}^{(\kappa;\ell)}\left[
	\one_{\LE_{\emptyset}(\bs{\gamma})}
	\exp\left(\frac{\mathfrak{c}}{2}\blm(\A_r; \gamma^1, \gamma^2)+\frac{\mathfrak{c}}{2}\blm(\A_r; (\gamma^1\cup\gamma^2), \gamma^3,\ldots,\gamma^n)\right)
	\right]\\
	=&\prod_{j=1}^{n}\LFcro{1}^{(\kappa;\ell)}(r;\alpha_j,\beta_j)\,
	\times \Eind{n}^{(\kappa;\ell)}\left[
	\one_{\LE_{\emptyset}(\bs{\gamma})}
	\exp\left(\frac{\mathfrak{c}}{2}\blm(\A_r; \gamma^1, \ldots, \gamma^n)\right)
	\right]=\text{LHS of~\eqref{eqn::multiannulus_pf_3rd_aux1}}.
\end{align*}
This completes the proof of~\eqref{eqn::multiannulus_pf_3rd_aux1}. 
\medbreak
Next, let us check~\eqref{eqn::multiannulus_pf_upperbound}. We have the following two estimates. 
\begin{itemize}
\item From~\cite[Theorem~4]{JahangoshahiLawlerMultiplepathsSLE}, there exists $C\in (0, \infty)$ independent of $(r; \bs{\alpha}, \bs{\beta})$ such that 
\begin{align*}
\LFcro{2}^{(\kappa)}(r; \alpha_1, \alpha_2, \beta_1, \beta_2)\le C\exp\left(\left(2\mathfrak{b}-\frac{(8-\kappa)}{8}\right)r+\frac{\mathfrak{c}}{2}\log r\right). 
\end{align*}
\item For the rainbow partition function, from~\eqref{eqn::rainbow_pf_upperbound} and~\eqref{eqn::rainbow_RN_aux2}, we have
\begin{align*}
\LZ_{\rainbow}^{(\kappa)}(\A_r\setminus(\gamma^1\cup\gamma^2); \ee^{\ii\ddot{\bs{\alpha}}_{12}}, \ee^{\ii\ddot{\bs{\beta}}_{12}-r})\le &\prod_{j=3}^n \Poisson(\A_r\setminus\gamma; \ee^{\ii\alpha_j}, \ee^{\ii\beta_j-r})^{\mathfrak{b}}
\le  \left(\ee^{r}\frac{\pi^2}{4r^2}\right)^{(n-2)\mathfrak{b}}.
\end{align*}
\end{itemize}
Plugging them into~\eqref{eqn::multiannulus_pf_3rd_aux1} and~\eqref{eqn::multiannulus_withoutspiral_def}, we have
\begin{align*}
\LFcro{n}^{(\kappa)}(r; \bs{\alpha}, \bs{\beta})=&\sum_{\ell\in\Z}\LFcro{n}^{(\kappa;\ell)}(r; \bs{\alpha}, \bs{\beta})\\
\le& \sum_{\ell\in\Z} \LFcro{2}^{(\kappa;\ell)}(r; \alpha_1, \alpha_2, \beta_1, \beta_2)\times \left(\ee^{r}\frac{\pi^2}{4r^2}\right)^{(n-2)\mathfrak{b}}\\
=& \LFcro{2}^{(\kappa)}(r; \alpha_1, \alpha_2, \beta_1, \beta_2)\times \left(\ee^{r}\frac{\pi^2}{4r^2}\right)^{(n-2)\mathfrak{b}}\\
\le & C\exp\left(\left(2\mathfrak{b}-\frac{(8-\kappa)}{8}\right)r+\frac{\mathfrak{c}}{2}\log r\right)\times \left(\ee^{r}\frac{\pi^2}{4r^2}\right)^{(n-2)\mathfrak{b}}. 
\end{align*}
This gives~\eqref{eqn::multiannulus_pf_upperbound} as desired. 
\end{proof}

\begin{figure}[ht!]
\begin{subfigure}[t]{0.2\textwidth}
\begin{center}
\includegraphics[width=\textwidth]{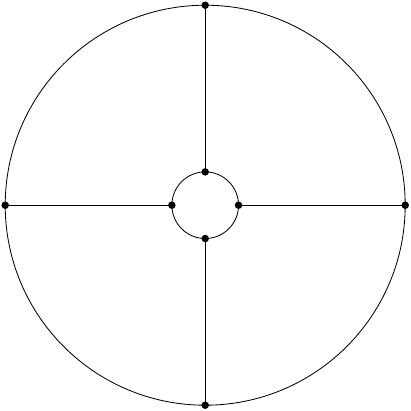}
\end{center}
\caption{}
\end{subfigure}
$\quad$
\begin{subfigure}[t]{0.2\textwidth}
\begin{center}
\includegraphics[width=\textwidth]{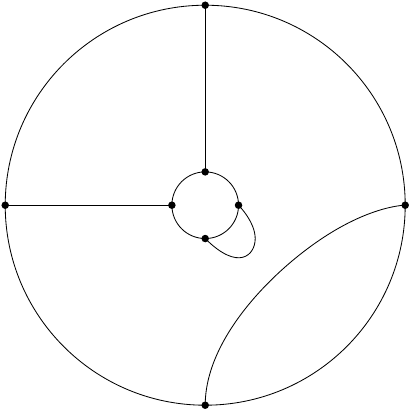}
\end{center}
\caption{}
\end{subfigure}
$\quad$
\begin{subfigure}[t]{0.2\textwidth}
\begin{center}
\includegraphics[width=\textwidth]{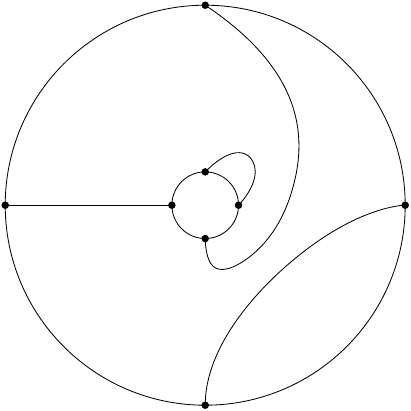}
\end{center}
\caption{}
\end{subfigure}\\
\begin{subfigure}[t]{0.2\textwidth}
\begin{center}
\includegraphics[width=\textwidth]{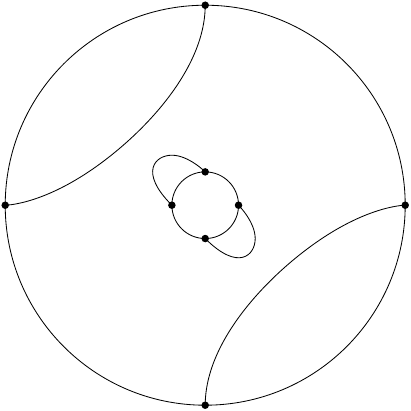}
\end{center}
\caption{}
\end{subfigure}
$\quad$
\begin{subfigure}[t]{0.2\textwidth}
\begin{center}
\includegraphics[width=\textwidth]{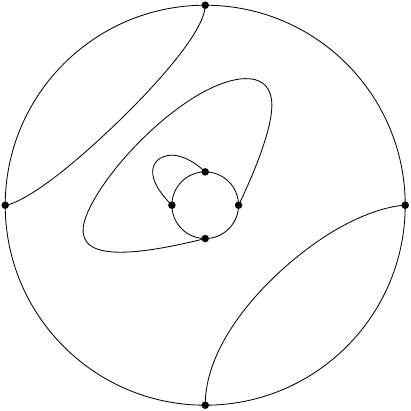}
\end{center}
\caption{}
\end{subfigure}
$\quad$
\begin{subfigure}[t]{0.2\textwidth}
\begin{center}
\includegraphics[width=\textwidth]{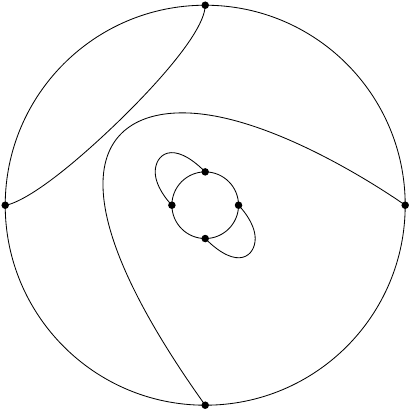}
\end{center}
\caption{}
\end{subfigure}
$\quad$
\begin{subfigure}[t]{0.2\textwidth}
\begin{center}
\includegraphics[width=\textwidth]{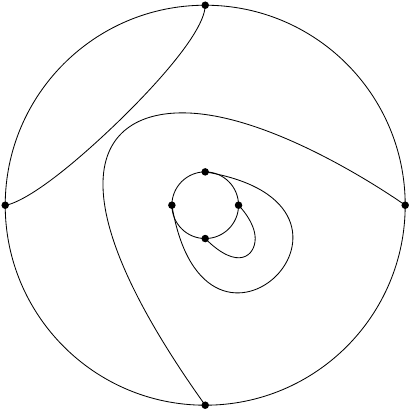}
\end{center}
\caption{}
\end{subfigure}
\caption{\label{fig::levellines_annulus_patterns} Various connectivity patterns for four curves in the annulus.}
\end{figure}

\begin{remark}\label{rem::annulus_otherpatterns}
The authors in~\cite{ZhanRestrictionAnnulusSLE, JahangoshahiLawlerMultiplepathsSLE} introduce multiple SLE in annulus by weighting the law of independent annulus SLEs by Brownian loop terms, as in~\eqref{eqn::ac_RN}. 
In their construction, all curves cross the annulus. 
The same idea can also be used to define multiple SLE in annulus with other connectivity patterns, see Figure~\ref{fig::levellines_annulus_patterns}. In particular, all resulting partition functions satisfy the annulus BPZ system~\eqref{eqn::annulus_BPZ_alpha}-\eqref{eqn::annulus_BPZ_beta} with $F(r)=F_n(r)$ defined in~\eqref{eqn::F(r)inBPZ}.  
Therefore, all of them generate multi-time martingales as in Proposition~\ref{prop::multitime_mart_annulus}. It is an interesting question to investigate all possible connectivity patterns in annulus and relate them to the basis of the solution space of annulus BPZ systems. 
\end{remark}

\section{Proof of Propositions~\ref{prop::annulusBPZ_kappa4} and~\ref{prop::two_pf_equal}}
\label{sec::multiannulu_SLE4}

\subsection{Proof of Proposition~\ref{prop::annulusBPZ_kappa4}}
\label{subsec::annulusBPZ}
\begin{proof}[Proof of Proposition~\ref{prop::annulusBPZ_kappa4}]
Let us fix $\kappa=4$. It suffices to show that $\LZann{n}$ defined in~\eqref{eqn::LZann_def} and $\LZcro{n}^{(0)}$ defined in~\eqref{eqn::LZcroell_def} with $\ell=0$ both satisfy the annulus BPZ equation~(\ref{eqn::annulus_BPZ_alpha}) for $\LD_{\alpha_1}^{(n)}$ with $\kappa=4$: 
\begin{align}\label{eqn::LZannLZcro_BPZ_alpha1}
(\partial_r-F_n(r))\LZ=\LD_{\alpha_1}^{(n)}\LZ, \qquad \text{where }F_n(r)=\frac{1}{2r} + \frac{3}{4} \LE(r) + \frac{n}{4},
\end{align}
and $\LE(t)$ is defined in~\eqref{eqn::Def_LE}. 
Let $\eps=(\eps_1, \ldots, \eps_n)\in \{\pm 1\}^n$, we set 
\begin{align*}
\LZ_{\eps}(r; \bs{\alpha}, \bs{\beta})=&\prod_{1\le i<j\le n}|\Theta_1(r; \alpha_j-\alpha_i)|^{\frac{1}{2}\eps_i\eps_j}|\Theta_1(r; \beta_j-\beta_i)|^{\frac{1}{2}\eps_i\eps_j}  \\
	& \times\prod_{1\le i, j\le n}|\Theta_3(r; \beta_j-\alpha_i)|^{-\frac{1}{2}\eps_i\eps_j} \times \exp\left(-\frac{1}{8r}\left(\sum_{j=1}^n\eps_j(\alpha_j-\beta_j)\right)^2\right). 
\end{align*}
Then, we write 
\begin{align}\label{eqn::LZannLZcro_BPZ_aux1}
\begin{split}
	\LZann{n}(r; \bs{\alpha}, \bs{\beta})=&\LZ_{\pm}(r; \bs{\alpha}, \bs{\beta}) \times \exp\left(\frac{nr}{8}\right)\times \prod_{k=1}^{\infty}(1-\ee^{-2kr})^{\frac{3n}{2}},\\
	\LZcro{n}^{(0)}(r; \bs{\alpha}, \bs{\beta})=&\LZ_{+}(r; \bs{\alpha}, \bs{\beta}) \times \exp\left(\frac{nr}{8}\right)\times \prod_{k=1}^{\infty}(1-\ee^{-2kr})^{\frac{3n}{2}},
	\end{split}
\end{align}
where $\LZ_{\pm}$ is the same as $\LZ_{\eps}$ when $\eps_j=(-1)^{j-1}$ for $1\le j\le n$; and that $\LZ_+$ is the same as $\LZ_{\eps}$ when $\eps_j=1$ for $1\le j\le n$. 
Let us consider the two parts in~\eqref{eqn::LZannLZcro_BPZ_aux1}. 
\begin{itemize}
\item On the one hand, for $\eps=(\eps_1, \ldots, \eps_n)\in\{\pm 1\}^n$, using~\eqref{eqn::H1H3_def}, we have 
\begin{align}\label{eqn::LZannLZcro_BPZ_partZeps}
\left(\partial_r-\LD_{\alpha_1}^{(n)}\right)\LZ_{\eps}=\left(\frac{1}{2r}+G_{\eps}(r; \bs{\alpha}, \bs{\beta})\right)\LZ_{\eps}, 
\end{align}
with
\begin{align}\label{eqn::Geps_def}
&G_{\eps}(r; \bs{\alpha}, \bs{\beta})\notag\\
=& - \sum_{i,j=1}^{n} \eps_i \eps_{j} \left( \frac{1}{4} \partial_z H_3(r;\beta_{j}-\alpha_i) + \frac{1}{8} H_3(r;\beta_{j}-\alpha_i)^2 \right)  \notag\\
& + \sum_{1\le i<j\le n} \eps_i \eps_{j} \left( \frac{1}{4} \partial_z H_1(r;\alpha_{j}-\alpha_i) + \frac{1}{8} H_1(r;\alpha_{j}-\alpha_i)^2 + \frac{1}{4} \partial_z H_1(r;\beta_{j}-\beta_i) + \frac{1}{8} H_1(r;\beta_{j}-\beta_i)^2 \right) \notag\\
& - \frac{1}{8} \left( \sum_{i=1}^{n} \eps_i H_3(r;\alpha_1-\beta_i) - \sum_{i=2}^{n} \eps_i H_1(r;\alpha_1-\alpha_i) \right)^2 \notag\\
& + \frac{1}{2} \left( \sum_{i=1}^{n} \eps_i \partial_z H_3(r;\alpha_1-\beta_i) - \sum_{i=2}^{n} \eps_i \partial_z H_1(r;\alpha_1-\alpha_i) \right) \notag\\
& - \frac{1}{4} \sum_{j=2}^{n} \partial_z H_1(r;\alpha_j-\alpha_1) - \frac{1}{4} \sum_{j=1}^{n} \partial_z H_3(r;\beta_j-\alpha_1) \notag\\
& + \frac{1}{4} \sum_{j=2}^{n}  H_1(r;\alpha_j-\alpha_1) \left( \sum_{i=1}^{n} \eps_i \eps_j H_3(r;\alpha_j-\beta_i) - \sum_{i\neq j} \eps_i \eps_j H_1(r;\alpha_j-\alpha_i) \right) \notag\\
& + \frac{1}{4} \sum_{j=1}^{n} H_3(r;\beta_j-\alpha_1) \left( \sum_{i=1}^{n} \eps_i \eps_j H_3(r;\beta_j-\alpha_i) - \sum_{i\neq j} \eps_i \eps_j H_1(r;\beta_j-\beta_i) \right).
\end{align}
We denote $G_{\eps}$ by $G_{\pm}$ when $\eps_j=(-1)^{j-1}$ for $1\le j\le n$; and denote $G_{\eps}$ by $G_{+}$ when $\eps_j=1$ for $1\le j\le n$. 
\item On the other hand, the $r$-dependent factor in~\eqref{eqn::LZannLZcro_BPZ_aux1} satisfies
\begin{equation} \label{eqn::LZannLZcro_BPZ_partr}
	\partial_r \log \left(\exp\left(\frac{nr}{8}\right)\times \prod_{k=1}^{\infty}(1-\ee^{-2kr})^{\frac{3n}{2}} \right) = \frac{n}{8} + \frac{3n}{2} \sum_{k=1}^{\infty} \frac{k \ee^{-2kr}}{1-\ee^{-2kr}} = \frac{3n}{4} \LE(r) + \frac{n}{4}.
\end{equation}
\end{itemize}
Plugging~\eqref{eqn::LZannLZcro_BPZ_partZeps} and~\eqref{eqn::LZannLZcro_BPZ_partr} into~\eqref{eqn::LZannLZcro_BPZ_alpha1}, it remains to show 
\begin{align}\label{eqn::LZannLZcro_BPZ_Geps_goal}
G_{\pm}(r; \bs{\alpha}, \bs{\beta})=G_{+}(r; \bs{\alpha}, \bs{\beta})=-\frac{3}{4}(n-1)\LE(r).  
\end{align}
\medbreak
Next, let us check~\eqref{eqn::LZannLZcro_BPZ_Geps_goal}. 
Although the coordinates $\bs{\alpha}$ and $\bs{\beta}$ are originally defined in the real configuration space $\mathcal{X}_n$, the expression for $G_{\eps}$ is composed of the meromorphic functions $H_1(r;\cdot)$ and $H_3(r;\cdot)$. Thus, we can treat $G_{\eps}$ as a meromorphic function on $\mathbb{C}^{2n}$. This allows us to utilize the theory of elliptic functions. We have the following observations for $G_{\eps}$. 
\begin{itemize}
\item For $\eps\in\{\pm, +\}$, the function $G_{\eps}(r; \cdot, \cdot)$ is double-periodic:
\begin{equation}\label{eqn::Geps_doubleperiodic}
\begin{split}
	G_{\eps}(r;\bs{\alpha},\bs{\beta})= & G_{\eps}(r;\bs{\alpha}+2\pi e_j, \bs{\beta})=G_{\eps}(r;\bs{\alpha},\bs{\beta}+2\pi e_j), \\
	G_{\eps}(r;\bs{\alpha},\bs{\beta})= & G_{\eps}(r;\bs{\alpha}+2\ii r e_j, \bs{\beta})=G_{\eps}(r;\bs{\alpha}, \bs{\beta}+2\ii r e_j),
\end{split}
\end{equation}
where $e_j=(0,\ldots,0,1,0,\ldots,0)$ is the vector whose $j$-th component is $1$. 
Combining~(\ref{eqn::Theta_period},\ref{eqn::Theta_shift},\ref{eqn::H1H3}), we have the following identities for $H_1$ and $H_3$:
\begin{equation} \label{eqn::H_shift_identities}
\begin{split}
	&H_1(r; z + 2\ii r) = H_1(r; z) - 2\ii, \quad H_3(r; z + 2\ii r) = H_3(r; z) - 2\ii,\\ 
	&H_1(r; z + 2\pi) = H_1(r; z), \quad H_3(r; z + 2\pi) = H_3(r; z). 
\end{split}
\end{equation}
The double-periodicity~\eqref{eqn::Geps_doubleperiodic} follows from~\eqref{eqn::H_shift_identities}, see details in Lemma~\ref{lem::Geps_doubleperiodic}. 

\item For $\eps\in\{\pm, +\}$, the function $G_{\eps}(r; \cdot, \cdot)$ does not have any poles. Note that~(\ref{eqn::JacobiTheta_productexpansion},\ref{eqn::H1H3}) implies $H_1(r;\cdot)$ is a meromorphic function in $\C$ with poles $2\pi\mathbb{Z} + 2\ii r\mathbb{Z}$ and $H_3(r;\cdot)$ is a meromorphic function in $\C$ with poles $2\pi\mathbb{Z} + (2\mathbb{Z}+1)\ii r$. By the double-periodicity of $G_{\eps}$ in~\eqref{eqn::Geps_doubleperiodic}, it suffices to investigate potential singularities of $G_{\eps}$ when $\alpha_i \to \alpha_j$ or $\alpha_i \to \beta_j + \ii r$.  Combining~(\ref{eqn::Theta_shift},\ref{eqn::H1H3}), we have
\begin{equation} \label{eqn::H1H3_shift}
	H_3(r; z \pm \ii r) = H_1(r; z) \mp \ii.
\end{equation}
As $\alpha_i\to \alpha_j$ ($i\neq j$), plugging~\eqref{eqn::Laurrent_expand_LE} into~\eqref{eqn::Geps_def}, we find the $(\alpha_i-\alpha_j)^{-2}$ and $(\alpha_i-\alpha_j)^{-1}$ singular terms in $G_{\eps}$ cancel out. As $\alpha_i\to \beta_j+ \ii r$, plugging~(\ref{eqn::Laurrent_expand_LE},\ref{eqn::H1H3_shift}) into~\eqref{eqn::Geps_def}, we find the $(\alpha_i-\beta_j\ii r)^{-2}$ and $(\alpha_i-\beta_j\ii r)^{-1}$ singular terms in $G_{\eps}$ cancel out. See details in Lemma~\ref{lem::Geps_nopoles}. 
Therefore $G_{\eps}(r;\bs{\alpha},\bs{\beta})$ does not have poles when $(\bs{\alpha},\bs{\beta})\in \C^{2n}$.
\end{itemize}
For $\eps\in\{\pm, +\}$, since $G_{\eps}$ is a doubly periodic function with no poles in its fundamental domain, by Liouville's Theorem, $G_{\eps}$ must be a constant depending only on $r$ and $n$.
Plugging~\eqref{eqn::Laurrent_expand_LE} into~\eqref{eqn::Geps_def} with $\alpha_j = x_j \delta$ and $\beta_j = y_j \delta$. By evaluating the constant term in the expansion of $G_{\eps}$ as $\delta\to 0$, we obtain~\eqref{eqn::LZannLZcro_BPZ_Geps_goal}. This completes the proof. 
\end{proof}

\subsection{Proof of Proposition~\ref{prop::two_pf_equal}}
\label{subsec::LZcron_ac}

The proof of Proposition~\ref{prop::two_pf_equal} relies on the following lemma.
\begin{lemma}\label{lem::multiannulusSLE4_pf_2nd}
Fix $n\ge 1$ and $r>0$ and $\bs{\alpha}, \bs{\beta}\in\LX_n$. Suppose $\gamma^1$ is annulus $\SLE_4(\LZcro{1}^{(\ell)})$ in $(\A_r; \ee^{\ii\alpha_1}, \ee^{\ii\beta_1-r})$ and denote its law by $\Pcro{1}^{(\ell)}$. 
We denote $\dot{\bs{\alpha}}_1=(\alpha_2, \ldots, \alpha_n)$ and $\dot{\bs{\beta}}_1=(\beta_n, \ldots, \beta_2)$.
Then we have
\begin{align}\label{eqn::multiannulus_SLE4_pf_rainbow}
\LZcro{n}^{(\ell)}(r; \bs{\alpha}, \bs{\beta})=\LZcro{1}^{(\ell)}(r; \alpha_1, \beta_1)
\Ecro{1}^{(\ell)}\left[\LZ_{\rainbow}^{(4)}(\A_r\setminus\gamma^1; \ee^{\ii\dot{\bs{\alpha}}_1}, \ee^{\ii\dot{\bs{\beta}}_1-r})\right].
\end{align}
\end{lemma}
To prove Lemma~\ref{lem::multiannulusSLE4_pf_2nd}, we first address the rainbow partition functions in RHS of~\eqref{eqn::multiannulus_SLE4_pf_rainbow}. 
The rainbow partition functions are usually defined as solutions to chordal BPZ equations as in Section~\ref{subsec::polyon_pre}, and they do not have explicit formula in general.
However, when $\kappa=4$, they do enjoy explicit formula. They are derived in~\cite[Theorem~1.5]{PeltolaWuGlobalMultipleSLEs}: for $2N$-polygon $(\Omega; \bs{x})$ with $\bs{x}=(x_1, \ldots, x_{2N})$, 
\begin{align}\label{eqn::rainbow_pf4}
\LZ_{\rainbow}^{(4)}(\Omega; {\bs{x}})=\prod_{j=1}^{N} \Poisson(\Omega;x_{j},x_{2N+1-j})^{\frac{1}{4}}\times \prod_{1\le i<k\le N}\left(\frac{\Poisson(\Omega; x_i, x_{2N+1-k})\Poisson(\Omega; x_k, x_{2N+1-i})}{\Poisson(\Omega; x_i, x_k)\Poisson(\Omega; x_{2N+1-i}, x_{2N+1-k})}\right)^{\frac{1}{4}}. 
\end{align}
Moreover, the upper bound~\eqref{eqn::rainbow_pf_upperbound} becomes
\begin{equation}\label{eqn::rainbow_pf4_bound}
\LZ_{\rainbow}^{(4)}(\Omega; {\bs{x}})\le \prod_{j=1}^{N} \Poisson(\Omega;x_{j},x_{2N+1-j})^{\frac{1}{4}}.
\end{equation}

\begin{proof}[Proof of Lemma~\ref{lem::multiannulusSLE4_pf_2nd}]
It suffices to show the conclusion for $\ell=0$. 
We need to compare the two partition functions: $\LZ_{\rainbow}^{(4)}$ defined in~\eqref{eqn::rainbow_pf4} and $\LZcro{n}^{(0)}$ defined in~\eqref{eqn::LZcroell_def}. The building blocks for $\LZ_{\rainbow}^{(4)}$ are boundary Poisson kernels and the building blocks for $\LZcro{n}^{(0)}$ are Jacobi theta functions. 
To compare the two partition functions $\LZ_{\rainbow}^{(4)}$ and $\LZcro{n}^{(0)}$, we use the equivalent expression for the rescaled Jacobi theta functions $\Theta_1, \Theta_3$ in Lemma~\ref{lem::JacobiTheta_technical}. 
In particular, we use the relation between the boundary Poisson kernel and the rescaled Jacobi theta functions~\eqref{eqn::LZcro_LZann_mono_aux2} and 
for $1\le j\le n$, 
\begin{align}\label{eqn::LZcro_LZann_mono_aux1}
	|\Theta_3(r; \beta_j-\alpha_j)|^{-1}=&\Poisson(\S_r; \alpha_j, \beta_j+\ii r)^{\frac{1}{2}} \times\left(\frac{\pi}{r}\right)^{-\frac{3}{2}}\ee^{\frac{(\beta_j-\alpha_j)^2}{4r}} \ee^{\frac{\pi^2}{4r}}\notag\\
	&\times\left(\prod_{m=1}^\infty\left(1-\ee^{\frac{-2m\pi^2}{r}}\right)\left(1+\ee^{\frac{-2m\pi^2}{r}-\frac{\pi (\beta_j - \alpha_j)}{r}}\right)\left(1+\ee^{\frac{-2m\pi^2}{r}+\frac{\pi (\beta_j - \alpha_j)}{r}}\right)\right)^{-1}. 
\end{align}

Recall that $\Pcro{n}^{(0)}$ denotes the law of $n$-annulus $\SLE_4(\LZcro{n}^{(0)})$. The marginal law of $\gamma^1$ under $\Pcro{n}^{(0)}$ is the same as tilting $\Pcro{1}^{(0)}$ by the local martingale
\begin{align*}
N_t=\exp\left((n-1)\frac{1}{4}t\right)
	\prod_{i=2}^{n}
	\left((\mathfrak{\covmap}_t^1)'(\alpha_i)
	(\mathfrak{\covmap}_t^1)'(\beta_i+\ii r)\right)^{\frac{1}{4}} \frac{\LZcro{n}^{(0)}
	\left(r-t;\zeta_t^1,\mathfrak{\covmap}_t^1(\dot{\bs{\alpha}}_1),
	\Re\mathfrak{\covmap}_t^1(\bs{\beta}+\ii r)\right)}
	{\LZcro{1}^{(0)}(r-t;\zeta_t^1,\Re\mathfrak{\covmap}_t^1(\beta_1+\ii r))}. 
\end{align*}
Define $D_t:=\A_r\setminus \gamma^1[0,t]$ and define $\Omega_t:=\S_r\setminus q^{-1}(\gamma^1[0,t])$. Denote by $\Omega_r$ the connected component of $\S_r\setminus q^{-1}(\gamma^1)$ having $\dot{\bs{\alpha}}_1$ and $\dot{\bs{\beta}}_1+\ii r$ on its boundary. For $t\in[0,r)$ and $z\in\C$, define
\begin{equation} \label{eqn::F(t,z)G(t,z)}
	F(t,z):=\prod_{m=1}^\infty\left(1-\ee^{\frac{-2m\pi^2}{r-t}+\frac{\pi z}{r-t}}\right)\quad\text{and}\quad G(t,z):=F(t,-z).
\end{equation}
From~\eqref{eqn::LZcroell_def},~\eqref{eqn::rainbow_pf4},~\eqref{eqn::eta},~\eqref{eqn::LZcro_LZann_mono_aux1} and~\eqref{eqn::LZcro_LZann_mono_aux2}, we have
\begin{align}\label{eqn::multiannulusSLE4_pf_aux4}
N_t=&\ee^{\frac{(n-1)r}{4}}\LZ_{\rainbow}^{(4)}(\Omega_t; \dot{\bs{\alpha}}_1,\dot{\bs{\beta}}_1+\ii r)\times F(t,0)^{(n-1)}\notag\\
&\times\underbrace{\left(\frac{\prod_{2\le i<j\le n}F(t,\mathfrak{\covmap}_t(\alpha_j)-\mathfrak{\covmap}_t(\alpha_i))\prod_{2\le j\le n}F(t,\mathfrak{\covmap}_t(\alpha_j)-\zeta_t)\prod_{1\le i<j\le n}F(t,\mathfrak{\covmap}_t(\beta_j+\ii r)-\mathfrak{\covmap}_t(\beta_i+\ii r))}{\prod_{2\le j\le n}F(t,\mathfrak{\covmap}_t(\beta_j+\ii r)-\zeta_t)\prod_{2\le i\le n,1\le j\le n}F(t,\mathfrak{\covmap}_t(\beta_j+\ii r)-\mathfrak{\covmap}_t(\alpha_i))}\right)^{\frac{1}{2}}}_{:=R_{1,t}}\notag\\
&\times\underbrace{\left(\frac{\prod_{2\le i<j\le n}G(t,\mathfrak{\covmap}_t(\alpha_j)-\mathfrak{\covmap}_t(\alpha_i))\prod_{2\le j\le n}G(t,\mathfrak{\covmap}_t(\alpha_j)-\zeta_t)\prod_{1\le i<j\le n}G(t,\mathfrak{\covmap}_t(\beta_j+\ii r)-\mathfrak{\covmap}_t(\beta_i+\ii r))}{\prod_{2\le j\le n}G(t,\mathfrak{\covmap}_t(\beta_j+\ii r)-\zeta_t)\prod_{2\le i\le n,1\le j\le n}G(t,\mathfrak{\covmap}_t(\beta_j+\ii r)-\mathfrak{\covmap}_t(\alpha_i))}\right)^{\frac{1}{2}}}_{:=R_{2,t}}\notag\\
&\times \prod_{2\le j\le n}\underbrace{\left|\frac{\left(1-\ee^{\frac{-\pi(\mathfrak{\covmap}_t(\alpha_j)-\zeta_t)}{r-t}}\right)\left(1-\ee^{\frac{-\pi(\mathfrak{\covmap}_t(\beta_j+\ii r)-\mathfrak{\covmap}_t(\beta_1+\ii r))}{r-t}}\right)}{\left(1+\ee^{\frac{-\pi(\Re\mathfrak{\covmap}_t(\beta_j+\ii r)-\zeta_t)}{r-t}}\right)\left(1+\ee^{\frac{-\pi(\mathfrak{\covmap}_t(\alpha_j)-\Re\mathfrak{\covmap}_t(\beta_1+\ii r))}{r-t}}\right)}\right|}_{:=L_{j,t}}.
\end{align}

We first prove $\{N_t\}_{t\in [0,r)}$ is uniformly bounded. By definition, we have
\begin{equation}\label{eqn::bound_aux1}
F(t,0)\le 1,\quad R_{1,t}\le 1,\quad R_{2,t}\le 1,\quad L_{j,t}\le 1\quad\text{for }2\le j\le n.
\end{equation}
By~\eqref{eqn::rainbow_pf4_bound} and the monotonicity~\eqref{eqn::bPoisson_monotone}, we have
\begin{equation}\label{eqn::bound_aux2}
\LZ_{\rainbow}^{(4)}(\Omega_t; \dot{\bs{\alpha}}_1,\dot{\bs{\beta}}_1+\ii r)\le \prod_{2\le j\le n}\Poisson(\S_r; \alpha_j, \beta_j+\ii r)^{\frac{1}{4}}.
\end{equation}
Combining~\eqref{eqn::bound_aux1} and~\eqref{eqn::bound_aux2} together, we have that $\{N_t\}_{t\in [0,r)}$ is uniformly bounded.
\medbreak
It remains to derive the terminal value $\lim_{t\to r}N_t$. 
We will prove in Lemma~\ref{lem::RNcontrol_aux} that 
\begin{align}\label{eqn::terminal_alpha}
&\lim_{t\to r}\frac{\mathfrak{\covmap}_t(\alpha_2)-\zeta_t}{r-t}=\infty,\qquad \lim_{t\to r}\frac{\zeta_t+2\pi-\mathfrak{\covmap}_t(\alpha_n)}{r-t}=\infty,\\
&\lim_{t\to r}\frac{\mathfrak{\covmap}_t(\beta_{2} + \ii r)-\mathfrak{\covmap}_t(\beta_{1} + \ii r)}{r-t}=\infty,\qquad\lim_{t\to r}\frac{\mathfrak{\covmap}_t(\beta_{1} + \ii r)+2\pi-\mathfrak{\covmap}_t(\beta_{n} + \ii r)}{r-t}=\infty,\label{eqn::terminal_beta}
\end{align}
and there exists a constant $C=C(\gamma^1)>0$, such that for every $2 \le i,j\le n$,
\begin{equation}\label{eqn::finite}
\mathfrak{\covmap}_t(\beta_{n} + \ii r)-\mathfrak{\covmap}_t(\beta_{2} + \ii r)\le C(r-t),\quad\quad\mathfrak{\covmap}_t(\alpha_{n})-\mathfrak{\covmap}_t(\alpha_{2})\le C(r-t), \quad 
\quad |\mathfrak{\covmap}_t(\beta_{j} + \ii r)-\mathfrak{\covmap}_t(\alpha_{i})|\le C(r-t).
\end{equation}
Assuming these, let us show that
\begin{align}\label{eqn::terminal_aux0}
\lim_{t\to r}N_t=\LZ_{\rainbow}^{(4)}(\A_r\setminus\gamma^1; \ee^{\ii\dot{\bs{\alpha}}_1}, \ee^{\ii\dot{\bs{\beta}}_1-r}).
\end{align}
Let check all terms in RHS of~\eqref{eqn::multiannulusSLE4_pf_aux4}. 
\begin{itemize}
\item By the convergence of boundary Poisson kernel and~\eqref{eqn::rainbow_pf4}, we have
\begin{equation*}
\lim_{t\to r}\LZ_{\rainbow}^{(4)}(\Omega_t; \dot{\bs{\alpha}}_1,\dot{\bs{\beta}}_1+\ii r)=\LZ_{\rainbow}^{(4)}(\Omega_r; \dot{\bs{\alpha}}_1,\dot{\bs{\beta}}_1+\ii r).
\end{equation*}
\item For the deterministic term, we have
\begin{equation*}
\lim_{t\to r}F(t,0)=1.
\end{equation*}
\item We argue that $\lim_{t\to r}R_{1,t}=1$. 
\begin{itemize}
\item
By~\eqref{eqn::terminal_alpha}, we have
\begin{equation*}
\lim_{t\to r}F(t,\mathfrak{\covmap}_t(\alpha_j)-\zeta_t)=1,\qquad\text{ for }2\le j\le n.
\end{equation*}
\item
By~\eqref{eqn::terminal_beta}, we have
\begin{equation*}
\lim_{t\to r}F(t,\mathfrak{\covmap}_t(\beta_j+\ii r)-\mathfrak{\covmap}_t(\beta_1+\ii r))=1,\qquad\text{ for }2\le j\le n.
\end{equation*}
\item
By~\eqref{eqn::finite}, for $2\le i<j\le n$, we have
\begin{equation*}
\begin{split}
&\lim_{t\to r}F(t,\mathfrak{\covmap}_t(\alpha_j)-\mathfrak{\covmap}_t(\alpha_i))=1,\qquad \lim_{t\to r}F(t,\mathfrak{\covmap}_t(\beta_j+\ii r)-\mathfrak{\covmap}_t(\beta_i+\ii r))=1,\\
&\lim_{t\to r}F(t,\mathfrak{\covmap}_t(\beta_j+\ii r)-\mathfrak{\covmap}_t(\alpha_i))=1.
\end{split}
\end{equation*}
\item
By~\eqref{eqn::terminal_beta} and~\eqref{eqn::finite}, we have
\[\lim_{t\to r}\frac{\zeta_t+2\pi-\Re\mathfrak{\covmap}_t(\beta_j+\ii r)}{r-t}=\infty,\qquad\text{for }2\le j\le n.\]
This implies that
\begin{equation*}
\lim_{t\to r}F(t,\mathfrak{\covmap}_t(\beta_j+\ii r)-\zeta_t)=1,\qquad\text{ for }2\le j\le n.
\end{equation*}
Similarly, we have
\begin{equation*}
\lim_{t\to r}F(t,\mathfrak{\covmap}_t(\beta_1+\ii r)-\mathfrak{\covmap}_t(\alpha_i))=1,\qquad\text{ for }2\le i\le n.
\end{equation*}
\end{itemize}
Combining the observations above, we obtain $\lim_{t\to r}R_{1,t}=1$. 
\item Similarly, we can check $\lim_{t\to r}R_{2,t}=1$ and $\lim_{t\to r}L_{j,t}= 1$ for $2\le j\le n$. 
\end{itemize}
Collecting the above limits, we obtain~\eqref{eqn::terminal_aux0} as desired. 
The uniform boundedness and the terminal value~\eqref{eqn::terminal_aux0} together gives~\eqref{eqn::multiannulus_SLE4_pf_rainbow} and completes the proof.
\end{proof}

\begin{lemma}\label{lem::RNcontrol_aux}
Assume the same setup as in the proof of Lemma~\ref{lem::multiannulusSLE4_pf_2nd}, the limits~\eqref{eqn::terminal_alpha},~\eqref{eqn::terminal_beta} and~\eqref{eqn::finite} hold. 
\end{lemma}

\begin{proof}
Denote $\gamma^1$ by $\gamma$ for simplicity. Let $\ell_1$ and $\ell_2$ be the lift of $\gamma$ starting from $\alpha_1$ and  $\alpha_1+2\pi$ respectively. Denote by $\Omega_r$ the connected component of $\S_r\setminus q^{-1}(\gamma)$ having $\ell_1$ and $\ell_2$ on its boundary. Fix $z\in \Omega_r$. We first prove that there exists a constant $c=c(\gamma)\in(0,1)$, such that for every $t\in [0,r)$,
\begin{equation}\label{eqn::heightratio}
\frac{\Im\mathfrak{\covmap}_t(z)}{r-t}\in (c,1-c).
\end{equation}
Denote by $\text{hm}(z;\cdot)$ the harmonic measure seen from $z$ in $\Omega_r$. It is clear that there exists a constant $a=a(\gamma)>0$, such that for every $2\le j\le n-1$, 
\begin{align}\label{eqn::harmonic_control}
\begin{split}
&\text{hm}(z;\ell_1\cup(\alpha_1,\alpha_2))\ge a,\qquad \text{hm}(z;(\alpha_n,\alpha_1+2\pi)\cup\ell_2)\ge a, \\
&\text{hm}(z;(\alpha_j,\alpha_{j+1}))\ge a,\qquad\text{hm}(z;(\beta_j + \ii r,\beta_{j+1} + \ii r))\ge a.
\end{split}
\end{align}
Without loss of generality, we may assume that $\Im\mathfrak{\covmap}_t(z)<\frac{r-t}{2}$. By Beurling estimate, there exists a constant $M>0$, such that
\begin{align*}
\text{hm}(z;(\beta_j + \ii r,\beta_{j+1} + \ii r))\le \PP[\LB_1\text{ hits }\partial B(\mathfrak{\covmap}_t(z),r-t-\Im\mathfrak{\covmap}_t(z))\text{ before hitting }\R]\le M\sqrt{\frac{\Im\mathfrak{\covmap}_t(z)}{r-t-\Im\mathfrak{\covmap}_t(z)}},
\end{align*}
where $\LB_1$ is the Brownian motion starting from $\mathfrak{\covmap}_t(z)$. Combining with~\eqref{eqn::harmonic_control}, we obtain~\eqref{eqn::heightratio}. 
\medbreak
We next prove that
\begin{equation}\label{eqn::inftyaux2}
\lim_{t\to r}\left|\frac{\mathfrak{\covmap}_t(z)-\zeta_t}{r-t}\right|=\infty.
\end{equation}
Let $\LB_2$ be a Brownian motion starting from $z$. We have
\begin{equation*}
	\PP[\LB_2\text{ hits }\partial \Omega_r\text{ at }\ell_1[r-t,r]]\ge \PP[\LB_1\text{ hits }\partial\S_{r-t}\text{ at }(-\infty,\zeta_t)]=
	\frac{1}{\pi}\left(\arg\left(\ee^{\frac{\pi\mathfrak{\covmap}_t(z)}{r-t}}-\ee^{\frac{\pi\zeta_t}{r-t}}\right)-\arg\left(\ee^{\frac{\pi\mathfrak{\covmap}_t(z)}{r-t}}\right)\right).
\end{equation*}
Letting $t\to r$, combining~\eqref{eqn::heightratio}, we obtain~\eqref{eqn::inftyaux2}.
Similarly, we also obtain
\begin{equation}\label{eqn::inftyaux1}
\lim_{t\to r}\left|\frac{\zeta_t+2\pi-\mathfrak{\covmap}_t(z)}{r-t}\right|=\infty,\qquad \lim_{t\to r}\left|\frac{\mathfrak{\covmap}_t(z)-\mathfrak{\covmap}_t(\beta_{1}+ \ii r)}{r-t}\right|=\infty,\qquad\lim_{t\to r}\left|\frac{\mathfrak{\covmap}_t(\beta_{1}+ \ii r)+2\pi-\mathfrak{\covmap}_t(z)}{r-t}\right|=\infty.
\end{equation}

\medbreak
Now, we will conclude the proof of~\eqref{eqn::terminal_alpha}-\eqref{eqn::terminal_beta} and~\eqref{eqn::finite} by showing that there exists a constant $C=C(\gamma)$, such that for $2\le j\le n$,
\begin{equation}\label{eqn::harm_control_aux1}
\left|\frac{\mathfrak{\covmap}_t(z)-\mathfrak{\covmap}_t(\alpha_j)}{r-t}\right|\le C\quad\text{ and }\quad\left|\frac{\mathfrak{\covmap}_t(z)-\mathfrak{\covmap}_t(\beta_j+ \ii r)}{r-t}\right|\le C.
\end{equation}
We will only prove the above estimate for $\alpha_j$ for $2\le j\le n$ and the other case is similar. 
We have
\begin{align*}
	\text{hm}(z;\ell_1[0,t]\cup(\alpha_1,\alpha_j))
	&\le \PP[\LB_1\text{ hits }\partial\S_{r-t}\text{ at }(\zeta_t,\mathfrak{\covmap}_t(\alpha_j))]
	= \frac{1}{\pi} \arg \left(  \frac{\ee^{\frac{\pi\mathfrak{\covmap}_t(z)}{r-t}}-\ee^{\frac{\pi\mathfrak{\covmap}_t(\alpha_j)}{r-t}}}{\ee^{\frac{\pi\mathfrak{\covmap}_t(z)}{r-t}}-\ee^{\frac{\pi\zeta_t}{r-t}}} \right), 
\end{align*}
and
\begin{align*}
	\text{hm}(z;(\alpha_j,\alpha_1+2\pi)\cup\ell_2[0,t])
	&\le \PP[\LB_1\text{ hits }\partial\S_{r-t}\text{ at }(\mathfrak{\covmap}_t(\alpha_j),\zeta_t+2\pi)]
	= \frac{1}{\pi} \arg \left(  \frac{\ee^{\frac{\pi\mathfrak{\covmap}_t(z)}{r-t}}-\ee^{\frac{\pi(\zeta_t+2\pi)}{r-t}}}{\ee^{\frac{\pi\mathfrak{\covmap}_t(z)}{r-t}}-\ee^{\frac{\pi\mathfrak{\covmap}_t(\alpha_j)}{r-t}}} \right).
\end{align*}
Combining with~\eqref{eqn::heightratio},~\eqref{eqn::harmonic_control},~\eqref{eqn::inftyaux2} and~\eqref{eqn::inftyaux1}, we obtain~\eqref{eqn::harm_control_aux1} and completes the proof.  
\end{proof}

\begin{proof}[Proof of Proposition~\ref{prop::two_pf_equal}]
The conclusion follows from~\eqref{eqn::single_annulusSLE4_pf} and~\eqref{eqn::multiannulus_pf_2nd} and Lemma~\ref{lem::multiannulusSLE4_pf_2nd}. 
\end{proof}
\section{Level lines of GFF in the annulus}
\label{sec::GFF_levellines_annulus}
\paragraph*{Green's function.}
For a domain $D\subsetneq \mathbb{C}$ and $z\in D$, the Green's function $w\mapsto\Green_D(z, w)$ is the unique function such that $\Green_D(z,w)+\log|z-w|$ is harmonic with respect to $w$ throughout the domain $D$, including at $z$, and that is zero on the boundary $\partial D$. 
For $z_1,z_2\in \S_r$, the Green's function in $\A_r$ is given by
\begin{equation} \label{eqn::multiannulus_levelline_aux11}
	\Green_{\A_r}(q(z_1),q(z_2))= - \log \left| \frac{\Theta_1(r;z_1-z_2)}{\Theta_1(r;z_1-\overline{z}_2)} \right| - \frac{\Im(z_1) \Im(z_2)}{r}. 
\end{equation}
\paragraph*{Gaussian free field (GFF).}
For a domain $D\subsetneq \mathbb{C}$ and two functions $ f, g \in L^2(D)$, we denote by $(f, g)$ their inner product in $L^2(D)$. 
We denote by $H_s(D)$ the space of real-valued smooth functions which are compactly supported in $D$. This space has a Dirichlet inner product defined by
\[
(f, g)_{\nabla} := \frac{1}{2\pi} \int_D \nabla f(z) \cdot \nabla g(z)d^2z.
\]
We denote by $H(D)$ the Hilbert space completion of \( H_s(D) \) with respect to the Dirichlet inner product.

The Dirichlet GFF on $D$ is a random sum of the form $\Gamma = \sum_{j=1}^\infty f_j X_j$, where $X_j$ are i.i.d. standard normal random variables and $ (f_j)_{j \geq 0} $ an orthonormal basis for $H(D)$. This sum almost surely diverges within $H(D)$; however, it does converge almost surely in the space of distributions. The limiting value as a function of \( g \) is almost surely a continuous functional on \( H_s(D) \). See e.g.~\cite{SheffieldGFFMath} and~\cite{BerestyckiPowellGFFandLQG} for more details.  
In general, for any harmonic function $\varphi$ on $D$, we define the GFF with boundary data $\varphi$ by the Dirichlet GFF plus $\varphi$. 

In this section, we first describe the law of the level lines of GFF in annulus $\A_r$ with alternating boundary data~\eqref{eqn::boundarydata}. 
Fix $r>0$ and an even number $n=2N$ and $\bs{\alpha}, \bs{\beta}\in\LX_n$. 
For $1\le j\le n$, let $\gamma^j$ be annulus $\SLE_{4}(\LZcro{1})$ in $(\A_r;\ee^{\ii \alpha_j};\ee^{\ii \beta_j-r})$.
Let $\Pind{n}(\LZcro{1})$ be the probability measure on $\bs{\gamma}=(\gamma^{1}, \ldots, \gamma^{n})$ under which the curves are independent.
We parameterize $\bs{\gamma}$ by $n$-time parameter $\bs{t}$, and let $\bs{W}_{\bs{t}} = (W_{\bs{t}}^{1},\ldots,W_{\bs{t}}^{n})$ be the multi-slit driving function (it is unique up to rotation). Using the same notation as in Proposition~\ref{prop::multitime_mart_annulus} and fixing $\kappa=4$, we define 
\begin{align}\label{eqn::multitime_mart_kappa4}
\begin{split}
	M_{\bs{t}}(\LZcro{1}; \LZann{n})=&\one_{\LE_{\emptyset}(\bs{\gamma}_{\bs{t}})} \, \exp\bigg(\frac{1}{2}\blm_{r,\bs{t}} + \frac{1}{4} \sum_{j=1}^{n} (r-t_j) - \frac{n}{4} \Mod(\A_r\setminus \bs{\gamma}_{[\bs{0},\bs{t}]}) \bigg) 
	\; \times \prod_{j=1}^{n} \mathfrak{\covmap}_{\bs{t},j}'(\zeta_{t_j}^{j})^{\frac{1}{4}}  \\
	& \times \prod_{j=1}^{n} \frac{\mathfrak{\covmap}_{\bs{t}}'(\beta_j+\ii r)^{\frac{1}{4}}}{(\mathfrak{\covmap}_{t_j}^j)'( \beta_j + \ii r)^{\frac{1}{4}}}  \times \frac{\LZann{n} (\Mod(\A_r\setminus \bs{\gamma}_{[\bs{0},\bs{t}]});\bs{W}_{\bs{t}},\Re \mathfrak{\covmap}_{\bs{t}} (\bs{\beta}+\ii r)) }{\prod_{j=1}^{n} \LZcro{1} (r-t_j;\zeta_{t_j}^j,\Re \mathfrak{\covmap}_{t_j}^j(\beta_j+\ii r))}.
\end{split}
\end{align}
As $\LZann{n}$ satisfies annulus BPZ equations~\eqref{eqn::annulus_BPZ_kappa4_alpha}, the process $M_{\bs{t}}$ is a multi-time local martingale, due to Proposition~\ref{prop::multitime_mart_annulus}. Our first conclusion in this section is about the law of the level lines of the GFF in annulus. 

\begin{theorem}\label{thm::GFF_levellines_annulus}
Fix $r>0$ and an even number $n=2N$. 
For $\bs{\alpha}, \bs{\beta}\in\LX_n$, we write $\bs{x}=\ee^{\ii\bs{\alpha}}$ and $\bs{y}=\ee^{\ii\bs{\beta}-r}$. 
Suppose $\Gamma$ is Dirichlet GFF in annulus $\A_r$. For $j\in \{1, \ldots, n\}$, let $\gamma^j$ be the level line of $\Gamma+\varphiann{n}$ starting from $x_j$, where $\varphiann{n}$ is the bounded harmonic function in annulus $\A_r$ with alternating boundary data~\eqref{eqn::boundarydata}. 
\begin{itemize}
\item The level lines $\bs{\gamma}=(\gamma^1, \ldots, \gamma^n)$ are continuous curves and they are deterministic functions of the field $\Gamma$ almost surely. 
\item We parameterize $\bs{\gamma}=(\gamma^1, \ldots, \gamma^n)$ by the $n$-time parameter. For $\bs{t}=(t_1, \ldots, t_n)$ such that $t_j<r$ for all $j$ and that $\gamma_{[0,t_i]}^i\cap \gamma_{[0,t_j]}^j=\emptyset$ for $i\neq j$, the law of $\bs{\gamma}_{[\bs{0},\bs{t}]}$ is the same as tilting $\Pind{n}(\LZcro{1})$ by the $n$-time local martingale $M_{\bs{t}}(\LZcro{1};\LZann{n})$ defined in~\eqref{eqn::multitime_mart_kappa4}. 
\end{itemize}
\end{theorem}

The first part of Theorem~\ref{thm::GFF_levellines_annulus} is a special case of~\cite[Proposition~3.18]{AruLupuSepulvedaFPSGFF}. We will prove the second part of Theorem~\ref{thm::GFF_levellines_annulus} in Section~\ref{subsec::GFF_levellines_annulus}. 
We then complete the proof of Theorem~\ref{thm::crossingproba} and of Proposition~\ref{prop::asymptotics} in Section~\ref{subsec::GFF_crossingproba}.


\subsection{GFF and multi-annulus SLE$_4$: proof of Theorem~\ref{thm::GFF_levellines_annulus}}
\label{subsec::GFF_levellines_annulus}
\begin{lemma} \label{lem::multiannulus_levelline}
Fix $r>0$ and an even number $n=2N$. 
For $\bs{\alpha}, \bs{\beta}\in\LX_n$, we write $\bs{x}=\ee^{\ii\bs{\alpha}}$ and $\bs{y}=\ee^{\ii\bs{\beta}-r}$.  
We denote by $\Pcro{1}$ the law of annulus $\SLE_4(\LZcro{1})$ in $(\A_r; x_1, y_1)$. 
Suppose $\gamma^1$ is the annulus Loewner chain whose law is the same as tilting $\Pcro{1}$ by the local martingale
\begin{align}\label{eqn::singletime_mart_kappa4}
M_t(\LZcro{1}; \LZann{n})=\ee^{\frac{n-1}{4}t}
	\prod_{j=2}^{n}
	\left(\mathfrak{\covmap}_t'(\alpha_j)
	\mathfrak{\covmap}_t'(\beta_j+\ii r)\right)^{\frac14}
	\frac{\LZann{n}\left(r-t;\zeta_t,\mathfrak{\covmap}_t(\alpha_2),
		\ldots,\mathfrak{\covmap}_t(\alpha_n),
		\Re\mathfrak{\covmap}_t(\bs{\beta}+\ii r)\right)}
	{\LZcro{1}\left(r-t;\zeta_t,
		\Re\mathfrak{\covmap}_t(\beta_1+\ii r)\right)},
\end{align}
where $t<T$ and $T$ is the life-time of $\gamma^1$, i.e. the minimum between $r$ and swallowing time of points $\{x_2, x_n\}$.
\begin{itemize}
\item The driving function $\zeta:[0,r)\to \R$ for $\gamma^1$ satisfies the SDE system
\begin{equation}\label{eqn::GFF_ann_SDE}
	\begin{cases}
		\displaystyle \ud \zeta_t = 2 \ud B_t + 4 \partial_{\alpha_1} \log \LZann{n}(r-t; \zeta_t, U_t^2,\ldots, U_t^n,V_t^1,\ldots, V_t^n) \ud t, \qquad \zeta_0=\alpha_1, \\
		\displaystyle \ud U_t^j = H_1(r-t; U_t^j-\zeta_t)\ud t, \qquad U_0^j=\alpha_j,\quad 2\le j\le n,\\
		\displaystyle \ud V_t^j = H_3(r-t; V_t^j-\zeta_t)\ud t, \qquad V_0^j=\beta_j,\quad 1\le j\le n,
	\end{cases}
\end{equation}
where $\{B_t\}_{t\ge 0}$ is a standard Brownian motion.
\item The Dirichlet GFF $\Gamma$ in annulus $\A_r$ can be coupled with $\gamma^1$ so that $\gamma^1$ is a  level line of $\Gamma+\varphiann{n}$ starting from $x_1$.
\end{itemize}
\end{lemma}
\begin{proof}
Let us prove the first item.
By Proposition~\ref{prop::multitime_mart_annulus}, the process	$M_t(\LZcro{1};\LZann{n})$ in~\eqref{eqn::singletime_mart_kappa4} is a local martingale under $\Pcro{1}$, because $\LZann{n}$ satisfies annulus BPZ equations~\eqref{eqn::annulus_BPZ_kappa4_alpha}. Set
\[
	U_t^j:=\mathfrak{\covmap}_t(\alpha_j),\quad 2\le j\le n,
	\quad
	V_t^j:=\Re\mathfrak{\covmap}_t(\beta_j+\ii r),\quad 1\le j\le n, \quad \bs{V}_t=(V_t^1,\ldots,V_t^n).
\] 
By~\eqref{eqn::multitime_mart_annulus_aux1}, we have
\begin{align*}
\begin{split}
	\frac{\ud M_t(\LZcro{1};\LZann{n})}{M_t(\LZcro{1};\LZann{n})}
	=2\left( \partial_{\alpha_1} \log \LZann{n} (r-t;\zeta_t,U_t^2,\ldots,U_t^n,V_t^1,\ldots,V_t^n) - \partial_\alpha \log \LZcro{1}(r-t;\zeta_t,V_t^1) \right) \ud B_t^1,
\end{split}
\end{align*}
where $B_t^1$ is a standard Brownian motion under $\Pcro{1}$.
Girsanov's theorem shows that under the measure obtained by tilting $\Pcro{1}$ by $M_t(\LZcro{1};\LZann{n})$, we have
\begin{align*}
	\ud\zeta_t =2\,\ud B_t +4\,\partial_{\alpha_1}\log\LZann{n}	(r-t;\zeta_t,U_t^2,\ldots,U_t^n,V_t^1,\ldots,V_t^n)\,\ud t,
\end{align*}
where $B_t$ is a standard Brownian motion under the measure obtained by tilting $\Pcro{1}$ by $M_t(\LZcro{1};\LZann{n})$. This finishes the proof of the first item.
\medbreak
Next, we prove the second item. Let us derive a martingale process for $\gamma^1$.
Denote the driving function of $\gamma^1$ by $\zeta_t$, the mapping-out function of $\gamma^1$ by $\mathfrak{g}_t$, and the covering map of $\mathfrak{g}_t$ by $\mathfrak{\covmap}_t$. Let $\eps_j=(-1)^j$. For $r>0$, $\bs{\alpha},\bs{\beta}\in \LX_n$, $z\in \S_r$, we define
\begin{equation*} 
	v(r;\bs{\alpha},\bs{\beta};z)=\sum_{j=1}^{2N} \eps_j \arg \Theta_1(r; z-\alpha_j) - \sum_{j=1}^{2N} \eps_j \arg \Theta_3(r; z-\beta_j) - \sum_{j=1}^{2N} \eps_j (\alpha_j-\beta_j) \frac{\Im(z)}{2r},
\end{equation*}
which is the harmonic function derived in~\eqref{eqn::varphiann_energy_harmonic_function}. 
We define
\begin{equation*}  
	M_t(z)=v(r-t;\zeta_t^1,\mathfrak{\covmap}_t(\alpha_2),\ldots,\mathfrak{\covmap}_t(\alpha_n),\Re \mathfrak{\covmap}_t(\bs{\beta}+\ii r);\mathfrak{\covmap}_t(z)). 
\end{equation*}
We will show in Lemma~\ref{lem::multiannulus_levelline_aux10} that
\begin{equation} \label{eqn::multiannulus_levelline_aux10}
	\ud M_t(z) = \left( \Im H_1(r-t;\mathfrak{\covmap}_t(z)-\zeta_t^1) + \frac{\Im \mathfrak{\covmap}_t(z)}{r-t} \right) \ud B_t,
\end{equation}
which implies that $M_t(z)$ is a local martingale. Moreover, the quadratic variation for the local martingale $M_t$ is 
\begin{equation} \label{eqn::multiannulus_levelline_aux12}
	\ud \langle M_t(z_1),M_t(z_2) \rangle=-\ud \Green_{\A_{r-t}}(\mathfrak{g}_{t}(q(z_1)),\mathfrak{g}_{t}(q(z_2))).
\end{equation}
Assuming~\eqref{eqn::multiannulus_levelline_aux10} and~\eqref{eqn::multiannulus_levelline_aux12}, let us derive the coupling between the GFF $\Gamma_{\A_r}$ and $\gamma^1$.
Combining~(\ref{eqn::multiannulus_levelline_aux10},\ref{eqn::multiannulus_levelline_aux12}), the pair $(\Gamma_{\A_r},\gamma^1)$ can be coupled such that given $\gamma_{[0,t]}^1$, the field $\Gamma_{\A_r}(\cdot)+ \varphiann{n}(r;\bs{\alpha},\bs{\beta};\cdot)$, restricted to $\A_r\setminus {\gamma_{[0,t]}^1}$, has the same law as
\begin{align*}
\left( \Gamma_{\A_{r-t}}(\cdot)+ \varphiann{n}(r-t;\zeta_t^1,\mathfrak{\covmap}_t(\alpha_2),\ldots,\mathfrak{\covmap}_t(\alpha_n);\Re \mathfrak{\covmap}_t(\bs{\beta}+\ii r);\cdot)  \right)\circ \mathfrak{g}_t .
\end{align*}
This confirms the coupling so that $\gamma^1$ is level of $\Gamma+\varphiann{n}$. 
\end{proof}
\begin{lemma}\label{lem::multiannulus_levelline_aux10}
Assume the same setup as in the proof of Lemma~\ref{lem::multiannulus_levelline}, the relation~\eqref{eqn::multiannulus_levelline_aux10} holds. 
\end{lemma}
\begin{proof}
We introduce the analytic function
\begin{equation*}  
	\Phi (r;\bs{\alpha},\bs{\beta};z)=\sum_{j=1}^{2N} \eps_j \log \Theta_1(r; z-\alpha_j) - \sum_{j=1}^{2N} \eps_j \log \Theta_3(r; z-\beta_j) - \sum_{j=1}^{2N} \eps_j (\alpha_j-\beta_j) \frac{z}{2r},
\end{equation*}
and let
\begin{equation*}  
	\mathcal{M}_t(z)=\Phi(r-t;\zeta_t^1,\mathfrak{\covmap}_t(\alpha_2),\ldots,\mathfrak{\covmap}_t(\alpha_n),\Re \mathfrak{\covmap}_t(\bs{\beta}+\ii r);\mathfrak{\covmap}_t(z)).
\end{equation*}
Then we have $M_t(z)=\Im(\mathcal{M}_t(z))$. Applying It\^{o}'s formula and~(\ref{eqn::annulus_Loewner_equation_cov},\ref{eqn::annulus_Loewner_equation_cov_H3},\ref{eqn::GFF_ann_SDE}), we have
\begin{equation}  \label{eqn::multiannulus_levelline_aux5}
	\ud \mathcal{M}_t(z) =2 \partial_{\alpha_1} \Phi \ud B_t + \mathcal{L} \Phi \ud t,
\end{equation}
where $B$ is a standard Brownian motion, and $\mathcal{L}$ is a differential operator defined as
\begin{equation}  \label{eqn::multiannulus_levelline_aux6}
\begin{split}
	\mathcal{L} \Phi(r;\bs{\alpha},\bs{\beta};z)= \bigg( & -\partial_r +4(\partial_{\alpha_1}\log\LZann{n}(r;\bs{\alpha},\bs{\beta})) \partial_{\alpha_1}
	+\sum_{j=2}^n H_1(r;\alpha_j-\alpha_1)\partial_{\alpha_j}\\
	&
	+\sum_{j=1}^n H_3(r;\beta_j-\alpha_1)\partial_{\beta_j}
	+2\partial_{\alpha_1}^2
	+H_1(r;z-\alpha_1)\partial_z \bigg) \Phi(r;\bs{\alpha},\bs{\beta};z),
\end{split}
\end{equation}
and 
\begin{equation*}  
	\Phi=\Phi(r-t;\zeta_t^1,\mathfrak{\covmap}_t(\alpha_2),\ldots,\mathfrak{\covmap}_t(\alpha_n),\Re \mathfrak{\covmap}_t(\bs{\beta}+\ii r);\mathfrak{\covmap}_t(z)).
\end{equation*}
Applying~(\ref{eqn::H1H3_def},\ref{eqn::Theta_heat}), we have
\begin{align}   \label{eqn::multiannulus_levelline_aux8}
\begin{split}
		\partial_z \Phi(r;\bs{\alpha},\bs{\beta};z) = & \frac{1}{2}\sum_{j=1}^{2N}\eps_j H_1(r;z-\alpha_j)-\frac{1}{2}\sum_{j=1}^{2N}\eps_j H_3(r;z-\beta_j) -\frac{1}{2r}\sum_{j=1}^{2N}\eps_j(\alpha_j-\beta_j),\\
	\partial_{\alpha_j} \Phi(r;\bs{\alpha},\bs{\beta};z) = & -\frac{\eps_j}{2} H_1(r;z-\alpha_j)-\frac{\eps_j z}{2r}, \quad \partial_{\beta_j} \Phi(r;\bs{\alpha},\bs{\beta};z) =  \frac{\eps_j}{2} H_3(r;z-\beta_j)+\frac{\eps_j z}{2r},\\
	\partial_{\alpha_1}^2 \Phi(r;\bs{\alpha},\bs{\beta};z) = & -\frac{1}{2}\partial_z H_1(r;z-\alpha_1),\\
	\partial_{r} \Phi(r;\bs{\alpha},\bs{\beta};z) = & \sum_{j=1}^{2N}\eps_j \left(\frac{1}{2}\partial_z H_1(r;z-\alpha_j)+\frac{1}{4}H_1(r;z-\alpha_j)^2\right)\\
	&-\sum_{j=1}^{2N}\eps_j \left(\frac{1}{2}\partial_z H_3(r;z-\beta_j)+\frac{1}{4}H_3(r;z-\beta_j)^2\right)+\frac{z}{2r^2}\sum_{j=1}^{2N}\eps_j(\alpha_j-\beta_j). 
\end{split}
\end{align}
Moreover, we have
\begin{equation}  \label{eqn::multiannulus_levelline_aux9}
	4\partial_{\alpha_1}\log\LZann{n}(r;\bs{\alpha},\bs{\beta})=
	\frac{1}{r}\sum_{j=1}^{2N}\eps_j(\alpha_j-\beta_j)
	+\sum_{j=2}^{2N}\eps_j H_1(r;\alpha_j-\alpha_1)
	-\sum_{j=1}^{2N}\eps_j H_3(r;\beta_j-\alpha_1).
\end{equation}
Combining~(\ref{eqn::H1H3_shift},\ref{eqn::multiannulus_levelline_aux6},\ref{eqn::multiannulus_levelline_aux8},\ref{eqn::multiannulus_levelline_aux9}), we find that 
\[
\mathcal{L} \Phi(r;\bs{\alpha},\bs{\beta};z) \in \R, \quad \text{ for } z\in \R \cup (\R+\ii r),
\]
which implies $\Im(\mathcal{L} \Phi(r;\bs{\alpha},\bs{\beta};z))=0$ for $z\in \partial \S_r$. 
Since $\Im(\mathcal{L} \Phi(r;\bs{\alpha},\bs{\beta};z))$ is a harmonic $2\pi$-periodic function in $\S_r$, we conclude that $\Im(\mathcal{L} \Phi(r;\bs{\alpha},\bs{\beta};z))=0$ for $z\in \S_r$. Plugging into~\eqref{eqn::multiannulus_levelline_aux5}, we obtain
\[	\ud M_t(z) = 2 \Im(\partial_{\alpha_1} \Phi) \ud B_t + \Im (\mathcal{L} \Phi) \ud t= \left( \Im H_1(r-t;\mathfrak{\covmap}_t(z)-\zeta_t^1) + \frac{\Im \mathfrak{\covmap}_t(z)}{r-t} \right) \ud B_t,
\]
which completes the proof of~\eqref{eqn::multiannulus_levelline_aux10}. 
\end{proof}

\begin{lemma}
Assume the same setup as in the proof of Lemma~\ref{lem::multiannulus_levelline}, the relation~\eqref{eqn::multiannulus_levelline_aux12} holds. 
\end{lemma}
\begin{proof}
Applying~(\ref{eqn::H1H3_def},\ref{eqn::Theta_heat}), we have
\begin{align} \label{eqn::multiannulus_levelline_aux13}
\begin{split}
	& \frac{\ud}{\ud t} \log \Theta_1(r-t;\mathfrak{\covmap}_t(z_1)-\mathfrak{\covmap}_t(z_2)) \\
	= & -\frac{1}{2}\partial_z H_1(r-t;\mathfrak{\covmap}_t(z_1)-\mathfrak{\covmap}_t(z_2))-\frac{1}{4}H_1(r-t;\mathfrak{\covmap}_t(z_1)-\mathfrak{\covmap}_t(z_2))^2 \\
	&+\frac{1}{2}H_1(r-t;\mathfrak{\covmap}_t(z_1)-\mathfrak{\covmap}_t(z_2))\left(H_1(r-t;\mathfrak{\covmap}_t(z_1)-\zeta_t^1)-H_1(r-t;\mathfrak{\covmap}_t(z_2)-\zeta_t^1)\right),\\
	& \frac{\ud}{\ud t} \log \Theta_1(r-t;\mathfrak{\covmap}_t(z_1)-\overline{\mathfrak{\covmap}_t(z_2)}) \\
	= & -\frac{1}{2}\partial_z H_1(r-t;\mathfrak{\covmap}_t(z_1)-\overline{\mathfrak{\covmap}_t(z_2)})-\frac{1}{4}H_1(r-t;\mathfrak{\covmap}_t(z_1)-\overline{\mathfrak{\covmap}_t(z_2)})^2 \\
	&+\frac{1}{2}H_1(r-t;\mathfrak{\covmap}_t(z_1)-\overline{\mathfrak{\covmap}_t(z_2)})\left(H_1(r-t;\mathfrak{\covmap}_t(z_1)-\zeta_t^1)-H_1(r-t;\overline{\mathfrak{\covmap}_t(z_2)}-\zeta_t^1)\right).
\end{split}
\end{align}
Plugging~\eqref{eqn::identity_H1} into~\eqref{eqn::multiannulus_levelline_aux13} with  $w_1=\mathfrak{\covmap}_t(z_1)-\zeta_t^1$ and $w_2=\mathfrak{\covmap}_t(z_2)-\zeta_t^1$, we obtain
\begin{align} \label{eqn::multiannulus_levelline_aux14}
	\frac{\ud}{\ud t} \log \left| \frac{\Theta_1(r-t;\mathfrak{\covmap}_t(z_1)-\mathfrak{\covmap}_t(z_2))}{\Theta_1(r-t;\mathfrak{\covmap}_t(z_1)-\overline{\mathfrak{\covmap}_t(z_2)})} \right| =\Im H_1(r-t;\mathfrak{\covmap}_t(z_1)-\zeta_t^1)\Im H_1(r-t;\mathfrak{\covmap}_t(z_2)-\zeta_t^1).
\end{align}
Moreover, we have
\begin{align} \label{eqn::multiannulus_levelline_aux15}
	\frac{\ud}{\ud t} \left( \frac{\Im \mathfrak{\covmap}_t(z_1)  \Im \mathfrak{\covmap}_t(z_2)}{r-t} \right)
	= &\frac{\Im H_1(r-t;\mathfrak{\covmap}_t(z_1)-\zeta_t^1)\Im \mathfrak{\covmap}_t(z_2)}{r-t}+\frac{\Im \mathfrak{\covmap}_t(z_1)\Im H_1(r-t;\mathfrak{\covmap}_t(z_2)-\zeta_t^1)}{r-t} \notag \\
	&+\frac{\Im \mathfrak{\covmap}_t(z_1)\Im \mathfrak{\covmap}_t(z_2)}{(r-t)^2}.
\end{align}
Combining~(\ref{eqn::multiannulus_levelline_aux10},\ref{eqn::multiannulus_levelline_aux11},\ref{eqn::multiannulus_levelline_aux14},\ref{eqn::multiannulus_levelline_aux15}), we conclude that
\begin{equation*}
	\ud\langle M_t(z_1),M_t(z_2)\rangle
	=-\ud \Green_{\A_{r-t}}\big(q(\mathfrak{\covmap}_t(z_1)),q(\mathfrak{\covmap}_t(z_2))\big) =-\ud \Green_{\A_{r-t}}\big(\mathfrak{g}_t(q(z_1)),\mathfrak{g}_t(q(z_2))\big).
\end{equation*}
This completes the proof of~\eqref{eqn::multiannulus_levelline_aux12}.
\end{proof}

\begin{proof}[Proof of Theorem~\ref{thm::GFF_levellines_annulus}]
By~\cite[Proposition~3.18]{AruLupuSepulvedaFPSGFF}, the level line $\gamma^j$ is continuous up to and including the hitting time of $\partial \A_r$, where 
the possible terminal points are $
\{x_{j+1},x_{j+3},\ldots,x_{j+2N-1}, y_j,y_{j+2},\ldots,y_{j+2N-2}\}
$; and $\gamma^j$ is measurable with respect to $\Gamma$, which gives the first item. 

For $\bs{t}=(t_1, \ldots, t_n)$ such that $t_j<r$ for all $j$ and that $\gamma_{[0,t_i]}^i\cap \gamma_{[0,t_j]}^j=\emptyset$ for $i\neq j$, we define measure $\widehat{\PP}$ locally by tilting $\Pind{n}(\LZcro{1})$ by the $n$-time local martingale $M_{\bs{t}}(\LZcro{1};\LZann{n})$. 
Applying Lemma~\ref{lem::multiannulus_levelline} and the domain Markov property of GFF and $\widehat{\PP}$, we conclude that $(\Gamma,\bs{\gamma})$ can be coupled such that given $\bs{\gamma}_{[\bs{0},\bs{t}]}=\left(\gamma_{[0,t_1]}^1,\ldots,\gamma_{[0,t_n]}^n\right)$, the field $\Gamma(\cdot)+ \varphiann{n}(r;\bs{\alpha},\bs{\beta};\cdot)$, restricted to $\A_r\setminus \bs{\gamma}_{[\bs{0},\bs{t}]}$, has the same law as 
\begin{align*}
	\left( \Gamma(\cdot)+ \varphiann{n}\left(\Mod(\A_r\setminus \bs{\gamma}_{[\bs{0},\bs{t}]});\bs{W}_{\bs{t}};\Re \mathfrak{\covmap}_{\bs{t}}(\bs{\beta}+\ii r);\cdot\right)  \right)\circ \mathfrak{g}_{\bs{t}},
\end{align*}
where $\mathfrak{g}_{\bs{t}}$ is the mapping-out function of $\bs{\gamma}_{[\bs{0},\bs{t}]}$ whose covering map sends $\gamma_{t_j}^j$ to $W_{\bs{t}}^j$, which implies that the law of $\bs{\gamma}_{[\bs{0},\bs{t}]}$ is the same as $\widehat{\PP}$. This finishes the proof of the second item.
\end{proof}

\subsection{Proof of Theorem~\ref{thm::crossingproba} and of Proposition~\ref{prop::asymptotics}}
\label{subsec::GFF_crossingproba}
We will complete the proof of Theorem~\ref{thm::crossingproba} in this section. Let us recall the strategy mentioned in the introduction.
We first run $\gamma^1$ and construct a martingale from the ratio of two annulus partition functions. When $\gamma^1$ completes the prescribed crossing, this ratio degenerates into a ratio of partition functions in the simply connected domain $\A_r\setminus\gamma^1$, reducing the remaining problem to the probability that the other level lines form the rainbow pattern; this is exactly the result from~\cite{PeltolaWuGlobalMultipleSLEs} recalled in the following Lemma~\ref{lem::PW19_rainbow_connection}.

\begin{lemma}[{\cite[Theorem~1.4]{PeltolaWuGlobalMultipleSLEs}}]
\label{lem::PW19_rainbow_connection}
Fix an even number $n=2N$ and a nice $n$-polygon $(\Omega;x_1,\ldots,x_n)$. Let $\Gamma$ be the GFF in $\Omega$ with alternating boundary data\footnote{with the convention that $x_{2N+j}=x_j$.}
\begin{align*}
	0 \text{ on } \bigcup_{j=1}^N (x_{2j-1}x_{2j}), \qquad \pi \text{ on }\bigcup_{j=1}^N (x_{2j}x_{2j+1}).
\end{align*} 
Let $\eta^j$ be the level line of the field starting from $x_{j}$ for $1\le j\le N$. Then the probability of the rainbow connectivity is
\begin{align*}
	\PP\left[ \eta^j \text{ terminates at } x_{n-j+1} \text{ for } 1\le j\le N \right]
	=\frac{\LZ_{\rainbow}^{(4)}(\Omega;x_1,\ldots,x_n)}{\LZ_n(\Omega;x_1,\ldots,x_n)},
\end{align*}
where $\LZ_{\rainbow}^{(4)}$ is defined in~\eqref{eqn::rainbow_pf4} and $\LZ_n$ is defined by conformal covariance from~\eqref{eqn::LZunitdisc_def}: if $\varphi:\Omega\to\U$ is conformal with $\varphi(x_j)=\ee^{\ii\theta_j}$ and $\bs{\theta}\in\LX_n$, then
\begin{align}\label{eqn::LZn_pf_polygon_Def}
	\LZ_n(\Omega;x_1,\ldots,x_n)
	:=\prod_{j=1}^{n}|\varphi'(x_j)|^{\frac{1}{4}}\LZ_n(\bs{\theta})
	=\prod_{1\le i<j\le n}\Poisson(\Omega;x_i,x_j)^{-\frac14(-1)^{j-i}}.
\end{align}
\end{lemma}

\begin{proof}[Proof of Theorem~\ref{thm::crossingproba}]
	It suffices to show~\eqref{eqn::cross_wind_proba} for $\ell=m=0$.	
	First, we derive a martingale process for $\gamma^1$. Denote the driving function of $\gamma^1$ by $\zeta_t$, the mapping-out function of $\gamma^1$ by $\mathfrak{g}_t$, and the covering map of $\mathfrak{g}_t$ by $\mathfrak{\covmap}_t$. Denote by $T$ the lifetime of $\gamma^1$. We show that
	\begin{align} \label{eqn::mtg_partition_ratio}
		M_{t}:=\frac{\LZcro{n}^{(0)}(r-t;\zeta_t,\mathfrak{\covmap}_t(\alpha_2),\ldots,\mathfrak{\covmap}_t(\alpha_n),\Re \mathfrak{\covmap}_t(\bs{\beta}+\ii r))}{\LZann{n}(r-t;\zeta_t,\mathfrak{\covmap}_t(\alpha_2),\ldots,\mathfrak{\covmap}_t(\alpha_n),\Re \mathfrak{\covmap}_t(\bs{\beta}+\ii r))},\quad t<T,
	\end{align}
	is a local martingale with respect to $\gamma^1$.
	Recall that the law of $\gamma^1$ is the same as tilting $\Pcro{1}$ by the local martingale $M_t(\LZcro{1}; \LZann{n})$ in~\eqref{eqn::singletime_mart_kappa4}. 
	From Proposition~\ref{prop::multitime_mart_annulus}, the process
	\begin{align*}
		M_t(\LZcro{1}; \LZcro{n}^{(0)})=\ee^{\frac{n-1}{4}t}	\prod_{j=2}^{n}	\left(\mathfrak{\covmap}_t'(\alpha_j)	\mathfrak{\covmap}_t'(\beta_j+\ii r)\right)^{\frac14}	\frac{\LZcro{n}^{(0)} \left(r-t;\zeta_t,\mathfrak{\covmap}_t(\alpha_2),		\ldots,\mathfrak{\covmap}_t(\alpha_n),		\Re\mathfrak{\covmap}_t(\bs{\beta}+\ii r)\right)}	{\LZcro{1}\left(r-t;\zeta_t,\Re\mathfrak{\covmap}_t(\beta_1+\ii r)\right)},
	\end{align*}
	is a local martingale for $\Pcro{1}$, because $\LZcro{n}^{(0)}$ satisfies annulus BPZ equations~\eqref{eqn::annulus_BPZ_kappa4_alpha}. Thus, $M_t=M_t(\LZcro{1}; \LZcro{n}^{(0)})/M_t(\LZcro{1}; \LZann{n})$ is a local martingale for $\gamma^1$. 
	Combining with~\eqref{eqn::LZcro_LZann_mono}, the process $M_t$ is a bounded martingale with respect to $\gamma^1$.
	
	\medbreak
		Then, we investigate $M_t$ as $t\to T$. Following the notations in~\eqref{eqn::F(t,z)G(t,z)}, we define 
		\[F(t,z):=\prod_{m=1}^\infty\left(1-\ee^{\frac{-2m\pi^2}{r-t}+\frac{\pi z}{r-t}}\right)\quad\text{and}\quad G(t,z):=F(t,-z).\]
		Define $D_t:=\A_r\setminus \gamma^1[0,t]$ and define $\Omega_t:=\S_r\setminus q^{-1}(\gamma^1[0,t])$. By~\eqref{eqn::LZcro_LZann_mono_aux3} we have
		\begin{align}	\label{eqn::crossingproba_aux2}
			M_t=&
			\prod_{\substack{2\le i<k\le n\\ 2\nmid (k-i)}}
			\underbrace{\left(\frac{\Poisson(\Omega_t;\alpha_i,\beta_k+\ii r)\Poisson(\Omega_t;\alpha_k,\beta_i+\ii r)}
				{\Poisson(\Omega_t;\alpha_i,\alpha_k)\Poisson(\Omega_t;\beta_i+\ii r,\beta_k+\ii r)}\right)^{\frac12}}_{:=S_{i,k,t}}\notag\\
			&\times\underbrace{\frac{
					\displaystyle\prod_{\substack{2\le i<k\le n\\2\nmid(k-i)}}F(t,\mathfrak{\covmap}_t(\alpha_k)-\mathfrak{\covmap}_t(\alpha_i))
					\prod_{\substack{2\le j\le n\\2\mid j}}F(t,\mathfrak{\covmap}_t(\alpha_j)-\zeta_t)
					\prod_{\substack{1\le i<k\le n\\2\nmid(k-i)}}F(t,\mathfrak{\covmap}_t(\beta_k+\ii r)-\mathfrak{\covmap}_t(\beta_i+\ii r))}
				{\displaystyle\prod_{\substack{2\le j\le n\\2\mid j}}F(t,\mathfrak{\covmap}_t(\beta_j+\ii r)-\zeta_t)
					\prod_{\substack{2\le i\le n,\,1\le k\le n\\2\nmid(k-i)}}F(t,\mathfrak{\covmap}_t(\beta_k+\ii r)-\mathfrak{\covmap}_t(\alpha_i))}}_{:=R_{1,t}}\notag\\
			&\times\underbrace{\frac{
					\displaystyle\prod_{\substack{2\le i<k\le n\\2\nmid(k-i)}}G(t,\mathfrak{\covmap}_t(\alpha_k)-\mathfrak{\covmap}_t(\alpha_i))
					\prod_{\substack{2\le j\le n\\2\mid j}}G(t,\mathfrak{\covmap}_t(\alpha_j)-\zeta_t)
					\prod_{\substack{1\le i<k\le n\\2\nmid(k-i)}}G(t,\mathfrak{\covmap}_t(\beta_k+\ii r)-\mathfrak{\covmap}_t(\beta_i+\ii r))}
				{\displaystyle\prod_{\substack{2\le j\le n\\2\mid j}}G(t,\mathfrak{\covmap}_t(\beta_j+\ii r)-\zeta_t)
					\prod_{\substack{2\le i\le n,\,1\le k\le n\\2\nmid(k-i)}}G(t,\mathfrak{\covmap}_t(\beta_k+\ii r)-\mathfrak{\covmap}_t(\alpha_i))}}_{:=R_{2,t}}\notag\\
			&\times\prod_{\substack{2\le j\le n\\2\mid j}}
			\underbrace{\frac{\left(1-\ee^{\frac{-\pi(\mathfrak{\covmap}_t(\alpha_j)-\zeta_t)}{r-t}}\right)
					\left(1-\ee^{\frac{-\pi(\mathfrak{\covmap}_t(\beta_j+\ii r)-\mathfrak{\covmap}_t(\beta_1+\ii r))}{r-t}}\right)}
				{\left(1+\ee^{\frac{-\pi(\Re\mathfrak{\covmap}_t(\beta_j+\ii r)-\zeta_t)}{r-t}}\right)
					\left(1+\ee^{\frac{-\pi(\mathfrak{\covmap}_t(\alpha_j)-\Re\mathfrak{\covmap}_t(\beta_1+\ii r))}{r-t}}\right)}}_{:=L_{j,t}}.
		\end{align}
Note that $R_{1,t}\le 1$, $R_{2,t}\le 1$, $L_{j,t}\le 1$ due to the numerators are smaller than one and the denominators are bigger than one, and $S_{i,k,t}\le 1$ due to~\eqref{eqn::LZcro_LZann_mono_aux4}. 
We consider all possible terminal points of $\gamma^1$:	
\[ \gamma^1_T:=\lim_{t\to T} \gamma_t^1\in \{x_{2},x_{4},\ldots,x_{n}, y_1,y_{3},\ldots,y_{n-1}\}. \] 
When $\gamma^1$ crosses the annulus, i.e. when it terminates in $\{y_1, y_3, \ldots, y_{n-1}\}$, it may have different winding. To simplify our notations, we write $\beta_{j+\ell n}=\beta_j+2\pi\ell$ as before. 
Using this notation, when $\gamma^1$ crosses the annulus, we have $q^{-1}(\gamma^1_T)\in \{\beta_{2j-1}+\ii r, j\in\Z\}$.
\begin{itemize}
\item Suppose $\gamma_t^1\to x_{2j}$ as $t\to T$. Then $T<r$ and $\mathfrak{\covmap}_t(\alpha_{2j})-\zeta_t\to 0 \text{ or } 2\pi$ as $t\to T$.
Thus we have $R_{1,t}\to 0$ as $t\to T$, which implies $M_t\to 0$ as $t\to T$.
		
\item Suppose $q^{-1}(\gamma^1)$ terminates at $\beta_{2j-1}+\ii r$ for $j\neq 1$. 
Then $T=r$. By~(\ref{eqn::terminal_beta},\ref{eqn::finite}), we have
\begin{equation} \label{eqn::crossingproba_aux4}
	\frac{\Re\mathfrak{\covmap}_t(\beta_{2j-1}+\ii r)-\mathfrak{\covmap}_t(\alpha_2)}{r-t}\to-\infty, \qquad \frac{\Re\mathfrak{\covmap}_t(\beta_{2j-1}+\ii r)-\mathfrak{\covmap}_t(\alpha_{2N}-2\pi)}{r-t}\to +\infty.	
\end{equation}
Plugging~\eqref{eqn::crossingproba_aux4} into~\eqref{eqn::LZcro_LZann_mono_aux4}, we obtain
\begin{align*}
\begin{cases}
	S_{2,2j-1,t} \le \exp\left(\frac{\pi\left(\Re\mathfrak{\covmap}_t(\beta_{2j-1}+\ii r)-\mathfrak{\covmap}_t(\alpha_2)\right)}{r-t}\right)\to 0, & 2\le j\le N,\\
	S_{2,2N,t} \le \exp\left(\frac{\pi\left(\Re\mathfrak{\covmap}_t(\beta_{2N}+\ii r)-\mathfrak{\covmap}_t(\alpha_2)\right)}{r-t}\right) \le \exp\left(\frac{\pi\left(\Re\mathfrak{\covmap}_t(\beta_{2j-1}+\ii r)-\mathfrak{\covmap}_t(\alpha_2)\right)}{r-t}\right) \to 0, & j\ge N+1, \\
	S_{2N-1,2N,t} \le \exp\left(\frac{\pi\left(\mathfrak{\covmap}_t(\alpha_{2N})-\Re\mathfrak{\covmap}_t(\beta_{2N-1}+\ii r)\right)}{r-t}\right) \le \exp\left(\frac{\pi\left(\mathfrak{\covmap}_t(\alpha_{2N}-2\pi)-\Re\mathfrak{\covmap}_t(\beta_{2j-1}+\ii r)\right)}{r-t}\right)\to 0, & j\le 0.
\end{cases}
\end{align*}
Thus $M_t\to0$ as $t\to T$.
		
\item Suppose $q^{-1}(\gamma^1)$ terminates at $\beta_1+\ii r$. Then $T=r$. 
Let us check all terms on the right-hand side of~\eqref{eqn::crossingproba_aux2}. By the convergence of boundary Poisson kernels and~(\ref{eqn::rainbow_pf4},\ref{eqn::LZn_pf_polygon_Def}), we have
\begin{align} \label{eqn::crossingproba_aux5}
	\prod_{\substack{2\le i<k\le n\\ 2\nmid (k-i)}}
	\left(\frac{\Poisson(\Omega_t;\alpha_i,\beta_k+\ii r)\Poisson(\Omega_t;\alpha_k,\beta_i+\ii r)}
	{\Poisson(\Omega_t;\alpha_i,\alpha_k)\Poisson(\Omega_t;\beta_i+\ii r,\beta_k+\ii r)}\right)^{\frac12}=&\frac{\LZ_{\rainbow}^{(4)}(\Omega_t;\alpha_2,\ldots,\alpha_n,\beta_n+\ii r,\ldots,\beta_2+\ii r)}
	{\LZ_{2n-2}(\Omega_t;\alpha_2,\ldots,\alpha_n,\beta_n+\ii r,\ldots,\beta_2+\ii r)} \notag \\
	\to &\frac{\LZ_{\rainbow}^{(4)}(\A_r\setminus\gamma^1;x_2,\ldots,x_n,y_n,\ldots,y_2)}
	{\LZ_{2n-2}(\A_r\setminus\gamma^1;x_2,\ldots,x_n,y_n,\ldots,y_2)},	
\end{align}
as $t\to r$.
Moreover, it is shown in the proof of Lemma~\ref{lem::multiannulusSLE4_pf_2nd} that the estimates~(\ref{eqn::terminal_alpha},\ref{eqn::terminal_beta},\ref{eqn::finite}) give
\begin{equation} \label{eqn::crossingproba_aux6}
	\lim_{t\to r}R_{1,t}=1, \quad \lim_{t\to r}R_{2,t}=1,\quad
	\lim_{t\to r}L_{j,t}=1,\quad 2\le j\le n.
\end{equation}
Plugging~(\ref{eqn::crossingproba_aux5},\ref{eqn::crossingproba_aux6}) into~\eqref{eqn::crossingproba_aux2}, we obtain
\begin{equation*}
	\lim_{t\to r}M_t=	\frac{\LZ_{\rainbow}^{(4)}(\A_r\setminus\gamma^1;x_2,\ldots,x_n,y_n,\ldots,y_2)}
	{\LZ_{2n-2}(\A_r\setminus\gamma^1;x_2,\ldots,x_n,y_n,\ldots,y_2)}.
\end{equation*}
\end{itemize}
Combining the above three items, we have
\begin{align}\label{eqn::crossingproba_terminal}
	M_T=\one\{q^{-1}(\gamma^1)\text{ connects }\alpha_1\text{ to }\beta_1+\ii r\}\frac{\LZ_{\rainbow}^{(4)}(\A_r\setminus\gamma^1;x_2,\ldots,x_n,y_n,\ldots,y_2)}
	{\LZ_{2n-2}(\A_r\setminus\gamma^1;x_2,\ldots,x_n,y_n,\ldots,y_2)}.
\end{align}
On the event in $\{q^{-1}(\gamma^1)\text{ connects }\alpha_1\text{ to }\beta_1+\ii r\}$, the domain Markov property of the GFF shows that, conditionally on $\gamma^1$, the remaining level lines $(\gamma^2,\ldots,\gamma^n)$ are the level lines of a GFF in the polygon $(\A_r\setminus\gamma^1;x_2,\ldots,x_n,y_n,\ldots,y_2)$ with alternating boundary data. By Lemma~\ref{lem::PW19_rainbow_connection}, the ratio in~\eqref{eqn::crossingproba_terminal} is the conditional probability of $\gamma^j$ terminates at $y_j$ for $2\le j\le n$. Since $M_t\le 1$, the optional stopping theorem gives
\begin{align*}
	\frac{\LZcro{n}^{(0)}(r;\bs{\alpha},\bs{\beta})}{\LZann{n}(r;\bs{\alpha},\bs{\beta})}=M_0=\E[M_T]=\PP\left[q^{-1}(\gamma^j)\text{ connects }\alpha_j\text{ to }\beta_j+\ii r,\ 1\le j\le n\right],
\end{align*}
which gives~\eqref{eqn::cross_wind_proba} when $m=\ell=0$. For general $m\in\{1,\ldots,N\}$ and $\ell\in\mathbb Z$, relabeling the marked points and repeating the above argument gives~\eqref{eqn::cross_wind_proba}. Summing over the disjoint connectivity events in~\eqref{eqn::cross_wind_proba} gives~\eqref{eqn::cross_proba}.
\end{proof}

\begin{proof}[Proof of Proposition~\ref{prop::asymptotics}]
We obtain~\eqref{eqn::crossproba_ell_asymp} from~\eqref{eqn::cross_wind_proba} and~\eqref{eqn::LZann_asymp} and~\eqref{eqn::LZcroell_asymp}. 
We obtain~\eqref{eqn::crossproba_asymp} from~\eqref{eqn::cross_proba} and~\eqref{eqn::LZann_asymp} and~\eqref{eqn::LZcro_sum_asymp}. 
\end{proof}

\appendix
\section{Calculation for Proposition~\ref{prop::multitime_mart_annulus}}
\label{appendix::multitime_mart}
We give the details for the It\^o calculation leading to~\eqref{eqn::multitime_mart_annulus_aux1}. 
The following calculations are compiled from ChatGPT 5.5 Pro (OpenAI)'s computation drafts.
We work before the first time when two curves intersect, and localize all
processes so that all marked points stay away from each other. On this time
interval the indicator $\one_{\LE_{\emptyset}(\bs{\gamma}_{\bs{t}})}$ is
constant.

Put
\[
R_{\bs{t}}:=\Mod(\A_r\setminus \bs{\gamma}_{[\bs{0},\bs{t}]}),
\qquad
s_j:=r-t_j,
\]
and set
\[
a_j:=\mathfrak{\covmap}_{\bs{t},j}'(\zeta_{t_j}^j),
\qquad
p_j:=
\frac{\mathfrak{\covmap}_{\bs{t},j}''(\zeta_{t_j}^j)}
{\mathfrak{\covmap}_{\bs{t},j}'(\zeta_{t_j}^j)},
\qquad
q_j:=
\frac{\mathfrak{\covmap}_{\bs{t},j}'''(\zeta_{t_j}^j)}
{\mathfrak{\covmap}_{\bs{t},j}'(\zeta_{t_j}^j)}.
\]
We also write
\[
Y_\ell:=\Re\mathfrak{\covmap}_{\bs{t}}(\beta_\ell+\ii r),
\qquad
V_j:=\Re\mathfrak{\covmap}_{t_j}^j(\beta_j+\ii r),
\qquad
\bs{Y}:=(Y_1,\ldots,Y_n),
\]
and
\[
\LF_n^{\bs{t}}
:=
\LF_n(R_{\bs{t}};\bs{W}_{\bs{t}},\bs{Y}),
\qquad
\LF_1^j
:=
\LF_1(s_j;\zeta_{t_j}^j,V_j).
\]
In the following, derivatives of $\LF_n$ with respect to the first and second
spatial coordinate lists are denoted by $\partial_{\alpha_i}$ and
$\partial_{\beta_\ell}$, and then evaluated at
$(r;\bs{\alpha},\bs{\beta})=(R_{\bs{t}};\bs{W}_{\bs{t}},\bs{Y})$.
Finally, define
\[
\Phi_j:=
\frac{\partial_{\alpha_j}\LF_n^{\bs{t}}}{\LF_n^{\bs{t}}},
\qquad
\Psi_j:=
\frac{\partial_\alpha \LF_1^j}{\LF_1^j}.
\]
Under $\Pind{n}^{(\kappa)}(\LF_1)$, the individual driving functions satisfy
\begin{equation}\label{eqn::app_annulus_single_driving_for_Ito}
	\ud \zeta_{t_j}^j
	=
	\sqrt{\kappa}\,\ud B_{t_j}^j
	+
	\kappa\Psi_j\,\ud t_j,
	\qquad
	\ud V_j
	=
	H_3(s_j; V_j-\zeta_{t_j}^j)\,\ud t_j .
\end{equation}
By the translation invariance of $\LF_n$,
\begin{equation}\label{eqn::app_translation_invar_diff_for_Ito}
	\sum_{i=1}^{n}\partial_{\alpha_i}\LF_n
	+
	\sum_{\ell=1}^{n}\partial_{\beta_\ell}\LF_n
	=0,
\end{equation}
where the derivatives are evaluated at
$(R_{\bs{t}};\bs{W}_{\bs{t}},\bs{Y})$.

We use the following standard variations for annulus Loewner chains, see~(\ref{eqn::Annulus_basic1}, \ref{eqn::Annulus_basic2}, \ref{eqn::Annulus_basic3}, \ref{eqn::Annulus_basic4}). For each
fixed $j$, when only the $t_j$-coordinate varies, we have
\begin{align}
	\ud R_{\bs{t}}
	&=
	-a_j^2\,\ud t_j, \label{eqn::app_Annulus_basic1}\\
	\ud(Y_\ell-W_{\bs{t}}^j)
	&=
	-a_j\,\ud\zeta_{t_j}^j
	+
	\kappa\mathfrak{b}\,a_jp_j\,\ud t_j
	+
	H_3(R_{\bs{t}}; Y_\ell-W_{\bs{t}}^j)a_j^2\,\ud t_j,
	\label{eqn::app_Annulus_basic2}\\
	\ud(W_{\bs{t}}^i-W_{\bs{t}}^j)
	&=
	-a_j\,\ud\zeta_{t_j}^j
	+
	\kappa\mathfrak{b}\,a_jp_j\,\ud t_j
	+
	H_1(R_{\bs{t}}; W_{\bs{t}}^i-W_{\bs{t}}^j)a_j^2\,\ud t_j,
	\qquad i\neq j, \label{eqn::app_Annulus_basic3}\\
	\frac{\ud a_j}{a_j}
	&=
	p_j\,\ud\zeta_{t_j}^j
	+
	\left(
	\frac{3\kappa-8}{6}q_j
	+
	\frac{1}{2}p_j^2
	-
	\LE(s_j)
	+
	\LE(R_{\bs{t}})a_j^2
	\right)\ud t_j, \label{eqn::app_Annulus_basic4}\\
	\frac{\ud a_i}{a_i}
	&=
	\partial_z H_1(R_{\bs{t}}; W_{\bs{t}}^i-W_{\bs{t}}^j)
	a_j^2\,\ud t_j,
	\qquad i\neq j, \label{eqn::app_Annulus_basic5}\\
	\frac{\ud\mathfrak{\covmap}_{\bs{t}}'(\beta_\ell+\ii r)}
	{\mathfrak{\covmap}_{\bs{t}}'(\beta_\ell+\ii r)}
	&=
	\partial_z H_3(R_{\bs{t}}; Y_\ell-W_{\bs{t}}^j)
	a_j^2\,\ud t_j, \label{eqn::app_Annulus_basic6}\\
	\frac{\ud(\mathfrak{\covmap}_{t_j}^j)'(\beta_j+\ii r)}
	{(\mathfrak{\covmap}_{t_j}^j)'(\beta_j+\ii r)}
	&=
	\partial_z H_3(s_j; V_j-\zeta_{t_j}^j)\,\ud t_j .
	\label{eqn::app_Annulus_basic7}
\end{align}
Moreover, Lemma~\ref{lem::mt_blm_annulus_diff} gives
\begin{equation}\label{eqn::app_blm_diff_for_Ito_annulus}
	\ud \blm_{r,\bs{t}}
	=
	\sum_{j=1}^n
	\left(
	-\frac{1}{3}\left(q_j-\frac{3}{2}p_j^2\right)
	+
	a_j^2\left(\LE(R_{\bs{t}})+\frac{1}{R_{\bs{t}}}\right)
	-
	\left(\LE(s_j)+\frac{1}{s_j}\right)
	\right)\ud t_j .
\end{equation}

We now expand the four non-constant factors in~\eqref{eqn::multitime_mart_annulus}. First, let
\[
I_{\bs{t}}
:=
\exp\bigg(
\frac{\mathfrak{c}}{2}\blm_{r,\bs{t}}
+
\mathfrak{b}\sum_{j=1}^n(r-t_j)
-
n\mathfrak{b}R_{\bs{t}}
\bigg).
\]
Using~(\ref{eqn::app_Annulus_basic1},\ref{eqn::app_blm_diff_for_Ito_annulus}), we get
\begin{align*}
	\frac{\ud I_{\bs{t}}}{I_{\bs{t}}}
	=
	\sum_{j=1}^n
	\left(
	-\frac{\mathfrak{c}}{6}\left(q_j-\frac{3}{2}p_j^2\right)
	+
	\frac{\mathfrak{c}}{2}
	\left(
	a_j^2\left(\LE(R_{\bs{t}})+\frac{1}{R_{\bs{t}}}\right)
	-
	\left(\LE(s_j)+\frac{1}{s_j}\right)
	\right)-\mathfrak{b}
	+
	n\mathfrak{b}a_j^2
	\right) \ud t_j .
\end{align*}

Second, for each $j$, It\^o's formula and~\eqref{eqn::app_Annulus_basic4} give
\begin{align*}
	\begin{split}
		\frac{\ud a_j^{\mathfrak{b}}}{a_j^{\mathfrak{b}}}
		=
		&\sqrt{\kappa}\,\mathfrak{b}p_j\,\ud B_{t_j}^j
		+
		\kappa\mathfrak{b}p_j\Psi_j\,\ud t_j +
		\mathfrak{b}
		\left(
		\frac{3\kappa-8}{6}q_j
		+
		\frac{1}{2}p_j^2
		-
		\LE(s_j)
		+
		\LE(R_{\bs{t}})a_j^2
		\right)\ud t_j\\
		&+
		\frac{\kappa\mathfrak{b}(\mathfrak{b}-1)}{2}p_j^2\,\ud t_j
		+
		\mathfrak{b}\sum_{i\neq j}
		\partial_z H_1(R_{\bs{t}}; W_{\bs{t}}^j-W_{\bs{t}}^i)
		a_i^2\,\ud t_i .
	\end{split}
\end{align*}
Therefore, for
\[
J_{\bs{t}}:=\prod_{j=1}^n a_j^{\mathfrak{b}},
\]
we have
\begin{align*}
	\begin{split}
		\frac{\ud J_{\bs{t}}}{J_{\bs{t}}}
		=
		\sum_{j=1}^n
		\sqrt{\kappa}\,\mathfrak{b}p_j\,\ud B_{t_j}^j +
		\sum_{j=1}^n
		\bigg( &
		\kappa\mathfrak{b}p_j\Psi_j
		+
		\mathfrak{b}
		\left(
		\frac{3\kappa-8}{6}q_j
		+
		\frac{1}{2}p_j^2
		-
		\LE(s_j)
		+
		\LE(R_{\bs{t}})a_j^2
		\right)\\
		&	+
		\frac{\kappa\mathfrak{b}(\mathfrak{b}-1)}{2}p_j^2
		+
		\mathfrak{b}a_j^2
		\sum_{i\neq j}
		\partial_z H_1(R_{\bs{t}}; W_{\bs{t}}^i-W_{\bs{t}}^j)
		\bigg)\ud t_j .
	\end{split}
\end{align*}

Third, set
\[
U_\ell:=\mathfrak{\covmap}_{\bs{t}}'(\beta_\ell+\ii r),
\qquad
V_\ell':=(\mathfrak{\covmap}_{t_\ell}^\ell)'(\beta_\ell+\ii r).
\]
Using~(\ref{eqn::app_Annulus_basic6},\ref{eqn::app_Annulus_basic7}), we obtain
\begin{align*}
	\frac{\ud\left(U_\ell^{\mathfrak{b}}/(V_\ell')^{\mathfrak{b}}\right)}
	{U_\ell^{\mathfrak{b}}/(V_\ell')^{\mathfrak{b}}}
	=
	\mathfrak{b}
	\sum_{j=1}^n
	\partial_z H_3(R_{\bs{t}}; Y_\ell-W_{\bs{t}}^j)
	a_j^2\,\ud t_j
	-
	\mathfrak{b}
	\partial_z H_3(s_\ell; V_\ell-\zeta_{t_\ell}^\ell)\,\ud t_\ell .
\end{align*}
Consequently, for
\[
K_{\bs{t}}
:=
\prod_{\ell=1}^n
\frac{\mathfrak{\covmap}_{\bs{t}}'(\beta_\ell+\ii r)^{\mathfrak{b}}}
{(\mathfrak{\covmap}_{t_\ell}^\ell)'(\beta_\ell+\ii r)^{\mathfrak{b}}},
\]
we have
\begin{align*}
	\frac{\ud K_{\bs{t}}}{K_{\bs{t}}}
	=
	\sum_{j=1}^n
	\mathfrak{b}
	\left(
	a_j^2
	\sum_{\ell=1}^n
	\partial_z H_3(R_{\bs{t}}; Y_\ell-W_{\bs{t}}^j)
	-
	\partial_z H_3(s_j; V_j-\zeta_{t_j}^j)
	\right)\ud t_j .
\end{align*}

Fourth, define
\[
L_{\bs{t}}
:=
\frac{\LF_n(R_{\bs{t}};\bs{W}_{\bs{t}},\bs{Y})}
{\prod_{j=1}^n \LF_1(s_j;\zeta_{t_j}^j,V_j)}.
\]
For the numerator, combine~\eqref{eqn::app_translation_invar_diff_for_Ito} with~(\ref{eqn::app_Annulus_basic1},\ref{eqn::app_Annulus_basic2},\ref{eqn::app_Annulus_basic3}). The
stochastic part in the $t_j$-direction is
$\sqrt{\kappa}a_j\Phi_j\,\ud B_{t_j}^j$, and the drift terms are obtained
directly from the same variations. Thus
\begin{align}\label{eqn::app_Ito_Zn_numerator_annulus}
	\begin{split}
		\frac{\ud \LF_n^{\bs{t}}}{\LF_n^{\bs{t}}}
		=
		&\sum_{j=1}^n
		\sqrt{\kappa}\,a_j\Phi_j\,\ud B_{t_j}^j+
		\sum_{j=1}^n
		\left(
		\kappa a_j\Phi_j\Psi_j
		-
		\kappa\mathfrak{b}a_jp_j\Phi_j
		+
		a_j^2\LA_{j,0}^{(n)}
		\right)\ud t_j,
	\end{split}
\end{align}
where all derivatives of $\LF_n$ are evaluated at
$(R_{\bs{t}};\bs{W}_{\bs{t}},\bs{Y})$, and
\begin{align*}
	\begin{split}
		\LA_{j,0}^{(n)}
		:=
		&-\frac{\partial_r\LF_n}{\LF_n}
		+
		\frac{\kappa}{2}
		\frac{\partial_{\alpha_j}^2\LF_n}{\LF_n}+
		\sum_{i\neq j}
		H_1(R_{\bs{t}}; W_{\bs{t}}^i-W_{\bs{t}}^j)
		\frac{\partial_{\alpha_i}\LF_n}{\LF_n}
		+
		\sum_{\ell=1}^n
		H_3(R_{\bs{t}}; Y_\ell-W_{\bs{t}}^j)
		\frac{\partial_{\beta_\ell}\LF_n}{\LF_n}.
	\end{split}
\end{align*}
For each denominator factor, It\^o's formula and~\eqref{eqn::app_annulus_single_driving_for_Ito} give
\begin{align*}
	\frac{\ud \LF_1^j}{\LF_1^j}
	=
	\sqrt{\kappa}\,\Psi_j\,\ud B_{t_j}^j
	+
	\left(
	\kappa\Psi_j^2
	+
	\LA_{j,0}^{(1)}
	\right)\ud t_j,
\end{align*}
where all derivatives of $\LF_1$ are evaluated at
$(s_j;\zeta_{t_j}^j,V_j)$, and
\begin{align*}
	\LA_{j,0}^{(1)}
	:=
	-\frac{\partial_r \LF_1}{\LF_1}
	+
	\frac{\kappa}{2}
	\frac{\partial_\alpha^2\LF_1}{\LF_1}
	+
	H_3(s_j; V_j-\zeta_{t_j}^j)
	\frac{\partial_\beta\LF_1}{\LF_1}.
\end{align*}
Hence
\[
\frac{\ud \bigl(\LF_1^j\bigr)^{-1}}
{\bigl(\LF_1^j\bigr)^{-1}}
=
-\sqrt{\kappa}\,\Psi_j\,\ud B_{t_j}^j
-
\LA_{j,0}^{(1)}\,\ud t_j .
\]
Combining this with~\eqref{eqn::app_Ito_Zn_numerator_annulus}, and including
the quadratic covariation between the numerator and the inverse denominator,
we obtain
\begin{align*}
	\begin{split}
		\frac{\ud L_{\bs{t}}}{L_{\bs{t}}}
		=
		&\sum_{j=1}^n
		\sqrt{\kappa}\,(a_j\Phi_j-\Psi_j)\,\ud B_{t_j}^j+
		\sum_{j=1}^n
		\left(
		a_j^2\LA_{j,0}^{(n)}
		-
		\LA_{j,0}^{(1)}
		-
		\kappa\mathfrak{b}a_jp_j\Phi_j
		\right)\ud t_j .
	\end{split}
\end{align*}

Now we have
\[
M_{\bs{t}}(\LF_1;\LF_n)
=I_{\bs{t}}J_{\bs{t}}K_{\bs{t}}L_{\bs{t}}.
\]
The factors $I_{\bs{t}}$ and $K_{\bs{t}}$ have finite variation. The only
non-trivial quadratic covariation between the four displayed factors is
\begin{equation*}
	\frac{\ud\langle J,L\rangle_{\bs{t}}}{J_{\bs{t}}L_{\bs{t}}}
	=
	\sum_{j=1}^n
	\kappa\mathfrak{b}p_j(a_j\Phi_j-\Psi_j)\,\ud t_j .
\end{equation*}
Therefore
\[
\frac{\ud M_{\bs{t}}(\LF_1;\LF_n)}
{M_{\bs{t}}(\LF_1;\LF_n)}
=
\frac{\ud I_{\bs{t}}}{I_{\bs{t}}}
+
\frac{\ud J_{\bs{t}}}{J_{\bs{t}}}
+
\frac{\ud K_{\bs{t}}}{K_{\bs{t}}}
+
\frac{\ud L_{\bs{t}}}{L_{\bs{t}}}
+
\frac{\ud\langle J,L\rangle_{\bs{t}}}{J_{\bs{t}}L_{\bs{t}}}.
\]
The four drift terms
\[
\kappa\mathfrak{b}p_j\Psi_j,
\qquad
-\kappa\mathfrak{b}p_j\Psi_j,
\qquad
-\kappa\mathfrak{b}a_jp_j\Phi_j,
\qquad
\kappa\mathfrak{b}a_jp_j\Phi_j
\]
cancel. The terms involving $q_j$ and $p_j^2$ are
\[
\left(
\frac{\mathfrak{b}(3\kappa-8)}{6}
-
\frac{\mathfrak{c}}{6}
\right)q_j
+
\left(
\frac{\mathfrak{c}}{4}
+
\frac{\mathfrak{b}}{2}
+
\frac{\kappa\mathfrak{b}(\mathfrak{b}-1)}{2}
\right)p_j^2,
\]
which vanish because
\[
\mathfrak{c}=\mathfrak{b}(3\kappa-8),
\qquad
\mathfrak{b}=\frac{6-\kappa}{2\kappa}.
\]
The remaining modulus terms are
\begin{align*}
	a_j^2
	\left(
	\frac{\mathfrak{c}}{2R_{\bs{t}}}
	+
	\left(\frac{\mathfrak{c}}{2}+\mathfrak{b}\right)\LE(R_{\bs{t}})
	+
	n\mathfrak{b}
	\right)
	-
	\left(
	\frac{\mathfrak{c}}{2s_j}
	+
	\left(\frac{\mathfrak{c}}{2}+\mathfrak{b}\right)\LE(s_j)
	+
	\mathfrak{b}
	\right) =
	a_j^2F_n(R_{\bs{t}})-F_1(s_j),
\end{align*}
where
\begin{equation*}
	F_m(u)
	=
	\frac{6\tilde{\mathfrak{b}}-\mathfrak{b}}{u}
	+
	6\tilde{\mathfrak{b}}\LE(u)
	+
	m\mathfrak{b},
	\qquad m=1,n,
\end{equation*}
and we used
\[
6\tilde{\mathfrak{b}}
=
\frac{\mathfrak{c}}{2}+\mathfrak{b},
\qquad
6\tilde{\mathfrak{b}}-\mathfrak{b}
=
\frac{\mathfrak{c}}{2}.
\]
Thus, we have
\begin{align*}
	\frac{\ud M_{\bs{t}}(\LF_1;\LF_n)}
	{M_{\bs{t}}(\LF_1;\LF_n)}
	=
	\sum_{j=1}^n
	\sqrt{\kappa}\,\Xi_j\,\ud B_{t_j}^j
	+
	\sum_{j=1}^n
	\left(
	a_j^2\LR_j^{(n)}
	-
	\LR_j^{(1)}
	\right)\ud t_j,
\end{align*}
where
\begin{equation*}
	\Xi_j
	:=
	a_j\Phi_j+\mathfrak{b}p_j-\Psi_j,
\end{equation*}
and
\begin{align*}
	\begin{split}
		\LR_j^{(n)}	:=
		\frac{
			\left(F_n(R_{\bs{t}})-\partial_r+\LD_{\alpha_j}^{(n)}\right)
			\LF_n(R_{\bs{t}};\bs{W}_{\bs{t}},\bs{Y})
		}
		{
			\LF_n(R_{\bs{t}};\bs{W}_{\bs{t}},\bs{Y})
		}, \quad 
		\LR_j^{(1)}
		:=
		\frac{
			\left(F_1(s_j)-\partial_r +\LD_{\alpha}^{(2)}\right)
			\LF_1(s_j;\zeta_{t_j}^j,V_j)
		}
		{
			\LF_1(s_j;\zeta_{t_j}^j,V_j)
		},
	\end{split}
\end{align*}
which gives~\eqref{eqn::multitime_mart_annulus_aux1} as desired.
\section{Calculation for Proposition~\ref{prop::annulusBPZ_kappa4}}
\label{appendix::annulusBPZ}
We give details for the calculation in the proof of Proposition~\ref{prop::annulusBPZ_kappa4}: 
	\begin{lemma}\label{lem::Geps_doubleperiodic}
		For $\eps\in \{\pm, +\}$, the function $G_{\eps}(r; \cdot, \cdot)$ defined in~\eqref{eqn::Geps_def} satisfies double-periodicity~\eqref{eqn::Geps_doubleperiodic}. 
	\end{lemma}
	
	\begin{lemma}\label{lem::Geps_nopoles}
		For $\eps\in \{\pm, +\}$, the function $G_{\eps}(r; \cdot, \cdot)$ defined in~\eqref{eqn::Geps_def} does not have any poles. 
	\end{lemma}
	The following calculations are compiled from ChatGPT 5.5 Pro (OpenAI)'s computation drafts.
	Throughout this appendix, the modulus $r>0$ is
	fixed and we write 
\[
h_a(z):=H_a(r;z), \qquad h_a'(z):=\partial_z H_a(r;z),
\qquad a=1,3 .
\]
From~\eqref{eqn::Theta_period}, \eqref{eqn::Theta_shift} and
\eqref{eqn::H1H3}, we will use repeatedly
\begin{align}\label{eqn::app_H_identities}
	\begin{split}
		&h_a(-z)=-h_a(z), \qquad h_a'(-z)=h_a'(z),\qquad a=1,3,\\
		&h_a(z+2\pi)=h_a(z), \qquad h_a(z+2\ii r)=h_a(z)-2\ii,\\
		&h_3(z+\ii r)=h_1(z)-\ii, \qquad h_3(z-\ii r)=h_1(z)+\ii .
	\end{split}
\end{align}
Moreover, by~\eqref{eqn::Laurrent_expand_LE},
\begin{equation}\label{eqn::app_H_laurent}
	h_1(z)=\frac{2}{z}+\LE(r)z+O(z^3), \qquad
	h_1'(z)=-\frac{2}{z^2}+\LE(r)+O(z^2).
\end{equation}
By~\eqref{eqn::app_H_identities}, the expansion near the poles of $h_3$ is
\begin{equation}\label{eqn::app_H3_laurent}
	h_3(\ii r+z)=\frac{2}{z}-\ii+\LE(r)z+O(z^3), \qquad
	h_3'(\ii r+z)=-\frac{2}{z^2}+\LE(r)+O(z^2).
\end{equation}
Using the notations above, we rewrite~\eqref{eqn::Geps_def} as
\begin{align}\label{eqn::app_Gepsilon_def}
	\begin{split}
		G_{\eps}(r; \bs{\alpha}, \bs{\beta}):=&
		-\sum_{i,j=1}^{n}\epsilon_i\epsilon_j
		\left(\frac14 h_3'(\beta_j-\alpha_i)
		+\frac18 h_3(\beta_j-\alpha_i)^2\right)\\
		&+\sum_{1\le i<j\le n}\epsilon_i\epsilon_j
		\left(\frac14 h_1'(\alpha_j-\alpha_i)
		+\frac18 h_1(\alpha_j-\alpha_i)^2
		+\frac14 h_1'(\beta_j-\beta_i)
		+\frac18 h_1(\beta_j-\beta_i)^2\right)\\
		&-\frac18\left(
		\sum_{i=1}^{n}\epsilon_i h_3(\alpha_1-\beta_i)
		-\sum_{i=2}^{n}\epsilon_i h_1(\alpha_1-\alpha_i)
		\right)^2\\
		&+\frac12\left(
		\sum_{i=1}^{n}\epsilon_i h_3'(\alpha_1-\beta_i)
		-\sum_{i=2}^{n}\epsilon_i h_1'(\alpha_1-\alpha_i)
		\right)\\
		&-\frac14\sum_{j=2}^{n} h_1'(\alpha_j-\alpha_1)
		-\frac14\sum_{j=1}^{n} h_3'(\beta_j-\alpha_1)\\
		&+\frac14\sum_{j=2}^{n}h_1(\alpha_j-\alpha_1)
		\left(
		\sum_{i=1}^{n}\epsilon_j\epsilon_i h_3(\alpha_j-\beta_i)
		-\sum_{i\neq j}\epsilon_j\epsilon_i h_1(\alpha_j-\alpha_i)
		\right)\\
		&+\frac14\sum_{j=1}^{n}h_3(\beta_j-\alpha_1)
		\left(
		\sum_{i=1}^{n}\epsilon_j\epsilon_i h_3(\beta_j-\alpha_i)
		-\sum_{i\neq j}\epsilon_j\epsilon_i h_1(\beta_j-\beta_i)
		\right).
	\end{split}
\end{align}

\begin{proof}[Proof of Lemma~\ref{lem::Geps_doubleperiodic}]
	The $2\pi$-periodicity follows immediately from~\eqref{eqn::app_H_identities}. We check the imaginary period. Let $\Delta_x$ denote the change when a variable
	$x$ is replaced by $x+2\ii r$. In the two choices of $\epsilon$, we
	always have $\epsilon_1=1$. Put
	\begin{align*}
		S=&\sum_{i=1}^{n}\epsilon_i h_3(\alpha_1-\beta_i)
		-\sum_{i=2}^{n}\epsilon_i h_1(\alpha_1-\alpha_i),\\
		X_j=&\sum_{i=1}^{n}\epsilon_j\epsilon_i h_3(\alpha_j-\beta_i)
		-\sum_{i\neq j}\epsilon_j\epsilon_i h_1(\alpha_j-\alpha_i),
		\qquad 2\le j\le n,\\
		Y_j=&\sum_{i=1}^{n}\epsilon_j\epsilon_i h_3(\beta_j-\alpha_i)
		-\sum_{i\neq j}\epsilon_j\epsilon_i h_1(\beta_j-\beta_i),
		\qquad 1\le j\le n.
	\end{align*}
	Write $G_{\epsilon}=T_1+\cdots+T_7$, where $T_1,\ldots,T_7$ are the seven
	lines in~\eqref{eqn::app_Gepsilon_def}. We have the following three observations computing the change of each non-constant line.
	
	\medbreak
	1. Let us shift $\alpha_1$ to $\alpha_1+2\ii r$. For $T_1$ only the terms
	with $i=1$ change; for $T_2$ only the terms with the pair $(1,j)$ change.
	Using $h_a(z-2\ii r)=h_a(z)+2\ii$, we get
	\begin{align*}
		\Delta_{\alpha_1}T_1
		=&-\sum_{j=1}^{n}\epsilon_j
		\frac18\left[\left(h_3(\beta_j-\alpha_1)+2\ii\right)^2
		-h_3(\beta_j-\alpha_1)^2\right] = -\frac{\ii}{2}\sum_{j=1}^{n}\epsilon_j h_3(\beta_j-\alpha_1)
		+\frac12\sum_{j=1}^{n}\epsilon_j,\\
		\Delta_{\alpha_1}T_2
		=&\sum_{j=2}^{n}\epsilon_j
		\frac18\left[\left(h_1(\alpha_j-\alpha_1)+2\ii\right)^2
		-h_1(\alpha_j-\alpha_1)^2\right] = \frac{\ii}{2}\sum_{j=2}^{n}\epsilon_j h_1(\alpha_j-\alpha_1)
		-\frac12\sum_{j=2}^{n}\epsilon_j.
	\end{align*}
	The quantity $S$ changes to $S-2\ii$, hence
	\begin{align*}
		\Delta_{\alpha_1}T_3
		=-\frac18\left[(S-2\ii)^2-S^2\right]
		=\frac{\ii}{2}S+\frac12.
	\end{align*}
	The derivative terms $T_4$ and $T_5$ do not change. For $T_6$, the outer
	factor $h_1(\alpha_j-\alpha_1)$ changes to
	$h_1(\alpha_j-\alpha_1)+2\ii$, and $X_j$ changes to $X_j-2\ii\epsilon_j$.
	Thus
	\begin{align*}
		\Delta_{\alpha_1}T_6
		=&\frac14\sum_{j=2}^{n}
		\left[\left(h_1(\alpha_j-\alpha_1)+2\ii\right)
		\left(X_j-2\ii\epsilon_j\right)
		-h_1(\alpha_j-\alpha_1)X_j\right] = \frac{\ii}{2}\sum_{j=2}^{n}X_j
		-\frac{\ii}{2}\sum_{j=2}^{n}\epsilon_jh_1(\alpha_j-\alpha_1)
		+\sum_{j=2}^{n}\epsilon_j.
	\end{align*}
	For $T_7$, the outer factor $h_3(\beta_j-\alpha_1)$ changes to
	$h_3(\beta_j-\alpha_1)+2\ii$, and $Y_j$ changes to $Y_j+2\ii\epsilon_j$.
	Thus
	\begin{align*}
		\Delta_{\alpha_1}T_7
		=&\frac14\sum_{j=1}^{n}
		\left[\left(h_3(\beta_j-\alpha_1)+2\ii\right)
		\left(Y_j+2\ii\epsilon_j\right)
		-h_3(\beta_j-\alpha_1)Y_j\right] = \frac{\ii}{2}\sum_{j=1}^{n}Y_j
		+\frac{\ii}{2}\sum_{j=1}^{n}\epsilon_jh_3(\beta_j-\alpha_1)
		-\sum_{j=1}^{n}\epsilon_j.
	\end{align*}
	Adding the constant terms gives
	\[
	\frac12\sum_{j=1}^{n}\epsilon_j
	-\frac12\sum_{j=2}^{n}\epsilon_j
	+\frac12
	+\sum_{j=2}^{n}\epsilon_j
	-\sum_{j=1}^{n}\epsilon_j
	=0.
	\]
	The remaining terms are
	\begin{align*}
		\Delta_{\alpha_1}G_{\epsilon}
		=&-\frac{\ii}{2}\sum_{i=1}^{n}\epsilon_i h_3(\beta_i-\alpha_1)
		+\frac{\ii}{2}\sum_{i=2}^{n}\epsilon_i h_1(\alpha_i-\alpha_1)
		+\frac{\ii}{2}S\\
		&+\frac{\ii}{2}\sum_{j=2}^{n}X_j
		-\frac{\ii}{2}\sum_{j=2}^{n}\epsilon_jh_1(\alpha_j-\alpha_1)
		+\frac{\ii}{2}\sum_{j=1}^{n}Y_j
		+\frac{\ii}{2}\sum_{j=1}^{n}\epsilon_jh_3(\beta_j-\alpha_1).
	\end{align*}
	Now
	\[
	S=-\sum_{i=1}^{n}\epsilon_i h_3(\beta_i-\alpha_1)
	+\sum_{i=2}^{n}\epsilon_i h_1(\alpha_i-\alpha_1),
	\]
	and, after expanding $X_j,Y_j$ and using oddness,
	\[
	\sum_{j=2}^{n}X_j+\sum_{j=1}^{n}Y_j
	=
	\sum_{i=1}^{n}\epsilon_i h_3(\beta_i-\alpha_1)
	-\sum_{i=2}^{n}\epsilon_i h_1(\alpha_i-\alpha_1).
	\]
	Substituting these two identities into the previous display gives
	\[
	\Delta_{\alpha_1}G_{\epsilon}=0.
	\]
	
	\medbreak
	2. Fix $m\in\{2,\ldots,n\}$ and shift $\alpha_m$ to
	$\alpha_m+2\ii r$. For $T_1$ only the terms with $i=m$ change. For $T_2$ the
	pairs $(i,m)$ with $i<m$ and $(m,j)$ with $j>m$ change. Therefore
	\begin{align*}
		\Delta_{\alpha_m}T_1
		=&-\sum_{i=1}^{n}\epsilon_m\epsilon_i
		\frac18\left[\left(h_3(\beta_i-\alpha_m)+2\ii\right)^2
		-h_3(\beta_i-\alpha_m)^2\right] = -\frac{\ii\epsilon_m}{2}
		\sum_{i=1}^{n}\epsilon_i h_3(\beta_i-\alpha_m)
		+\frac{\epsilon_m}{2}\sum_{i=1}^{n}\epsilon_i,\\
		\Delta_{\alpha_m}T_2
		=&\sum_{i<m}\epsilon_i\epsilon_m
		\frac18\left[\left(h_1(\alpha_m-\alpha_i)-2\ii\right)^2
		-h_1(\alpha_m-\alpha_i)^2\right]+\sum_{j>m}\epsilon_m\epsilon_j
		\frac18\left[\left(h_1(\alpha_j-\alpha_m)+2\ii\right)^2
		-h_1(\alpha_j-\alpha_m)^2\right]\\
		=&\frac{\ii\epsilon_m}{2}
		\sum_{i\neq m}\epsilon_i h_1(\alpha_i-\alpha_m)
		-\frac{\epsilon_m}{2}\sum_{i\neq m}\epsilon_i.
	\end{align*}
	The quantity $S$ changes to $S-2\ii\epsilon_m$, and hence
	\begin{align*}
		\Delta_{\alpha_m}T_3
		=-\frac18\left[(S-2\ii\epsilon_m)^2-S^2\right]
		=\frac{\ii\epsilon_m}{2}S+\frac12.
	\end{align*}
	The derivative terms $T_4$ and $T_5$ do not change. For $T_6$, the summand
	with outer index $m$ has
	\[
	h_1(\alpha_m-\alpha_1)\mapsto h_1(\alpha_m-\alpha_1)-2\ii,
	\qquad X_m\mapsto X_m-2\ii,
	\]
	whereas, for $j\neq m$, only the term in $X_j$ containing
	$h_1(\alpha_j-\alpha_m)$ changes and gives $X_j\mapsto X_j-2\ii\epsilon_j\epsilon_m$.
	Thus
	\begin{align*}
		\Delta_{\alpha_m}T_6
		=&\frac14\left[\left(h_1(\alpha_m-\alpha_1)-2\ii\right)(X_m-2\ii)
		-h_1(\alpha_m-\alpha_1)X_m\right] -\frac{\ii\epsilon_m}{2}
		\sum_{\substack{2\le j\le n\\j\neq m}}
		\epsilon_j h_1(\alpha_j-\alpha_1)\\
		=&-\frac{\ii}{2}X_m
		-\frac{\ii\epsilon_m}{2}
		\sum_{j=2}^{n}\epsilon_j h_1(\alpha_j-\alpha_1)
		-1.
	\end{align*}
	For $T_7$, only the first sum in $Y_j$ changes, and
	$Y_j\mapsto Y_j+2\ii\epsilon_j\epsilon_m$. Hence
	\begin{align*}
		\Delta_{\alpha_m}T_7
		=&\frac{\ii\epsilon_m}{2}
		\sum_{j=1}^{n}\epsilon_j h_3(\beta_j-\alpha_1).
	\end{align*}
	The constants are
	\[
	\frac{\epsilon_m}{2}\sum_{i=1}^{n}\epsilon_i
	-\frac{\epsilon_m}{2}\sum_{i\neq m}\epsilon_i
	+\frac12-1
	=0.
	\]
	The non-constant part is
	\begin{align*}
		\Delta_{\alpha_m}G_{\epsilon}
		=&-\frac{\ii\epsilon_m}{2}
		\sum_{i=1}^{n}\epsilon_i h_3(\beta_i-\alpha_m)
		+\frac{\ii\epsilon_m}{2}
		\sum_{i\neq m}\epsilon_i h_1(\alpha_i-\alpha_m)
		+\frac{\ii\epsilon_m}{2}S\\
		&-\frac{\ii}{2}X_m
		-\frac{\ii\epsilon_m}{2}
		\sum_{i=2}^{n}\epsilon_i h_1(\alpha_i-\alpha_1)
		+\frac{\ii\epsilon_m}{2}
		\sum_{i=1}^{n}\epsilon_i h_3(\beta_i-\alpha_1).
	\end{align*}
	By the definition of $X_m$ and oddness,
	\[
	X_m
	=
	\epsilon_m\left(
	-\sum_{i=1}^{n}\epsilon_i h_3(\beta_i-\alpha_m)
	+\sum_{i\neq m}\epsilon_i h_1(\alpha_i-\alpha_m)
	\right).
	\]
	This cancels the first two sums in $\Delta_{\alpha_m}G_{\epsilon}$. The
	remaining terms vanish because
	\[
	S=
	-\sum_{i=1}^{n}\epsilon_i h_3(\beta_i-\alpha_1)
	+\sum_{i=2}^{n}\epsilon_i h_1(\alpha_i-\alpha_1).
	\]
	Thus
	\[
	\Delta_{\alpha_m}G_{\epsilon}=0,\qquad 2\le m\le n.
	\]
	
	\medbreak
	3. Fix $m\in\{1,\ldots,n\}$ and shift $\beta_m$ to $\beta_m+2\ii r$. For
	$T_1$ only the terms with $j=m$ change. For the $\beta$-part of $T_2$, the
	pairs $(i,m)$ with $i<m$ and $(m,j)$ with $j>m$ change. Thus
	\begin{align*}
		\Delta_{\beta_m}T_1
		=&-\sum_{i=1}^{n}\epsilon_i\epsilon_m
		\frac18\left[\left(h_3(\beta_m-\alpha_i)-2\ii\right)^2
		-h_3(\beta_m-\alpha_i)^2\right] = \frac{\ii\epsilon_m}{2}
		\sum_{i=1}^{n}\epsilon_i h_3(\beta_m-\alpha_i)
		+\frac{\epsilon_m}{2}\sum_{i=1}^{n}\epsilon_i,\\
		\Delta_{\beta_m}T_2
		=&\sum_{i<m}\epsilon_i\epsilon_m
		\frac18\left[\left(h_1(\beta_m-\beta_i)-2\ii\right)^2
		-h_1(\beta_m-\beta_i)^2\right] +\sum_{j>m}\epsilon_m\epsilon_j
		\frac18\left[\left(h_1(\beta_j-\beta_m)+2\ii\right)^2
		-h_1(\beta_j-\beta_m)^2\right]\\
		=&\frac{\ii\epsilon_m}{2}
		\sum_{i\neq m}\epsilon_i h_1(\beta_i-\beta_m)
		-\frac{\epsilon_m}{2}\sum_{i\neq m}\epsilon_i.
	\end{align*}
	The quantity $S$ changes to $S+2\ii\epsilon_m$, so
	\begin{align*}
		\Delta_{\beta_m}T_3
		=-\frac18\left[(S+2\ii\epsilon_m)^2-S^2\right]
		=-\frac{\ii\epsilon_m}{2}S+\frac12.
	\end{align*}
	The derivative terms $T_4$ and $T_5$ do not change. For $T_6$, only the first
	sum in $X_j$ changes and $X_j\mapsto X_j+2\ii\epsilon_j\epsilon_m$. Therefore
	\begin{align*}
		\Delta_{\beta_m}T_6
		=&\frac{\ii\epsilon_m}{2}
		\sum_{j=2}^{n}\epsilon_j h_1(\alpha_j-\alpha_1).
	\end{align*}
	For $T_7$, the summand with outer index $m$ has
	\[
	h_3(\beta_m-\alpha_1)\mapsto h_3(\beta_m-\alpha_1)-2\ii,
	\qquad Y_m\mapsto Y_m-2\ii,
	\]
	whereas, for $j\neq m$, the term in $Y_j$ containing
	$h_1(\beta_j-\beta_m)$ gives $Y_j\mapsto Y_j-2\ii\epsilon_j\epsilon_m$.
	Consequently,
	\begin{align*}
		\Delta_{\beta_m}T_7
		=&\frac14\left[\left(h_3(\beta_m-\alpha_1)-2\ii\right)(Y_m-2\ii)
		-h_3(\beta_m-\alpha_1)Y_m\right] -\frac{\ii\epsilon_m}{2}
		\sum_{\substack{1\le j\le n\\j\neq m}}
		\epsilon_j h_3(\beta_j-\alpha_1)\\
		=&-\frac{\ii}{2}Y_m
		-\frac{\ii\epsilon_m}{2}
		\sum_{j=1}^{n}\epsilon_j h_3(\beta_j-\alpha_1)
		-1.
	\end{align*}
	The constants add up to
	\[
	\frac{\epsilon_m}{2}\sum_{i=1}^{n}\epsilon_i
	-\frac{\epsilon_m}{2}\sum_{i\neq m}\epsilon_i
	+\frac12-1
	=0.
	\]
	The non-constant part is
	\begin{align*}
		\Delta_{\beta_m}G_{\epsilon}
		=&\frac{\ii\epsilon_m}{2}
		\sum_{i=1}^{n}\epsilon_i h_3(\beta_m-\alpha_i)
		+\frac{\ii\epsilon_m}{2}
		\sum_{i\neq m}\epsilon_i h_1(\beta_i-\beta_m)
		-\frac{\ii\epsilon_m}{2}S\\
		&+\frac{\ii\epsilon_m}{2}
		\sum_{j=2}^{n}\epsilon_j h_1(\alpha_j-\alpha_1)
		-\frac{\ii}{2}Y_m
		-\frac{\ii\epsilon_m}{2}
		\sum_{j=1}^{n}\epsilon_j h_3(\beta_j-\alpha_1).
	\end{align*}
	By the definition of $Y_m$ and oddness,
	\[
	Y_m
	=
	\epsilon_m\left(
	\sum_{i=1}^{n}\epsilon_i h_3(\beta_m-\alpha_i)
	+\sum_{i\neq m}\epsilon_i h_1(\beta_i-\beta_m)
	\right),
	\]
	which cancels the first two sums. The remaining terms cancel after substituting
	\[
	S=
	-\sum_{j=1}^{n}\epsilon_j h_3(\beta_j-\alpha_1)
	+\sum_{j=2}^{n}\epsilon_j h_1(\alpha_j-\alpha_1).
	\]
	Therefore
	\[
	\Delta_{\beta_m}G_{\epsilon}=0,\qquad 1\le m\le n.
	\]
	\medbreak
	Combining the three observations above, we conclude the $2\ii r$-periodicity for all $\alpha$- and $\beta$-coordinates. Therefore $G_{\epsilon}$ is doubly periodic in each variable.
\end{proof}

\begin{proof}[Proof of Lemma~\ref{lem::Geps_nopoles}]
	By the periodicity shown in Lemma~\ref{lem::Geps_doubleperiodic}, it suffices to inspect the
	possible singularities
	\[
	\alpha_i\to\alpha_j,\qquad \beta_i\to\beta_j,\qquad
	\alpha_i\to\beta_j+\ii r.
	\]
	We inspect the singular parts in each case. All unlisted terms are analytic in the corresponding local coordinate.
	
	\medbreak
	1. Let $m\in\{2,\ldots,n\}$ and put $z=\alpha_1-\alpha_m\to0$. Define the
	regular part
	\[
	A_m:=
	\sum_{i=1}^{n}\epsilon_i h_3(\alpha_m-\beta_i)
	-\sum_{\substack{i=2\\i\neq m}}^{n}
	\epsilon_i h_1(\alpha_m-\alpha_i).
	\]
	Using~\eqref{eqn::app_H_laurent},
	\[
	S=-\epsilon_m h_1(z)+A_m+O(z)
	=-\frac{2\epsilon_m}{z}+A_m+O(z),
	\]
	and
	\[
	X_m=-\epsilon_m h_1(\alpha_m-\alpha_1)+\epsilon_m A_m+O(z)
	=\frac{2\epsilon_m}{z}+\epsilon_m A_m+O(z).
	\]
	The terms in $G_\epsilon$ with possible poles are
	\begin{align*}
		T_2^{\mathrm{sing}}
		=&\epsilon_m\left(
		\frac14 h_1'(\alpha_m-\alpha_1)
		+\frac18h_1(\alpha_m-\alpha_1)^2\right)=O(1),\\
		T_3^{\mathrm{sing}}
		=&-\frac18\left(-\frac{2\epsilon_m}{z}+A_m\right)^2
		=-\frac{1}{2z^2}+\frac{\epsilon_m A_m}{2z}+O(1),\\
		T_4^{\mathrm{sing}}
		=&-\frac{\epsilon_m}{2}h_1'(z)
		=\frac{\epsilon_m}{z^2}+O(1),\\
		T_5^{\mathrm{sing}}
		=&-\frac14 h_1'(-z)
		=\frac{1}{2z^2}+O(1),\\
		T_6^{\mathrm{sing}}
		=&\frac14 h_1(-z)X_m
		=\frac14\left(-\frac2z+O(z)\right)
		\left(\frac{2\epsilon_m}{z}+\epsilon_m A_m+O(z)\right) = -\frac{\epsilon_m}{z^2}
		-\frac{\epsilon_m A_m}{2z}+O(1).
	\end{align*}
	The coefficient of $z^{-2}$ is
	\[
	-\frac12+\epsilon_m+\frac12-\epsilon_m=0,
	\]
	and the coefficient of $z^{-1}$ is
	\[
	\frac{\epsilon_m A_m}{2}-\frac{\epsilon_m A_m}{2}=0.
	\]
	Thus $G_\epsilon$ has no pole when $\alpha_1\to\alpha_m$.
	
	\medbreak
	2. Let $2\le p<q\le n$ and put $z=\alpha_p-\alpha_q\to0$. The direct
	interaction in $T_2$ is
	\[
	T_2^{\mathrm{sing}}
	=
	\epsilon_p\epsilon_q
	\left(
	\frac14 h_1'(\alpha_q-\alpha_p)
	+\frac18h_1(\alpha_q-\alpha_p)^2\right)
	=O(1).
	\]
	The only other possible poles come from the $p$-th and $q$-th summands in
	$T_6$. In $X_p$ and $X_q$ we have
	\[
	X_p=-\epsilon_p\epsilon_q h_1(z)+O(1)
	=-\frac{2\epsilon_p\epsilon_q}{z}+O(1),
	\]
	and
	\[
	X_q=-\epsilon_q\epsilon_p h_1(-z)+O(1)
	=\frac{2\epsilon_p\epsilon_q}{z}+O(1).
	\]
	Therefore
	\begin{align*}
		T_6^{\mathrm{sing}}
		=&\frac14 h_1(\alpha_p-\alpha_1)X_p
		+\frac14 h_1(\alpha_q-\alpha_1)X_q = \frac{\epsilon_p\epsilon_q}{2z}
		\left(h_1(\alpha_q-\alpha_1)-h_1(\alpha_p-\alpha_1)\right)+O(1).
	\end{align*}
	Since
	\[
	h_1(\alpha_p-\alpha_1)
	=h_1(\alpha_q-\alpha_1)+z\,h_1'(\alpha_q-\alpha_1)+O(z^2),
	\]
	the last display is $O(1)$. Thus $G_\epsilon$ has no pole when
	$\alpha_p\to\alpha_q$ for $2\le p<q\le n$.
	
	\medbreak
	3. Let $m\in\{1,\ldots,n\}$ and put
	$z=\alpha_1-\beta_m-\ii r\to0$. Define
	\[
	A_m^{\alpha\beta}
	:=
	-\ii\epsilon_m
	+\sum_{i\neq m}\epsilon_i h_3(\beta_m+\ii r-\beta_i)
	-\sum_{i=2}^{n}\epsilon_i h_1(\beta_m+\ii r-\alpha_i),
	\]
	so that
	\[
	S=\frac{2\epsilon_m}{z}+A_m^{\alpha\beta}+O(z).
	\]
	For the $m$-th summand of $Y_m$, define
	\[
	C_m^{\alpha\beta}
	:=
	\ii\epsilon_m
	+\sum_{i=2}^{n}\epsilon_m\epsilon_i h_3(\beta_m-\alpha_i)
	-\sum_{i\neq m}\epsilon_m\epsilon_i h_1(\beta_m-\beta_i),
	\]
	so that
	\[
	Y_m=-\frac{2\epsilon_m}{z}+C_m^{\alpha\beta}+O(z).
	\]
	Using~\eqref{eqn::app_H3_laurent} and oddness,
	\[
	h_3(\alpha_1-\beta_m)=\frac2z-\ii+O(z),
	\qquad
	h_3(\beta_m-\alpha_1)=-\frac2z+\ii+O(z).
	\]
	The possible singular terms are
	\begin{align*}
		T_1^{\mathrm{sing}}
		=&-\epsilon_m\left(
		\frac14 h_3'(\beta_m-\alpha_1)
		+\frac18h_3(\beta_m-\alpha_1)^2\right)
		=\frac{\ii\epsilon_m}{2z}+O(1),\\
		T_3^{\mathrm{sing}}
		=&-\frac18\left(\frac{2\epsilon_m}{z}+A_m^{\alpha\beta}\right)^2
		=-\frac{1}{2z^2}
		-\frac{\epsilon_m A_m^{\alpha\beta}}{2z}+O(1),\\
		T_4^{\mathrm{sing}}
		=&\frac{\epsilon_m}{2}h_3'(\alpha_1-\beta_m)
		=-\frac{\epsilon_m}{z^2}+O(1),\\
		T_5^{\mathrm{sing}}
		=&-\frac14 h_3'(\beta_m-\alpha_1)
		=\frac{1}{2z^2}+O(1),\\
		T_7^{\mathrm{sing}}
		=&\frac14 h_3(\beta_m-\alpha_1)Y_m = \frac14\left(-\frac2z+\ii+O(z)\right)
		\left(-\frac{2\epsilon_m}{z}+C_m^{\alpha\beta}+O(z)\right) = \frac{\epsilon_m}{z^2}
		-\frac{C_m^{\alpha\beta}}{2z}
		-\frac{\ii\epsilon_m}{2z}+O(1).
	\end{align*}
	The coefficient of $z^{-2}$ is
	\[
	-\frac12-\epsilon_m+\frac12+\epsilon_m=0.
	\]
	The coefficient of $z^{-1}$ is
	\[
	\frac{\ii\epsilon_m}{2}
	-\frac{\epsilon_m A_m^{\alpha\beta}}{2}
	-\frac{C_m^{\alpha\beta}}{2}
	-\frac{\ii\epsilon_m}{2}
	=
	-\frac12\left(\epsilon_m A_m^{\alpha\beta}+C_m^{\alpha\beta}\right).
	\]
	By~\eqref{eqn::app_H_identities},
	\[
	h_3(w+\ii r)=h_1(w)-\ii,\qquad
	h_1(w+\ii r)=h_3(w)-\ii.
	\]
	Applying these identities to $A_m^{\alpha\beta}$ gives
	\[
	A_m^{\alpha\beta}
	=
	\sum_{i\neq m}\epsilon_i h_1(\beta_m-\beta_i)
	-\sum_{i=2}^{n}\epsilon_i h_3(\beta_m-\alpha_i)
	-\ii.
	\]
	Consequently,
	\[
	\epsilon_m A_m^{\alpha\beta}+C_m^{\alpha\beta}=0.
	\]
	Thus $G_\epsilon$ has no pole when $\alpha_1\to\beta_m+\ii r$.
	
	\medbreak
	4. Let $p\in\{2,\ldots,n\}$, $m\in\{1,\ldots,n\}$ with $p\neq m$, and put
	$z=\alpha_p-\beta_m-\ii r\to0$. The possible pole from $T_1$ is
	\begin{align*}
		T_1^{\mathrm{sing}}
		=&-\epsilon_p\epsilon_m
		\left(
		\frac14 h_3'(\beta_m-\alpha_p)
		+\frac18h_3(\beta_m-\alpha_p)^2\right)
		=\frac{\ii\epsilon_p\epsilon_m}{2z}+O(1).
	\end{align*}
	The $p$-th summand in $T_6$ contains
	\[
	X_p=\epsilon_p\epsilon_m h_3(\alpha_p-\beta_m)+O(1)
	=\epsilon_p\epsilon_m\left(\frac2z-\ii\right)+O(z),
	\]
	and hence
	\[
	T_6^{\mathrm{sing}}
	=
	\frac14 h_1(\alpha_p-\alpha_1)X_p
	=
	\frac{\epsilon_p\epsilon_m}{2z}
	h_1(\alpha_p-\alpha_1)+O(1).
	\]
	The $m$-th summand in $T_7$ contains
	\[
	Y_m=\epsilon_m\epsilon_p h_3(\beta_m-\alpha_p)+O(1)
	=\epsilon_p\epsilon_m\left(-\frac2z+\ii\right)+O(z),
	\]
	and hence
	\[
	T_7^{\mathrm{sing}}
	=
	\frac14 h_3(\beta_m-\alpha_1)Y_m
	=
	-\frac{\epsilon_p\epsilon_m}{2z}
	h_3(\beta_m-\alpha_1)+O(1).
	\]
	Therefore the residue is
	\[
	\frac{\epsilon_p\epsilon_m}{2}
	\left(\ii+h_1(\alpha_p-\alpha_1)-h_3(\beta_m-\alpha_1)\right).
	\]
	At the limit $\alpha_p=\beta_m+\ii r$, we have
	\[
	h_1(\alpha_p-\alpha_1)
	=h_1(\beta_m-\alpha_1+\ii r)
	=h_3(\beta_m-\alpha_1)-\ii,
	\]
	so the residue is zero. Thus $G_\epsilon$ has no pole when
	$\alpha_p\to\beta_m+\ii r$ with $p\neq m$.
	
	\medbreak
	For the diagonal collision $\alpha_p\to\beta_p+\ii r$ with $p\ge2$, put
	$z=\alpha_p-\beta_p-\ii r$. The singular terms are
	\[
	T_1^{\mathrm{sing}}=\frac{\ii}{2z}+O(1),\qquad
	T_6^{\mathrm{sing}}=\frac{h_1(\alpha_p-\alpha_1)}{2z}+O(1),\qquad
	T_7^{\mathrm{sing}}=-\frac{h_3(\beta_p-\alpha_1)}{2z}+O(1).
	\]
	Their residue is
	\[
	\frac12\left(\ii+h_1(\alpha_p-\alpha_1)-h_3(\beta_p-\alpha_1)\right),
	\]
	which is zero by the identity displayed in the previous paragraph. Hence this
	diagonal collision has no pole.
	
	\medbreak
	5. Let $1\le p<q\le n$ and put $z=\beta_p-\beta_q\to0$. The direct
	$\beta$-interaction in $T_2$ is
	\[
	T_2^{\mathrm{sing}}
	=
	\epsilon_p\epsilon_q
	\left(
	\frac14 h_1'(\beta_q-\beta_p)
	+\frac18 h_1(\beta_q-\beta_p)^2
	\right)
	=O(1).
	\]
	The possible pole in $T_7$ comes from the $p$-th and $q$-th summands. In
	$Y_p$ and $Y_q$,
	\[
	Y_p=-\epsilon_p\epsilon_qh_1(z)+O(1)
	=-\frac{2\epsilon_p\epsilon_q}{z}+O(1),
	\]
	and
	\[
	Y_q=-\epsilon_q\epsilon_ph_1(-z)+O(1)
	=\frac{2\epsilon_p\epsilon_q}{z}+O(1).
	\]
	Therefore
	\begin{align*}
		T_7^{\mathrm{sing}}
		=&\frac14 h_3(\beta_p-\alpha_1)Y_p
		+\frac14 h_3(\beta_q-\alpha_1)Y_q = \frac{\epsilon_p\epsilon_q}{2z}
		\left(h_3(\beta_q-\alpha_1)-h_3(\beta_p-\alpha_1)\right)+O(1).
	\end{align*}
	Since
	\[
	h_3(\beta_p-\alpha_1)
	=h_3(\beta_q-\alpha_1)+z\,h_3'(\beta_q-\alpha_1)+O(z^2),
	\]
	the last display is $O(1)$. Thus $G_\epsilon$ has no pole when
	$\beta_p\to\beta_q$.
	
	\medbreak
	Combining the observations above, we conclude that all possible singularities are removable. Thus $G_{\epsilon}$ has no poles in $\C^{2n}$.
\end{proof}


\end{document}